\documentclass[preprint,3p]{amsart}
\usepackage{wrapfig}
\usepackage{amssymb}

\usepackage{float}

\usepackage{amsmath,amsfonts,amsthm,amstext}
\usepackage[latin1]{inputenc}
\usepackage{fancybox}
\usepackage{niceframe}
\usepackage{array}
\usepackage{newlfont}
\usepackage{verbatim}
\usepackage[dvips]{psfrag}
\usepackage{color}
\usepackage{float}
\usepackage[scriptsize]{caption}
\usepackage{indentfirst}
\usepackage[normalem]{ulem}
\usepackage{pst-text}
\usepackage{graphicx}
\usepackage{graphics}
\usepackage{overpic}

\graphicspath{{./Figuras/}}
\usepackage{wrapfig}
\newcommand{\R}{\ensuremath{\mathbb{R}}}

\newcommand{\Q}{\ensuremath{\mathbb{Q}}}

\newcommand{\s}{\Sigma}
\newcommand{\dx}{\dfrac{\partial}{\partial x}}
\newcommand{\dy}{\dfrac{\partial}{\partial y}}

\newcommand{\la}{\lambda}
\newcommand{\al}{\alpha}
\newcommand{\be}{\beta}

\newcommand{\V}{\mathcal{V}}
\newtheorem {prop} {Proposition}
\newtheorem {corollary} {Corollary}
\newtheorem {lemma} {Lemma}
\newtheorem {definition} {Definition}
\newtheorem {remark} {Remark}

\newtheorem*{thmA}{Theorem A}
\newtheorem*{thmB}{Theorem B}
\newtheorem*{thmC}{Theorem C}
\newtheorem*{thmD}{Theorem D}
\newtheorem*{thmE}{Theorem E}
\newtheorem*{thmF}{Theorem F}
\newtheorem*{thmG}{Theorem G}
\newtheorem*{thmH}{Theorem H}
\definecolor{verde}{rgb}{0.0,0.5,0.0}
\definecolor{azul}{rgb}{0,0,128}
\definecolor{roxo}{rgb}{0.44,0.16,0.39}
\definecolor{vinho}{rgb}{0.5,0.0,0.13}
\definecolor{lilas1}{rgb}{0.6,0.33,0.73}
\definecolor{rosa}{rgb}{0.84,0.04,0.33}
\definecolor{mostarda}{rgb}{0.91,0.41,0.17}
\definecolor{mostarda2}{rgb}{1.0,0.66,0.07}

\begin{document}

\title[Homoclinic boundary-saddle bifurcations]{Homoclinic boundary-saddle bifurcations in nonsmooth vector fields}
%{Bifurcation Diagrams for Nonsmooth Vector Fields Presenting a Degenerate Cycle Through Nonresonant Saddles}

%\address[UFG]{Department of Mathematics, UFG, IME\\ Goi\^ania-GO, 74690-900, Brazil}
%\address[UoB]{Engineering Mathematics, University of Bristol\\ Merchant Venturer's Building, Bristol BS8 1UB, UK}
%\address[Unicamp]{Department of Mathematics, Unicamp, IMECC\\ Campinas-SP, 13083-970, Brazil}
%

\author{Kamila da S. Andrade}
\address[KSA]{Department of Mathematics, UFG, IME\\ Goi\^ania-GO, 74690-900, Brazil}
\email{kamila.andrade@ufg.com}

\author{Mike R. Jeffrey}
\address[MRJ]{Engineering Mathematics, University of Bristol\\ Merchant Venturer's Building, Bristol BS8 1UB, UK}
\email{mike.jeffrey@bristol.ac.uk}

\author{Ricardo M. Martins}
\address[RMM]{Department of Mathematics, Unicamp, IMECC\\ Campinas-SP, 13083-970, Brazil}
\email{rmiranda@ime.unicamp.br}

\author{Marco A. Teixeira}
\address[MAT]{Department of Mathematics, Unicamp, IMECC\\ Campinas-SP, 13083-970, Brazil}
\email{teixeira@ime.unicamp.br}

\maketitle

\begin{abstract}
In a smooth dynamical system, a homoclinic connection is a closed orbit returning to a saddle equilibrium. Under perturbation, homoclinics are associated with bifurcations of periodic orbits, and with chaos in higher dimensions. 
Homoclinic connections in nonsmooth systems are complicated by their interaction with discontinuities in their vector fields. A connection may involve a regular saddle outside a discontinuity set, or a pseudo-saddle on a discontinuity set, with segments of the connection allowed to cross or slide along the discontinuity. Even the simplest case, that of connection to a regular saddle that hits a discontinuity as a parameter is varied, is surprisingly complex. Bifurcation diagrams are presented here for non-resonant saddles in the plane, including an example in a forced pendulum. 
\end{abstract}

\section{Introduction}

In smooth dynamical systems, a homoclinic orbit is a closed trajectory connecting a saddle equilibrium to itself. Under perturbation the homoclinic orbit can create a limit cycle or, in more than two dimensions, chaos. Homoclinic orbits in nonsmooth systems come in multiple different forms, only the simplest of which have so far been studied. For example, regular saddles with homoclinic orbits that involve segments of sliding along a line of discontinuity, or homoclinics to so-called pseudo-saddles in the sliding dynamics itself, are studied as one parameter bifurcations in \cite{KRG}. 

A boundary homoclinic orbit, which involves a regular saddle lying on the discontinuity set of a nonsmooth system, is novel so far as classification, because it involves both a local bifurcation (a saddle lying on the discontinuity set, or a so-called {\it boundary equilibrium} bifurcation \cite{bernardo2008piecewise}), and a global bifurcation in the form of the homoclinic connection. Its unfolding then involves the transition of the saddle into a pseudo-saddle, and the appearance of not one limit cycle, but multiple. The precise unfolding is surprisingly complex even in two dimensions. The bifurcation diagrams for non-resonant saddles in the plane are derived here, including an example in a pendulum with a discontinuous forcing.

Our interest is in systems of the form 
\begin{equation}\label{disc-def}
	\dot{x}=Z(x) %F(x)+\mbox{sgn}(h(x))G(x),
	=\left\{\begin{array}{ll}X(x)\quad& h(x)\geq0, \\ Y(x)\quad & h(x)\leq0, \end{array}\right.
\end{equation}
where $h:\R^2\rightarrow\R$ is a smooth function having $0$ as a regular value, and $X,Y\in\chi^r$, where $\chi^r$ is the set of all $C^r$ vector fields defined in $\R^2$, and $\chi^r$ is endowed with the $C^r$ topology. 

The righthand side $Z$ is a nonsmooth vector field, which for brevity we may write as $Z=(X,Y)$. Denote by $\Omega^r=\chi^r\times\chi^r$ the set of all nonsmooth vector fields $Z$ endowed with the product topology. The set $\s=\{ p\in\R^2;\,h(p)=0\}$ is called the switching surface and the definition of trajectories follows Filippov's convention, see \cite{F}.

Our objective is to study bifurcations of a degenerate cycle passing through a saddle point of $X$ lying on $\s$, called a hyperbolic saddle-regular point of $Z$. We are thus concerned with vector fields $Z=(X,Y)\in\Omega^r$ where $X$ has a hyperbolic saddle point $S_{X}\in\s$ which is a regular point for $Y$ (meaning $Y$ is non-zero and transverse to $\s$ at $S_X$). The stable and unstable manifolds of $S_X$ are transverse to $\s$ in $S_X$ and the unstable manifold intersects $\s$ transversely at a point $P_X\in\s\setminus\{S_X\}$. The trajectory of $Y$ passing through $P_X$ intersects $\s$ transversely at $S_X$ and $P_X$. An example of this kind of cycle is illustrated in figure \ref{degcycle}.
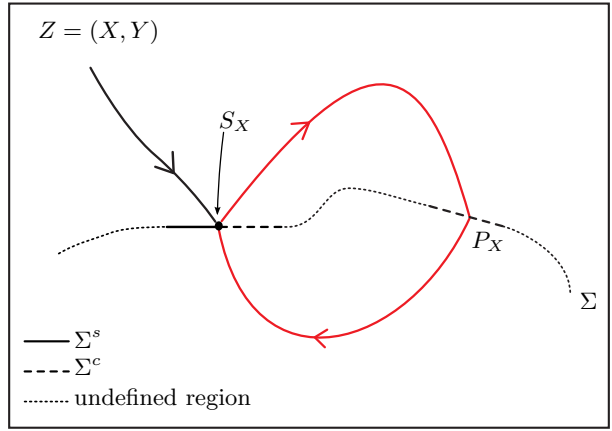
\begin{figure}[H]
	\centering
	\begin{overpic}[width=8cm]{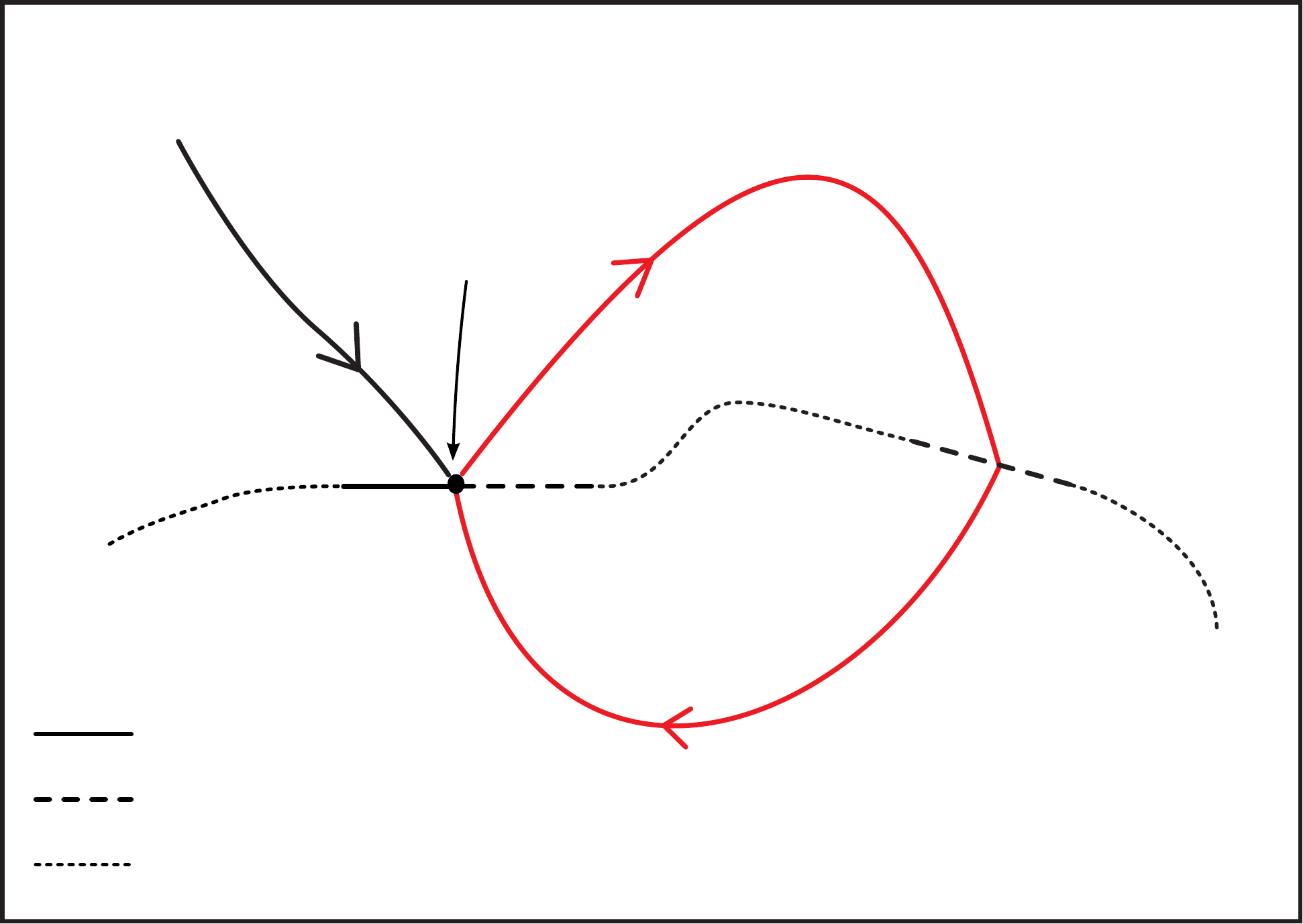}
		\put(5,65){\small$Z=(X,Y)$}\put(95,20){\small $\Sigma$}\put(77,30){\small$P_X$}\put(35,50){\small $S_X$}
		\put(11,13.5){\small$\s^s$}\put(11,9){\small$\s^c$}\put(11,4){\small undefined region}
	\end{overpic}
	\caption{A degenerate cycle passing through a saddle-regular point. $\s^s$, $\s^c$ are, respectively, the sliding and crossing regions.}\label{degcycle}
\end{figure}

%This paper is organized as follows. 
An overview of concepts and definitions are given in Section \ref{preliminaries}. The setting of the problem and the study of the first return map for the degenerate cycle are given in Section \ref{sectionproblem}. In Section \ref{mainsec} the main results and the corresponding bifurcation diagrams are presented. Section \ref{sec-resonant} presents models having a degenerate cycle through resonant saddle-regular point. In Section \ref{sectionapplication} a model presenting a degenerate cycle is given.

% % % % % % % % % % % % % % % % % % % % % % % % % % % % % % % % % % % % % % % % % % % % % % % % % % % % % % % % % % % % % % % % % % % % % % % % % % % % % % % % % % % % % % % % % % % % % % % % % % % % % % % % % % % % % % % % % % % % % % % % % % % % % % % % % % % % % % % % % % % % % % % % % % % % % % % % % % % % % % % % % % %
\section{Preliminaries}\label{preliminaries}

%Taking a planar nonsmooth vector field of the form given in \eqref{disc-def}, t
The switching manifold $\s=\{p\in\R^2;\, h(p)=0 \}$ is the hypersurface boundary between the regions $\s^+=\{p\in\R^2;\, h(p)>0 \}$ and $\s^-=\{p\in\R^2;\, h(p)<0 \}$. It is partitioned into the following regions depending on the directions of $X$ and $Y$:
\begin{itemize}
	\item[-] the crossing region, where $\s^c=\{p\in\s;\, Xh\cdot Yh(p)>0 \}$;
	\item[-] the sliding region, where $\s^s=\{p\in\s;\, Xh(p)<0 \mbox{ and } Yh(p)>0 \}$;
	\item[-] the escaping region, where $\s^e=\{p\in\s;\, Xh(p)>0 \mbox{ and } Yh(p)>0 \}$. 
\end{itemize}
For all $X\in\chi^r$, the scalar $Xh(p)=\langle X,\nabla h\rangle(p)$ is the Lie derivative of $h$ with respect to $X$ at $p$. The regions $\s^c,\s^s,\s^e,$ are open in $\s$ and their complement in $\s$ is the set of all points satisfying $Xh(p)\cdot Yh(p)=0$, called tangency points. A smooth vector field $X$ is transversal to $\s$ at $p\in\s$ if $Xf(p)\neq0$. 

Taking $Z=(X,Y)\in\Omega^r$, if $p\in\s^+$ (resp. $p\in\s^-$) then the trajectory of $Z$ through $p$ is the local trajectory of $X$ (resp. $Y$) through this point. If $p\in\s^c$ the trajectory of $Z$ through $p$ is the concatenation of the respective trajectories of $X$ in $\s^-$ and of $Y$ in $\s^-$. If $p\in\s^s\cup\s^e$ then the trajectory of $Z$ through this point is the trajectory of the sliding vector field $Z^s$ through $p$. The vector field $Z^s$ is defined as the unique convex combination of $X$ and $Y$ that is tangent to $\s$, given by
\begin{equation}\label{sliding_vector_field}
Z^s(p)=\dfrac{1}{Yf(p)-Xf(p)}\left(Yf(p)X(p)-Xf(p)Y(p) \right),
\end{equation}
for $p\in\s^s\cup\s^e$. 
It is useful to define also the normalized sliding vector field $Z^s_N(p)=Yf(p)X(p)-Xf(p)Y(p)$.

The singular points of $Z^s$ in $\s^s\cup\s^e$ are called \textit{pseudo-singular} points. Those singular points of $X$ (resp. $Y$) that lie on $\s^+$ (resp. $\s^-$) are called \textit{real singular} points, and those singular points of $X$ (resp. $Y$) that lie on $\s^-$ (resp. $\s^+$) are called \textit{virtual singular} points. The singularities of $Z=(X,Y)$ are: real singular points, pseudo-equilibria and tangency points. The points that are not singularities are called regular points.

\begin{definition}
	A smooth vector field $X$ has a fold singularity or a quadratic tangency at $p\in\s$ if $Xh(p)=0 $ and $X^2h(p)\neq0$.
	A nonsmooth vector field $Z=(X,Y)$ has a fold singularity at $p\in\s$ if $p$ is a fold for $X$ or $Y$. A fold point $p$ for $X$ (resp. $Y$) is \textit{visible} if $X^2h(p)>0$ (resp. $Y^2h(p)<0$) and \textit{invisible} if $X^2h(p)<0$ (resp. $Y^2h(p)>0$). If $p$ is a fold for $X$ and a regular point for $Y$ (or vice-versa) then $p$ is a \textit{fold-regular point} for $Z$. %Moreover, if $p$ is a fold for both $X$ and $Y$ then $p$ is a \textit{fold-fold point} for $Z$.
\end{definition}

\begin{definition}
	A nonsmooth vector field $Z$ has a \textit{saddle-regular point} at $p\in\s$ if $p$ is a saddle point for $X$ (resp. $Y$), and $Y$ (resp. $X$) is transversal to $\s$ at $p$.
\end{definition}

Denote a trajectory of $\dot x=Z=(X,Y)$ by $\varphi_Z(t,q)$ where $\varphi_Z(0,q)=q$. Taking $q\in\s^+\cup\s^-$, a point $p\in\s$, $p$ is said to be a departing point (resp. arriving point) of $\varphi_Z(t,q)$ if there exists $t_0<0$ (resp. $t_0>0$) such that $\lim_{t\rightarrow t_0^+}\varphi_X(t,q)=p$ (resp. $\lim_{t\rightarrow t_0^-}\varphi_Z(t,q)=p$). With these definitions if $p\in\s^c$, then it is a departing point (resp. arriving point) of $\varphi_X(t,q)$ for any $q\in\gamma^+(p)$ (resp. $q\in\gamma^-(p)$), where
\begin{equation*}
\begin{array}{ccc}
\gamma^+(p)=\{\varphi_Z(t,p);t\in I\cap\{t\geq0\}\} & \mbox{and} & \gamma^-(p)=\{\varphi_Z(t,p);t\in I\cap\{t\leq0\}\}.
\end{array}
\end{equation*}

To distinguish between the main types of orbits and cycles we have the following. 

\begin{definition}
A continuous closed curve $\Gamma$ is said to be a cycle of the vector field $Z$ if it is composed by a finite union of segments of regular orbits and singularities, $\gamma_1,\gamma_2,\ldots,\gamma_n$, of $Z$. There are different types of cycle $\Gamma$:
	\begin{itemize}
		\item[-] $\Gamma$ is a \textit{simple cycle} if none of the $\gamma_i$'s are singular points and the set $\gamma_{i} \cap \Sigma$ is either empty or composed only by points of $\s^c$, $\forall\, i=1,\ldots,n$. If such a cycle is isolated in the set of all simple cycles of $Z$ then it is called a \textit{limit cycle}. See figure \ref{cycles}$(a)$;
		\item[-] $\Gamma$ is a \textit{regular polycycle} if either, for all $i=1,\ldots,n$, the set $\gamma_{i} \cap \Sigma$ is empty and at least one of $\gamma_i'$s is a singular point or, for some $i=1,\dots,n$, $\gamma_{i} \cap \Sigma$ is nonempty but only contains points of $\overline{\Sigma^c}$ that are not tangent points of $Z$. See figures \ref{cycles}$(b)$ and \ref{cycles}$(c)$;
		\item[-] $\Gamma$ is a \textit{sliding cycle} if there exists $i\in\{1,2,\ldots,n\}$ such that $\gamma_{i}$ is a segment of sliding orbit and, for any two consecutive curves, the departing or arriving points in $\s$ are not the same. See figure \ref{cycles}$(d)$;
		\item[-] $\Gamma$ is a \textit{pseudo-cycle} if for some $i\in\{1,2,\ldots,n\}$, the arriving points (or departing points) of $\gamma_i$ and $\gamma_{i+1}$ are the same. See figures \ref{cycles}$(e)$ and \ref{cycles}$(f)$.
	\end{itemize}
\end{definition}
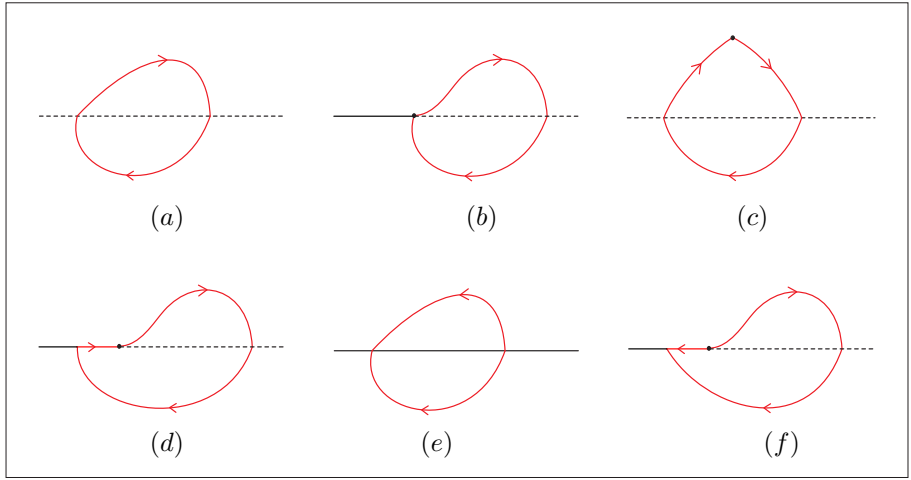
\begin{figure}[H]
	\centering
	\begin{overpic}[width=12cm]{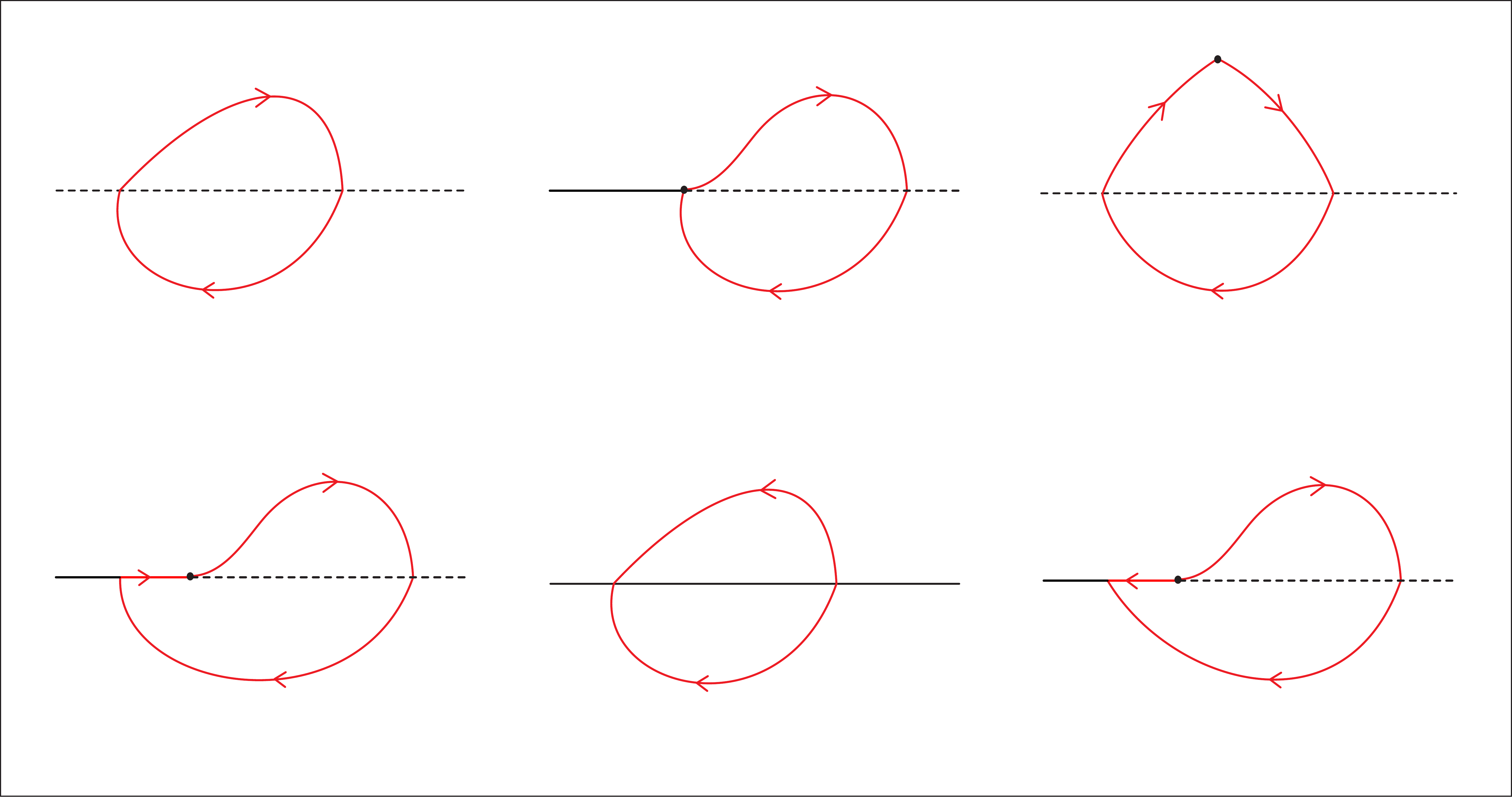}
		\put(15,28){ $(a)$}\put(50,28){ $(b)$}\put(80,28){ $(c)$}\put(15,3){ $(d)$}\put(45,3){ $(e)$}\put(83,3){ $(f)$}
	\end{overpic}
	\caption{Illustration of different types of cycles: $(a)$ simple cycle; $(b),\,(c)$ regular polycycles; $(d)$ sliding cycle; $(e)$ pseudo-cycle and $(f)$ sliding pseudo-cycle.}\label{cycles}
\end{figure}

\begin{definition}(See \cite{GST}) An unstable (resp. a stable) separatrix is either:
	\begin{itemize}
		\item[-]a regular orbit $\Gamma$ that is the unstable (resp. stable) invariant manifold of a saddle point $p\in\overline{\s^+}$ of $X$ or $p\in\overline{\s^-}$ of $Y$, denoted by $W^u(p)$ (resp. $W^s(p)$); or
		\item[-]a regular orbit that has a distinguished singularity $p\in\s$ as departing (resp. arriving) point. It is denoted by $W^u_{\pm}(p)$ (resp. $W^s_{\pm}(p)$) and $\pm$ means that it leaves (resp. arrives) from $\s^{\pm}$.
	\end{itemize}
\end{definition}
Where necessary, we use the notation $W^{s,u}_{\pm}(X,p)$ to indicate which vector field is being considered. If a separatrix is at the same time unstable and stable then it is a \textit{separatrix connection}.
A orbit $\Gamma$ that connects two singularities, $p$ and $q$, of $Z$, will be called either a \textit{homoclinic connection} if $p=q$ or a \textit{heteroclinic connection} if $p\neq q$.

\begin{definition}\label{pseudonode_saddle}
	A hyperbolic pseudo-equilibrium point $p$ is said to be a
	\begin{itemize}
		\item[-] \textit{pseudonode} if $p\in\s^s$ (resp. $p\in\s^e$) and it is an attractor (resp. a repeller) for the sliding vector field;
		\item[-] \textit{pseudosaddle} if $p\in\s^s$ (resp. $p\in\s^e$) and it is a repeller (resp. an attractor) for the sliding vector field.
	\end{itemize}
\end{definition}

\section{Degenerate cycle passing through a saddle-regular point}\label{sectionproblem}

We start by establishing the necessary generic conditions to obtain a degenerate cycle with lowest possible codimension. Consider a nonsmooth vector field $Z_0=(X_0,Y_0)\in\Omega^r$ satisfying the following conditions:

\begin{itemize}
	\item[-]$BS(1):$ $X_0$ has a hyperbolic saddle at $S_{X_0}\in\s$ and the invariant manifolds of $X_0$ at the saddle point $S_{X_0}$, $W^u(X_0,S_{X_0})$ and $W^s(X_0,S_{X_0})$, are transversal to $\s$ at $S_{X_0}$;
	\item[-]$BS(2):$ $Y_0$ is transversal to $\s$, $W^u(X_0,S_{X_0})$, and $W^s(X_0,S_{X_0})$ at $S_{X_0}$;
	\item[-]$BS(3):$ the normalized sliding vector field has $S_{X_0}$ as a hyperbolic singularity, by taking $x$ as a local chart on $\s$ at $S_{X_0}$, $Z^s_N(x)=\mu x+O(x^2)$ with $\mu\neq0$;
	\item[-]$BSC(1):$ the unstable manifold of the saddle that lies in $\s^+$, $W^u_+(X_0,S_{X_0})$, is transversal to $\s$ at $P_{X_0}\neq S_{X_0}$. We have $\varphi_{X_0}(t,P_{X_0})\in\s^+$ for all $t<0$;
	\item[-]$BSC(2)$: $Y_0$ is transversal to $\s$ at $P_{X_0}$ and there exists $t_0>0$ such that $\varphi_{Y_0}(t_0,P_{X_0})=S_{X_0}$ and $\varphi_{Y_0}(t,P_{X_0})\in\s^-$ for all $0<t<t_0$.
\end{itemize}

\begin{remark}
	Under the conditions above, the saddle-regular point is on the boundary of a crossing region and an escaping or sliding region, i.e, $S_{X_0}\in\partial\s^e\cup\partial\s^c$ or $S_{X_0}\in\partial\s^s\cup\partial\s^c$.
\end{remark}

There are two different topological types of cycles satisfying $BS(1)$-$BS(3)$ and $BSC(1)$-$BSC(2)$, see Figure \ref{degcycles-saddle-regular2}.
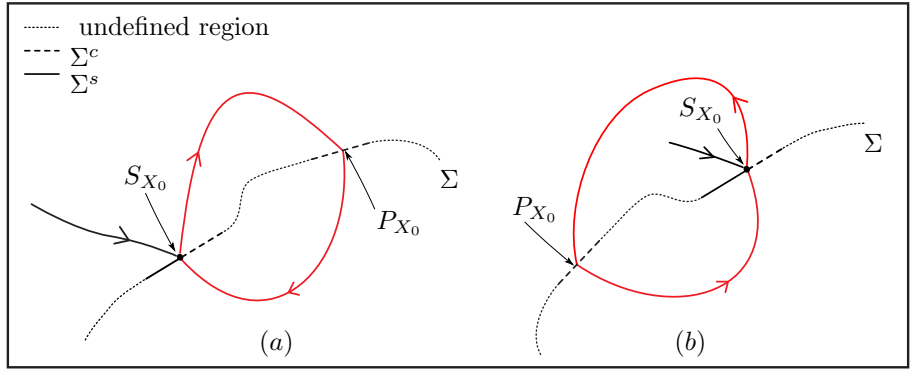
\begin{figure}[H]
	\centering
	\begin{overpic}[width=12cm]{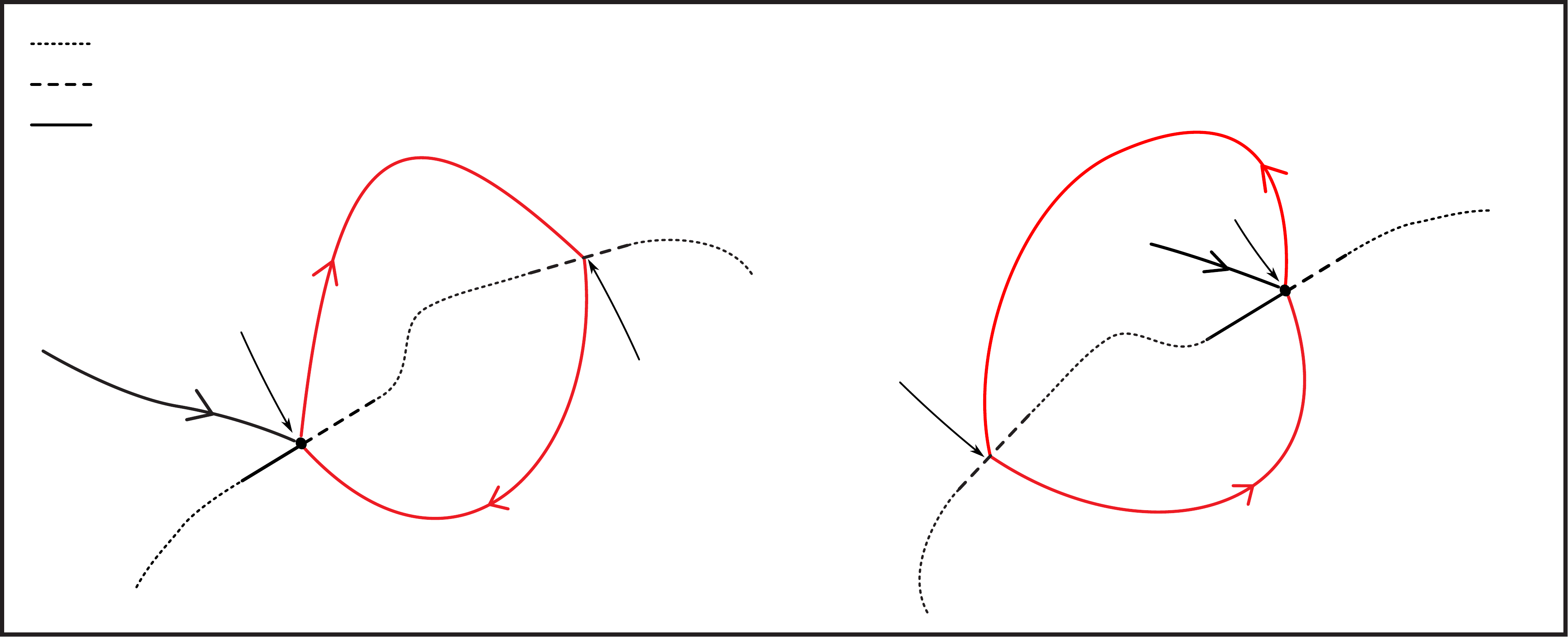}
		\put(40,15.5){ $P_{X_0}$}\put(12,20.5){ $S_{X_0}$}\put(47,20){ $\Sigma$}\put(94,24){ $\Sigma$}\put(56,17){$P_{X_0}$}\put(73.5,28){ $S_{X_0}$}
		\put(7,30.5){$\s^s$}\put(7,33.5){$\s^c$}\put(7,37){ undefined region}\put(27,2){ $(a)$}\put(73,2){ $(b)$}
	\end{overpic}
	\caption{A degenerate cycle passing through a saddle-regular point: $(a)$ $W^s_+(X_0,S_{X_0})$ contained in the unbounded region and $(b)$ $W^s_+(X_0,S_{X_0})$ contained in the bounded region. }\label{degcycles-saddle-regular2}
\end{figure}

Despite the fact that both cases shown in Figure \ref{degcycles-saddle-regular2} are topologically distinct, their analysis is similar, so we focus on case $(a)$. Whenever we refer to a degenerate cycle through a saddle point, we refer to a cycle as given in case $(a)$ of Figure \ref{degcycles-saddle-regular2}. 

To study the unfolding of the cycle, we look carefully at the local saddle-regular bifurcation, after that we perform a study on the first return map defined near the cycle. 

\subsection{Bifurcation of a saddle-regular singularity}\label{section_saddle-regular}

The simplest case is the codimension $1$ bifurcation studied in \cite{GST,KRG}, which we review briefly here for completeness. Consider a nonsmooth vector field $Z_0=(X_0,Y_0)\in \Omega^r$ satisfying conditions $BS(1)$-$BS(3)$. 
To study bifurcations of $Z_0$ near $S_{X_0}$, the following result from \cite{T1977} is important, describing bifurcations of a hyperbolic saddle point on the boundary $\s$ of the manifold with boundary $\overline{\s^+}=\s\cup\s^+$.

\begin{lemma}\label{saddle-boundary-lemma}
	Let $p\in\s$ be a hyperbolic saddle point of $X_0|_{\overline{\s^+}}$, where $X_0\in\chi^r$. Then there exist neighborhoods $B_0$ of $p$ in $\R^2$ and $\mathcal{V}_{0}$ of $X_0$ in $\chi^r$, and a $C^r$map $\be:\V_0\rightarrow\R$, such that:
	\begin{itemize}
		\item[(a)] $\be(X)=0$ if and only if $X$ has a unique equilibrium $p_X\in\s\cap B_0$ that is a hyperbolic saddle point;
		\item[(b)] if $\be(X)>0$, $X$ has a unique equilibrium $p_X\in B_0\cap \mbox{int}(\s^+)$ that is a hyperbolic saddle point;
		\item[(c)] if $\be(X)<0$, $X$ has no equilibria in $B_0\cap\overline{\s^+}$.
	\end{itemize}
\end{lemma}

Since $Y_0$ is transversal to $\s$ at $S_{X_0}$, there exist neighborhoods $B_1$, of $S_{X_0}$ in $\s$, and $\mathcal{V}_1$, of $Y_0$ in $\chi^r$, such that for any $Y\in\V_1$ and $p\in B_1$, $Y$ is transversal to $\s$ at $p$. Taking $B_0$ as given in Lemma \ref{saddle-boundary-lemma} there is no loss of generality in supposing  $B_0\cap\s=B_1$ and then restrict ourselves to the neighborhood $\V_{Z_0}=\V_0\times\V_1$ of $Z_0$ in $\Omega^r$.

If $\be(X)\neq0$, for $X\in\V_0$, there exists a fold point of $X$ in $B_1$. The fold point is located between the points where the invariant manifolds of the saddle cross $\s$. The map $s:\V_0\rightarrow\s$ that associates each $X\in\V_0$ to a tangent point $F_X\in\s$ is of class $\mathcal{C}^r$, $F_X$ is visible fold point if $\be(X)<0$, $F_X$ is a hyperbolic saddle point if $\be(X)=0$, $F_X$ is an invisible point if $\be(X)>0$.

Each $Z=(X,Y)\in\mathcal{V}_{Z_0}$ can be associated with two curves, $T_X$ and $PE_Z$, defined as follows:
\begin{itemize}
	\item[(i)] $T_X$ is the curve given implicitly by the equation $Xh(p)=0$, i.e., $T_X$ is formed by the point where $X$ is parallel to $\s$. Therefore, the intersection of $T_X$ with $\s$ gives the fold point $F_X$.
	\item[(ii)] $PE_Z$ is composed by those points where $X$ and $Y$ are parallel. So, when the intersection of $PE_Z$ with $\s$ is in $\s^s\cup\s^e$, this intersection gives the position of the pseudo-equilibrium point.
\end{itemize}
The maps $X\mapsto T_X$ and $Z\mapsto PE_Z$ are of class $C^r$. It follows from conditions $BS(1)$-$BS(3)$ that the curves $T_{X_0},\,PE_{Z_0},\, W^u(X_0,S_{X_0})$ and $W^s(X_0,S_{X_0})$ have empty intersection in $B_0$ up to the saddle point $S_{X_0}$. In fact, all these curves contain the singular point $S_{X_0}$. Condition $BS(2)$ guarantees that $PE_{Z_0},\, W^u(X_0,S_{X_0})$ and $W^s(X_0,S_{X_0})$ do not coincide in $B_0$ up to the saddle point. Condition $BS(3)$ also ensures that $T_{X_0}$ and $PE_{Z_0}$ are different in $B_0$ except at $S_{X_0}$.

The continuous dependence of the curves $T_{X_0},\,PE_{Z_0},\, W^u(X_0,S_{X_0})$ and $W^s(X_0,S_{X_0})$, on the vector field ensures that $\V_{Z_0}$ can be taken in such a way that the relative position of these curves do not change for all $Z\in\V_{Z_0}$. The curve $T_{X_0}$ is located between $W^s_+(X_0,S_{X_0})$ and $W^u_+(X_0,S_{X_0})$, see Figure \ref{BS-positions}. 

Since $S_{X_0}\in\partial\s^s\cup\partial\s^c$ there are three different cases to consider depending on the position of $PE_{Z_0}$ in relation to $T_{X_0}$, $W^s_+(X_0,S_{X_0})$, and $W^u_+(X_0,S_{X_0})$, see 
\ref{BS-positions}. These cases are named $BS_1$, $BS_2$ and $BS_3$ as in \cite{KRG}.

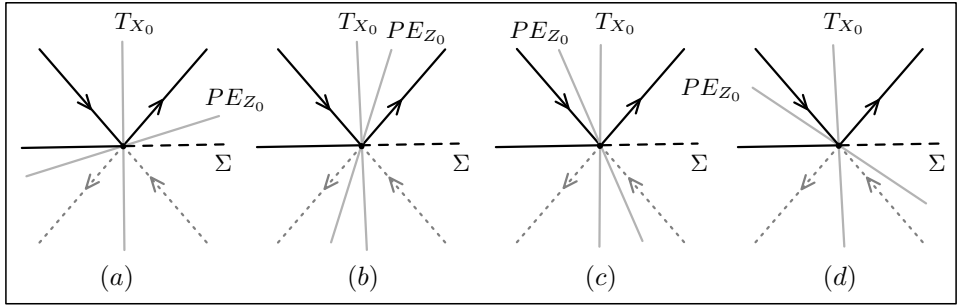
\begin{figure}[H]
	\centering
	\begin{overpic}[width=\textwidth]{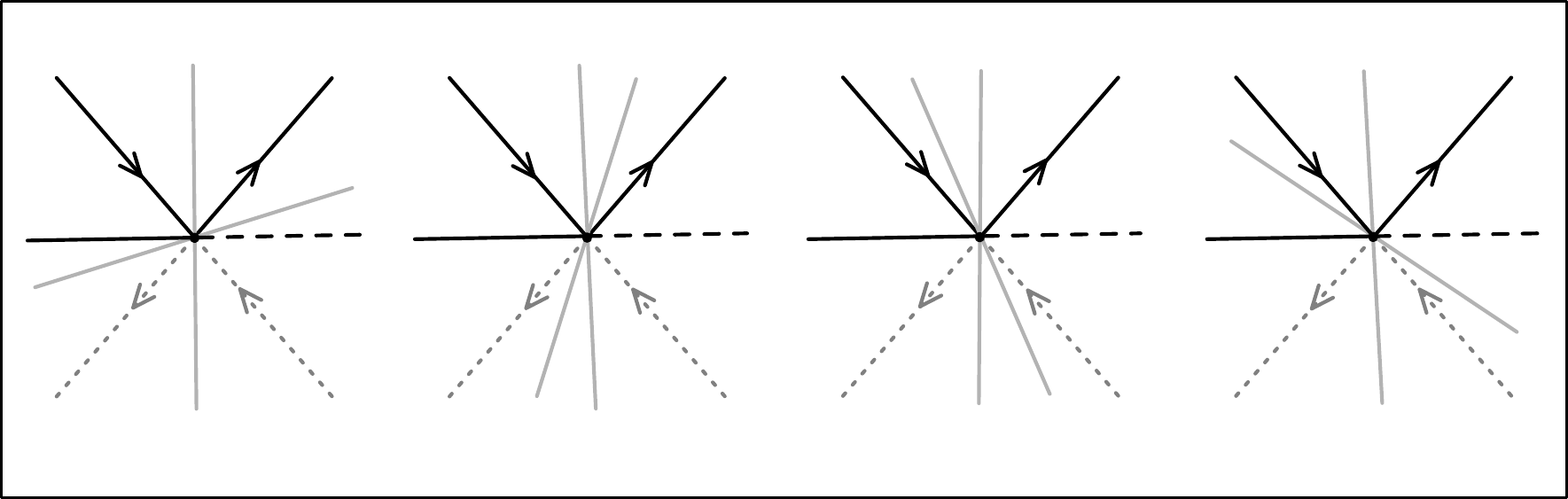}
		\put(11.5,29){\small $T_{X_0}$}\put(21,21){\small $PE_{Z_0}$}\put(35,29){\small $T_{X_0}$}\put(40,28){\small $PE_{Z_0}$}\put(62,29){\small $T_{X_0}$}\put(53,28){\small $PE_{Z_0}$}\put(86,29){\small $T_{X_0}$}\put(71,22){\small $PE_{Z_0}$}\put(22,14){\small $\s$}\put(47,14){\small $\s$}\put(72,14){\small $\s$}\put(97,14){\small $\s$}\put(9,2){ $(a)$}\put(35,2){ $(b)$} \put(60,2){ $(c)$}\put(85,2){ $(d)$}
	\end{overpic}
	\caption{Relative position of the curves $T_{X_0}$, $PE_{Z_0}$, $W^s_+(X_0,S_{X_0})$, and $W^u_+(X_0,S_{X_0})$. $(a)$ corresponds to case $BS_1$, $(b)$ corresponds to case $BS_2$ and $(c)-(d)$ both correspond to case $BS_3$.}\label{BS-positions}
\end{figure}

To understand the difference between $BS_1,\,BS_2$ and $BS_3$, for $Z=(X,Y)\in\V_{Z_0}$ let $\beta=\be(X)$ be given by Lemma \ref{saddle-boundary-lemma}. Then $S_X$ is a real saddle if $\beta>0$, a boundary saddle if $\be=0$, or a virtual saddle if $\be<0$.
\begin{itemize}
	\item[1.] Case $BS_1$: this happens when $W^u_+(S_{X_0},X_0)$ is between $T_{X_0}$ and $PE_{Z_0}$ in $\s^+$, see Figure \ref{BS-positions}$(a)$. 
	If the saddle is virtual ($\be<0$), the saddle-regular point turns into a visible fold-regular point and there is no pseudo-equilibrium. The fold-regular point is an attractor for the sliding vector field. Also, when the saddle is real ($\be>0$), an invisible fold-regular point emerges and there exists an attracting pseudo-node. The point in $\s^s$ where the unstable manifold of the saddle crosses $\s$ is located between the pseudo-equilibrium and the fold-regular point, see Figure \ref{BS1_fig}.
	\begin{figure}[H]
		\centering
		\begin{overpic}[width=12cm]{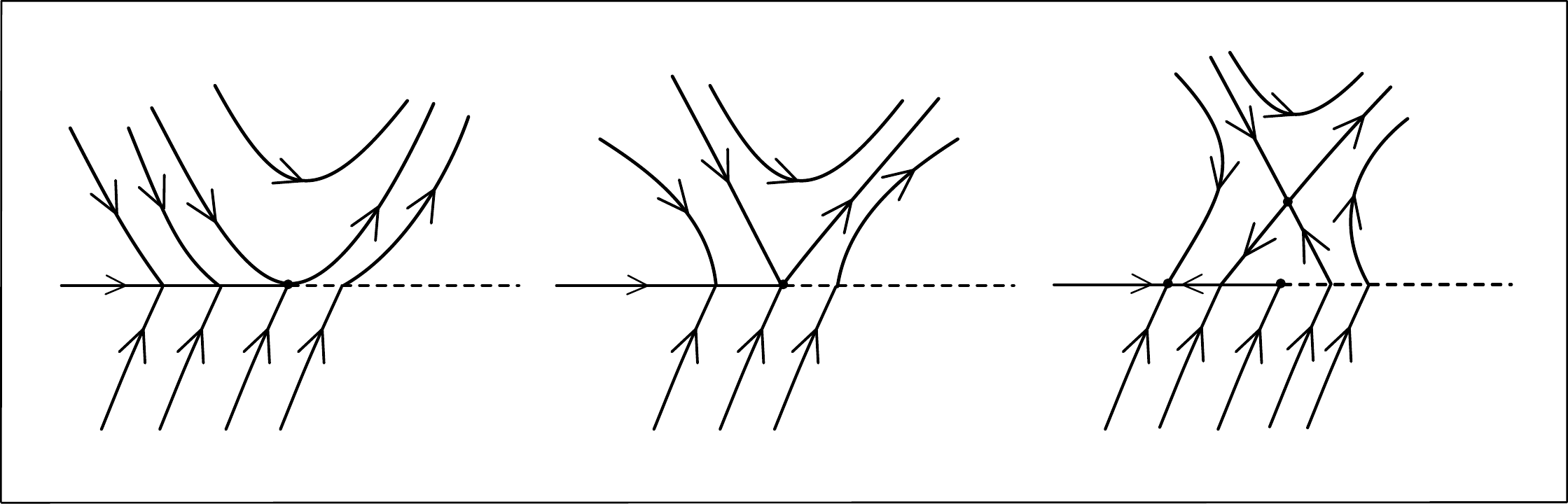}
			\put(13,1.5){ $\beta<0$}\put(45,1.5){ $\beta=0$}\put(75,1.5){ $\beta>0$}
		\end{overpic}
		\caption{Bifurcation of a saddle-regular point: case $BS_1$.}\label{BS1_fig}
	\end{figure}
	\item[2.] Case $BS_2$: this case happens when $PE_{Z_0}$ is between $T_{X_0}$ and $W^u_+(S_{X_0},X_0)$ in $\s^+$, see Figure \ref{BS-positions}$(b)$. There is a visible fold-regular point and there is no pseudo-equilibrium when the saddle is virtual ($\be<0$). The fold-regular point is an attractor for the sliding vector field. When the saddle is real ($\be>0$), an invisible fold-regular point and an attracting pseudo-node coexist. The pseudo-equilibrium is located between the point (in $\s^s$) where the unstable manifold of the saddle meets $\s$ and the fold-regular point. See Figure \ref{BS2_fig}.
	\begin{figure}[H]
		\centering
		\begin{overpic}[width=12cm]{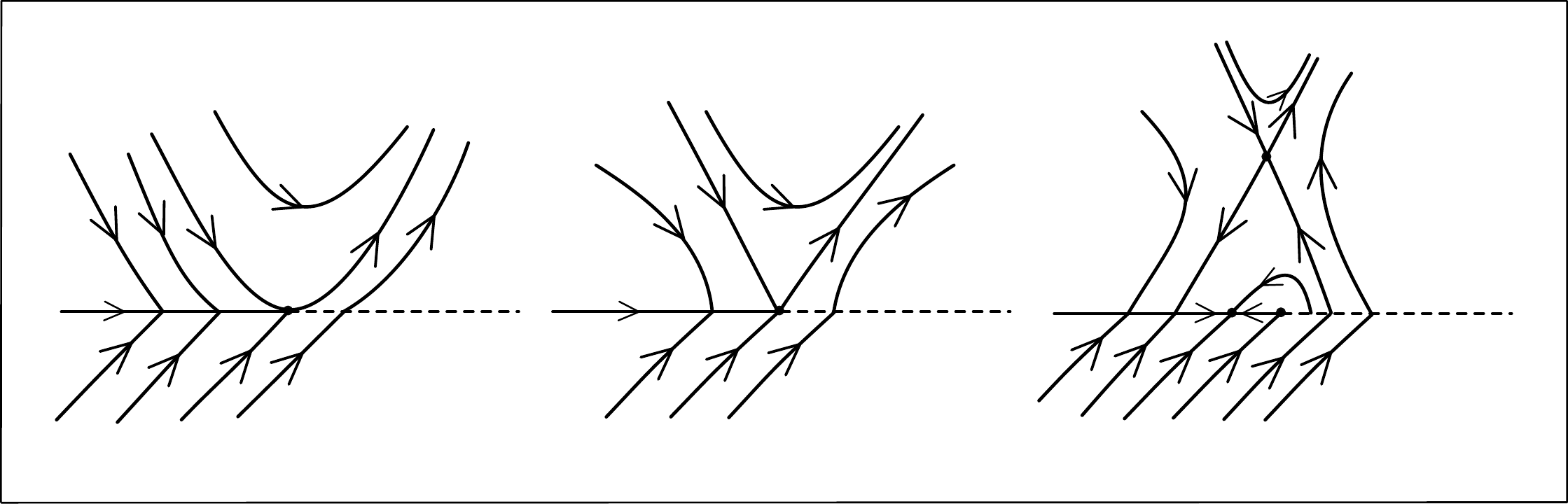}
			\put(13,1.5){ $\beta<0$}\put(45,1.5){ $\beta=0$}\put(75,1.5){ $\beta>0$}
		\end{overpic}
		\caption{Bifurcation of a saddle-regular point: case $BS_2$.}\label{BS2_fig}
	\end{figure}
	\item[3.] Case $BS_3$: this case happens when $T_{X_0}$ is between $W^u_+(S_{X_0},X_0)$ and $PE_{Z_0}$ in $\s^+$, see Figures \ref{BS-positions}$(c)$-$(d)$. When the saddle is virtual ($\be<0$), a visible fold-point coexists with a pseudo-saddle. There exists no pseudo-equilibrium when the saddle is real ($\be>0$). In this case, the saddle-regular point turns into an invisible fold-regular point that is a repeller for the sliding vector field. See Figure \ref{BS3_fig}.
	\begin{figure}[H]
		\centering
		\begin{overpic}[width=12cm]{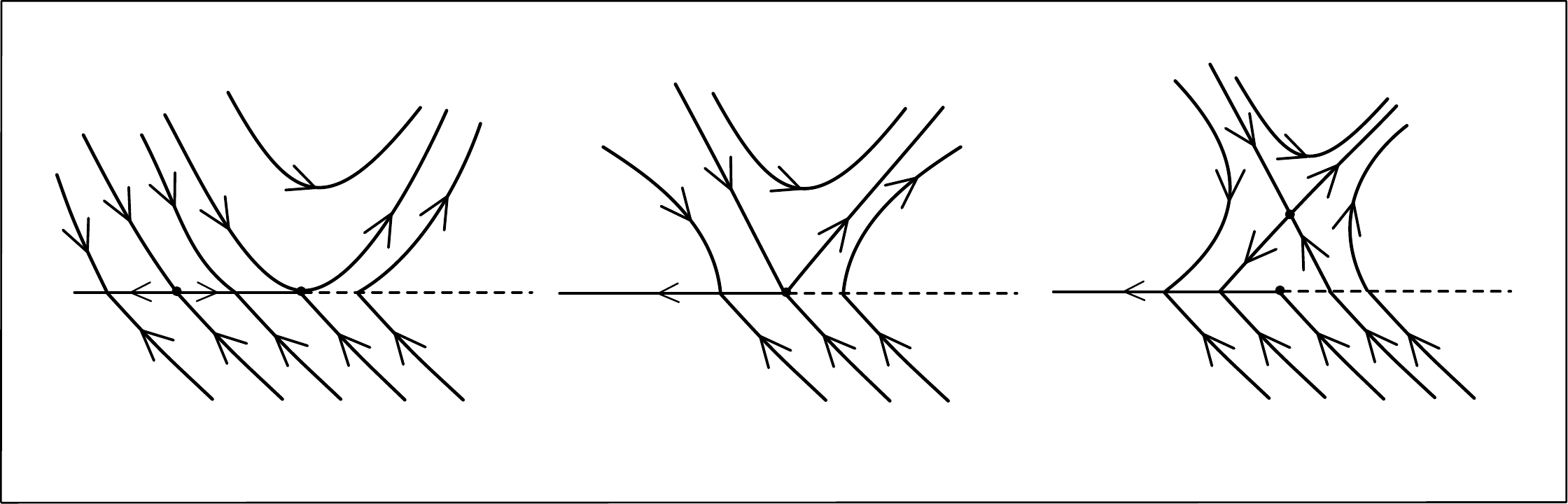}
			\put(13,1.5){ $\beta<0$}\put(45,1.5){ $\beta=0$}\put(79,1.5){$\beta>0$}
		\end{overpic}
		\caption{Bifurcation of a saddle-regular point: case $BS_3$.}\label{BS3_fig}
	\end{figure}
\end{itemize}

\subsection{Structure of the first return map}
Let $\Gamma_0$ be the degenerate cycle of $Z_0$. We show that $\V_{Z_0}$ can be chosen in such a way that, for each $Z\in\V_{Z_0}$, a first return map is defined in a half-open interval, near the cycle $\Gamma_0$. 

\begin{prop}\label{prop_exist_first_return}
	Let $Z_0$ be a nonsmooth vector field satisfying conditions $BS(1)$-$BS(3)$ and $BSC(1)$-$BSC(2)$. In addition, suppose $S_{X_0}$ is located at the origin. Then there exists a neighborhood $\mathcal{V}_{Z_0}$ of $Z_0$ in $\Omega^r$ such that, for all $Z\in\mathcal{V}_{Z_0}$, there is a well defined first return map in a half-open interval $\left[a_Z,a_Z+\delta_Z\right)$ with $\delta_Z>0$ and $a_Z\approx0$. The first return map of $Z$ can be written as $$\pi_{Z}(x)=\rho_3\circ\rho_2\circ\rho_1(x),$$
	where $\rho_2$ and $\rho_3$ are orientation reversing diffeomorphisms and $\rho_1$ is a transition map near a saddle or a fold point. 
\end{prop}

\begin{proof} We have already determined a neighborhood $\V_{Z_0}=\V_0\times\V_1$ where, for all $Z=(X,Y)\in\V_{Z_0}$, Lemma \ref{saddle-boundary-lemma} holds for $X$ and transversality conditions hold for $Y$ in $B_1=B_0\cap\s$. The claimed existence follows directly by means of continuous dependence results and properties of transversal sets.
	
Now we determine $a_Z$ for each $Z\in\V_{Z_0}$. When the saddle of $X$ is not in $\s$, there are at least three different points in which the invariant manifolds $W^{u,s}(X,S_X)$ meet $\s$. Denote these points by $P_X^i$, $i=1,2,3$, where $P_X^1,\,P_X^2\in B_1$ and $P_X^3\in I_1$ ($I_1$ is a neighborhood of $P_{X_0}$ in $\s$). Assume $P_X^1,P_X^3\in W^u(S_X,X)\cap\s$ and $P_X^2\in W^s(S_X,X)\cap\s$. We keep this notation if the saddle is on $\s$, in this case $P_X^1=P_X^2=S_X$. By taking $x$ as a local chart for $\s$ near $0$ (with $x<0$ corresponding to sliding region and $x>0$ corresponding to crossing region) and by denoting $P_X^i=x_i$, $i=1,2$, we have $x_1\leq x_2$. If $S_X\notin\s$ then there exists a tangency point $F_X\in B_1$, also $S_X=F_X$ when the saddle is on the boundary. Considering the previous chart for $\s$, denote $F_X=x_f$, then $x_1\leq x_f\leq x_2$, see Figure \ref{first_return_maps_allcases-saddle}. Observe that $\lim_{ Z\rightarrow Z_0}x_i=0$ for $i=1,2,f$. 
For $Z\in\mathcal{V}_{Z_0}$ let $a_Z$ be defined as $a_Z=x_f$ if $\be(X)<0$, $a_Z=x_1=x_2$ if $\be(X)=0$, or $a_Z=x_2$ if $\be(X)>0$.
Thus, by choosing $\delta_Z>0$ small enough the existence of the first return map in $[a_Z,a_Z+\delta_Z]$ is ensured by continuity.

To study the structure of this first return map for $Z\in\mathcal{V}_{Z_0}$, we analyze it near the saddle point. 
Without loss of generality, suppose that $\s$ is transversal to the $y$ axis at the origin. Consider a sufficiently small $\varepsilon>0$ and let $\sigma$ denote a section transversal to the flow of $X$, for $x>0$, through the point $(0,\varepsilon)$. There are three options for the position of $S_X$ in relation to $\s$, in each case the first return map is analysed differently. As above, consider $\be=\be(X)$ then:
	\begin{itemize}
		\item if $\be<0$, $S_X$ is virtual, then the first return map is limited by the visible fold point, $F_X$. So we have a transition map, $\rho_1$, from $\s$ to $\sigma$ near a fold point. See Figure \ref{first_return_maps_allcases-saddle}$(a)$;
		\item if $\be=0$, $S_X \in\s$, then we have a transition map, $\rho_1$, from $\s$ to $\sigma$ near a boundary saddle point. See Figure \ref{first_return_maps_allcases-saddle}$(b)$;
		\item if $\be>0$, $S_X$ is real, then the limit point of the first return map is the point $P_X^2$ where the stable manifold $W^s_+(X,S_X)$ intersects $\s$ near $0$. There is a transition map, $\rho_1$, from $\s$ to $\sigma$, near a real saddle point. See Figure \ref{first_return_maps_allcases-saddle}$(c)$.
	\end{itemize}
	
	After crossing through $\sigma$, the orbits will cross $\s$ near $P_X^3$. Since $\sigma$ is a transversal section, the transition from $\sigma$ to $\s$ is performed by means of a diffeomorphism $\rho_2$. The flow of $Y$ makes the transition from $\s$ (near $P_X^3$) to $\s$ (near $0$), then the transversal conditions satisfied by $Y$ give another diffeomorphism, $\rho_3$, performing this transition. The diffeomorphisms $\rho_2$ and $\rho_3$ are orientation reversing, so we can write $\pi_Z(x)=\rho_3\circ\rho_2\circ\rho_1(x)$ for $x\in[a_Z,a_Z+\delta_Z)$.
\end{proof}
\begin{figure}[H]
	\centering
	\begin{overpic}[width=14cm]{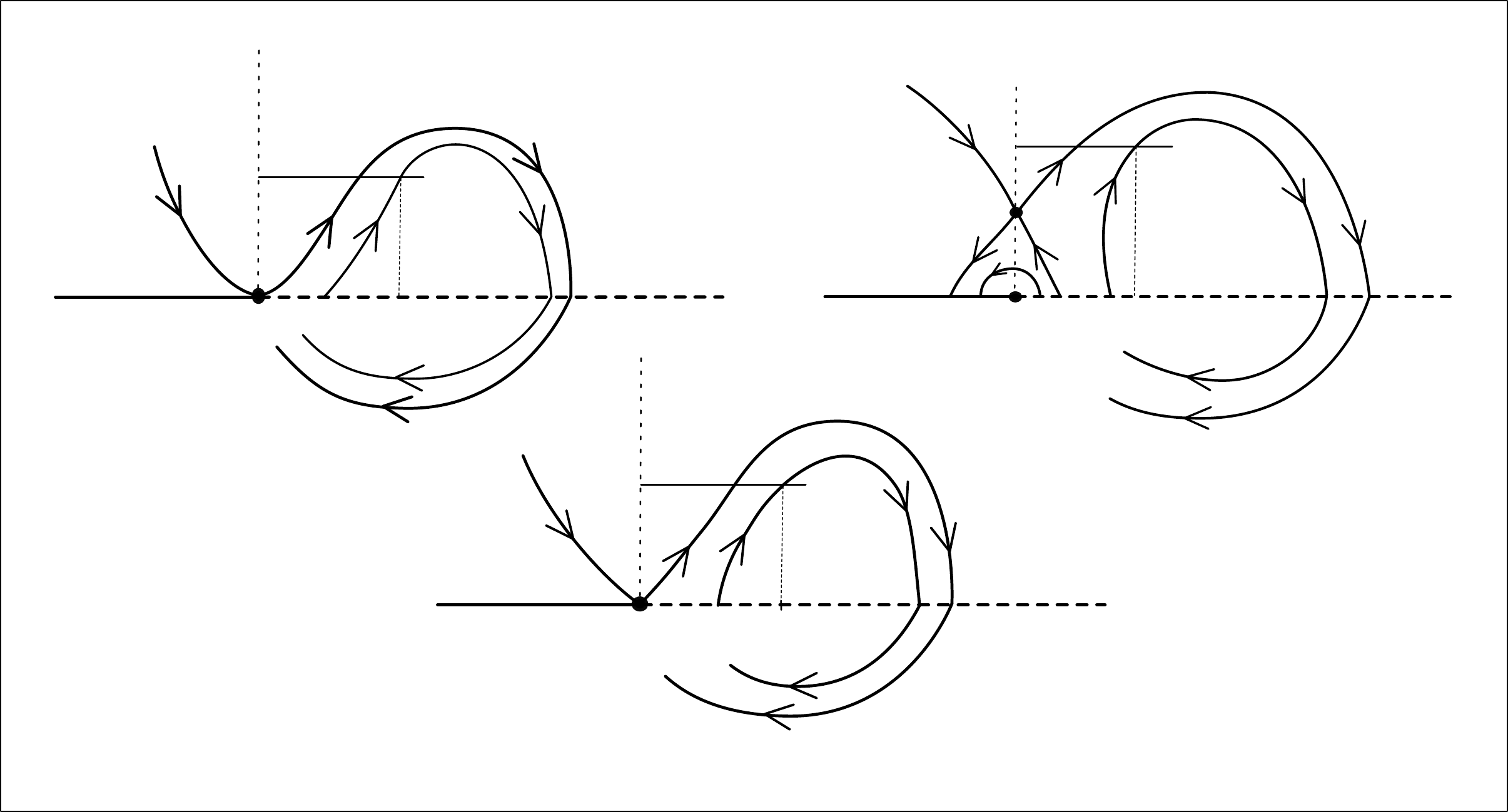}
		\put(95,31.5){\small $\Sigma$}\put(47,31.5){\small $\Sigma$}\put(72,11){\small $\Sigma$}\put(77.5,42.5){\small$\sigma$}\put(28,40.5){\small$\sigma$}\put(53,20){\small$\sigma$}\put(41,11){\small$S_X$}
		\put(16,31.5){\small$F_X$}\put(61.5,31.5){\small$P_X^1$}	\put(65.5,31.5){\small$F_X$}\put(69.5,31.5){\small$P_X^2$}\put(38,31.5){\small$P_X^3$}
		\put(91,35.5){\small$P_X^3$}
		\put(63,11){\small$P_X^3$}\put(52,16){\small$\rho_1$}\put(57,19){\small$\rho_2$}\put(52,10){\small$\rho_3$}\put(27,37){\small$\rho_1$}\put(31.5,38){\small$\rho_2$}\put(27,30.5){\small$\rho_3$}\put(76,38){\small$\rho_1$}\put(81,41){\small$\rho_2$}\put(80,30.5){\small$\rho_3$}\put(79,22){ $(c)$}\put(25,22){ $(a)$}\put(51,2){$(b)$}
	\end{overpic}
	\caption{Illustration of the first return map with transversal section $\tau$: (a) $\be<0$, (b) $\be=0$ and (c) $\be>0$.}\label{first_return_maps_allcases-saddle}
\end{figure}
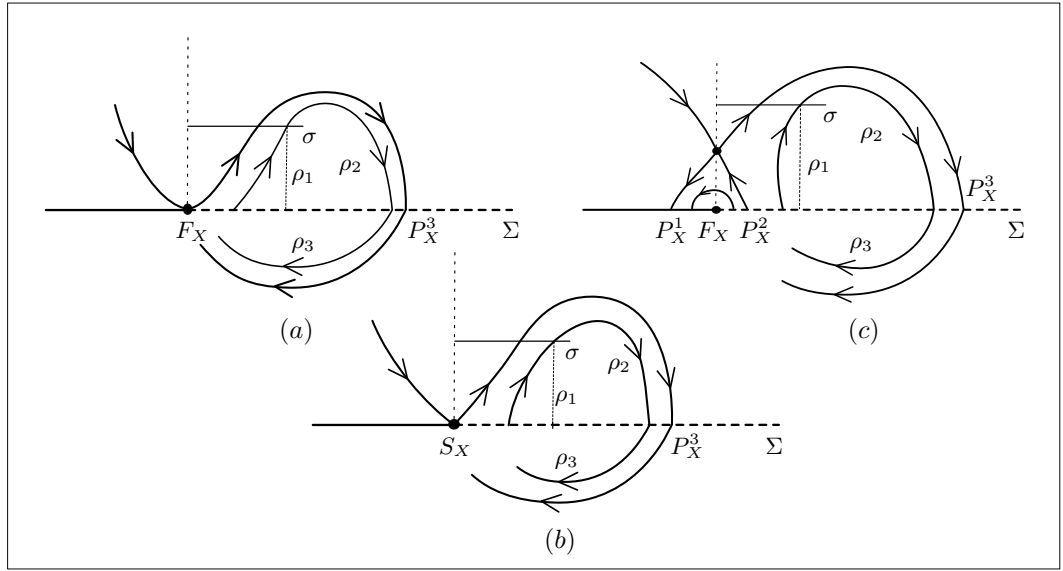

Since $\rho_2$ and $\rho_3$ given in Proposition \ref{prop_exist_first_return} are diffeomorphisms, the difficult part in understanding the first return map $\pi_Z$ is the structure of $\rho_1$. From now on we assume $\sigma\subset\{(x,1)\in\R^2;x\in\R\}$ and $Z_0=(X_0,Y_0)\in\Omega^{\infty}$, since high differentiability classes are required. We restrict the analysis to nonresonant saddles, meaning we consider $Z=(X,Y)\in\V_{Z_0}$ such that the hyperbolicity ratio of $S_X$, $r$, is an irrational number, ($r=-\la_2/\la_1$ where $\la_2<0<\la_1$ are the eigenvalues of $DX(S_X)$). The point $S_X$ is assumed to lie at the origin. 

According to \cite{R1998}, for each $l$, $X$ is $C^l$-conjugated, around $S_{X}$, to the normal form 
\begin{equation}\label{normal-form-irrat-eq}
\tilde{X}(x,y)= -rx\dx+y\dy.
\end{equation}
Let us assume that $\s=h_k^{-1}(0)$ where $h_k(x,y)=y-x+k$. Then $\s$ intersects the axes at $(0,-k)$ and $(k,0)$, implying that if $k>0$ then the saddle is real, if $k=0$ then the saddle is on the boundary, and if $k<0$ the saddle is virtual. 

Denote the flow of $\tilde{X}$ by $\varphi_{\tilde{X}}(t,x,y)=(\varphi_1(t,x,y),\varphi_2(t,x,y))^T$. 
In this case, the transition time from $\s$ to $\sigma$ is easily calculated and it is given by $t_1(x)=-\ln\left(x-k\right)$ for each $(x,x-k)\in\s$. Therefore, the transition map $\tilde{\rho}$, from $\s$ to $\sigma$, is given by $\tilde{\rho}(x)=\varphi_2(t_1(x),x,x-k)=e^{-rt_1(x)}x=x(x-k)^r=k(x-k)^r+(x-k)^{r+1}$. Since $X$ and $\tilde{X}$ are conjugated, there must exist diffeomorphisms $\phi$ and $\psi$, defined in a neighborhood of the origin, such that $\rho=\phi\circ\tilde{\rho}\circ\psi$ and $\psi(0)=\phi(0)=0$. We can also assume that $\tilde{a}=\psi(a_Z)$ is equivalent to $a_Z$, i.e., the transition map $\tilde{\rho}$ of $\tilde{X}$ is defined in $[\tilde{a},\tilde{a}+\tilde{\delta})$, $\tilde{\delta}>0$. Then: 
\begin{itemize}
	\item if $k\geq0$, then $\tilde{a}=k$. This means $P_2=(k,0)$ for $k>0$ and $S_X=(0,0)$ for $k=0$; see Figures \ref{transition_map_saddle-normal}$(a)$ and \ref{transition_map_saddle-normal}$(b)$.
	\item if $k<0$, then $k<\tilde{a}=\dfrac{k}{1+r}<0$. This means that $F_X=\left(\dfrac{k}{1+r},\dfrac{-kr}{1+r}\right)$; see Figure \ref{transition_map_saddle-normal}$(c)$. 
\end{itemize}
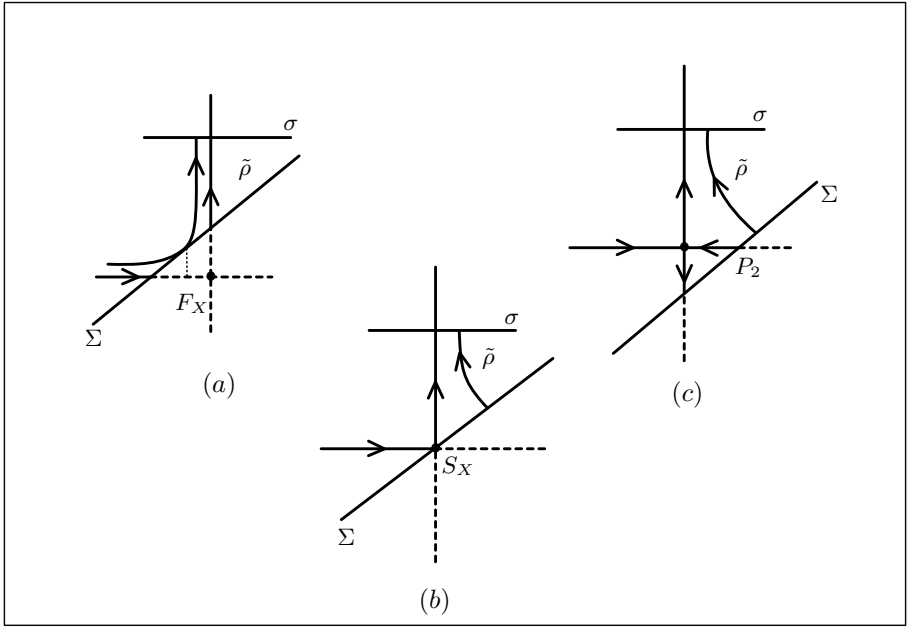
\begin{figure}[H]
	\centering
	\begin{overpic}[width=12cm]{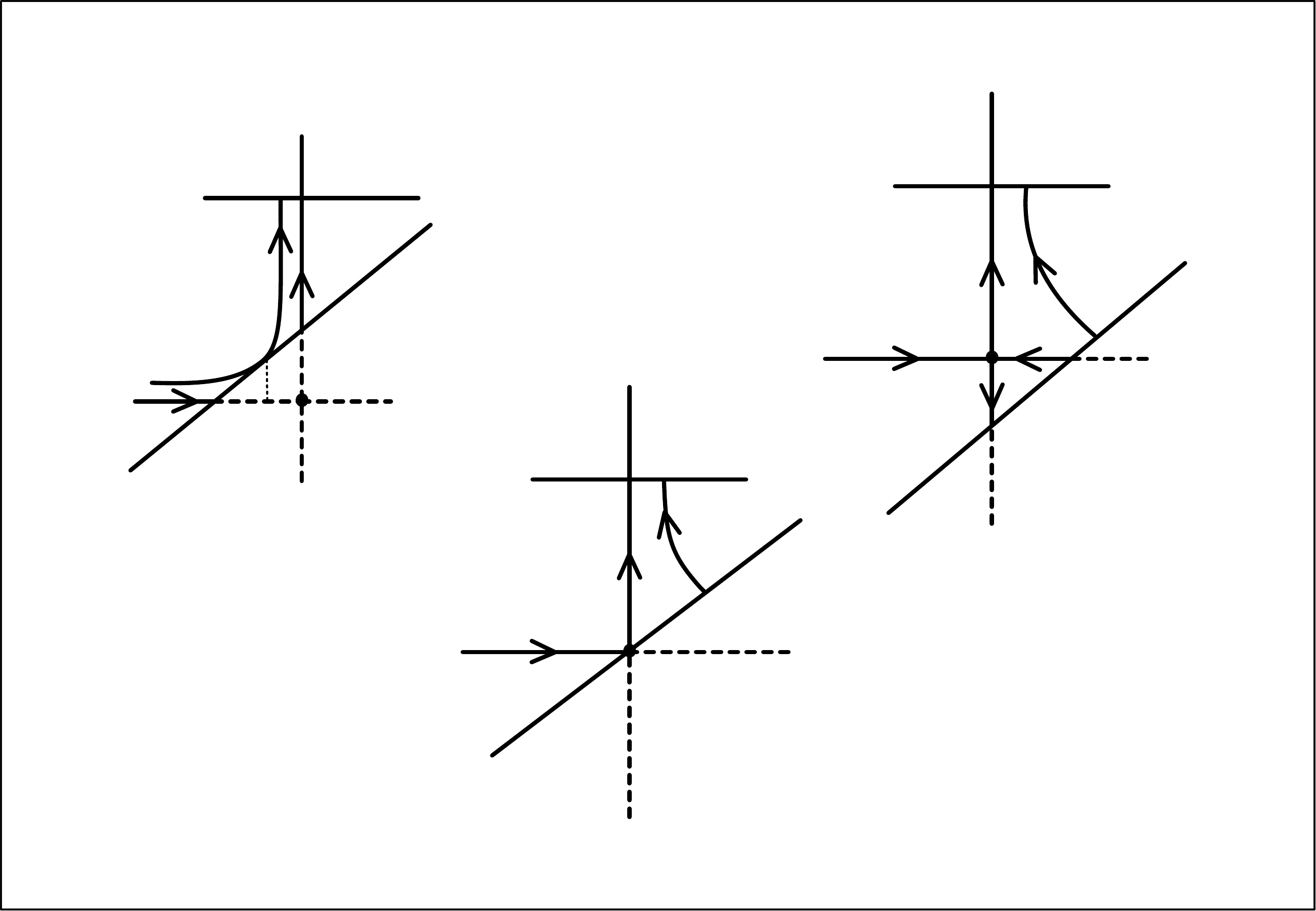}
		\put(9,31){\small $\Sigma$}\put(90.5,47){\small $\Sigma$}\put(37,9){\small $\Sigma$}\put(31,55){\small$\sigma$}\put(83,56){\small$\sigma$}\put(55.5,33.5){\small$\sigma$}\put(19,35){\small$F_X$}\put(81,39){\small$P_2$}\put(48.5,17){\small$S_X$}\put(53,29){\small$\tilde{\rho}$}\put(26,50){\small$\tilde{\rho}$}\put(81,50){\small$\tilde{\rho}$}\put(73,25){ $(c)$}\put(21,26){ $(a)$}\put(46,2){$(b)$}
	\end{overpic}
	\caption{Illustration of the transition map $\tilde{\rho}$, from $\s$ to $\sigma$: (a) $k<0$, (b) $k=0$ and (c) $k>0$.}\label{transition_map_saddle-normal}
\end{figure}

\begin{prop}\label{derivatives-of-transition2}
	Consider $Z=(X,Y)\in\V_{Z_0}$, let $r\notin\Q$ be the hyperbolicity ratio of $S_X$ and $s=[r]$. Let $\be=\be(X)$ be defined in Lemma \ref{saddle-boundary-lemma} and $\pi_Z$ defined in Proposition \ref{prop_exist_first_return}. Then:
	\begin{itemize}
		\item[(i)] $\dfrac{d}{dx}\pi_Z(a_Z)=0$ and $ \dfrac{d^2}{dx^2}\pi_Z(a_Z)>0$ if $\be<0$;
		\item[(ii)] $\lim\limits_{x\rightarrow a_Z^+}\dfrac{d}{dx}\pi_Z(x)=\cdots=\lim\limits_{x\rightarrow a_Z^+}\dfrac{d^{s+1}}{dx^{s+1}}\pi_Z(x)=0$ and $\lim\limits_{x\rightarrow a_Z^+}\dfrac{d^{s+2}}{dx^{s+2}}\pi_Z(x)=+\infty$ if $\be=0$ and $r>1$;
		\item[(iii)] $\lim\limits_{x\rightarrow a_Z^+}\dfrac{d}{dx}\pi_Z(x)=+\infty$ if $\be\leq0$ and $r<1$; 
		\item[(iv)] $\lim\limits_{x\rightarrow a_Z^+}\dfrac{d}{dx}\pi_Z(x)=\cdots=\lim\limits_{x\rightarrow a_Z^+}\dfrac{d^{s}}{dx^{s}}\pi_Z(x)=0$ and $\lim\limits_{x\rightarrow a_Z^+}\dfrac{d^{s+1}}{dx^{s+1}}\pi_Z(x)=+\infty$ if $\be>0$ and $r>1$.
	\end{itemize}
\end{prop}
\begin{proof}
	From Proposition \ref{prop_exist_first_return} we know $\pi_Z=\rho_3\circ\rho_2\circ\rho_1$. Observe that $\rho_2$ and $\rho_3$ are orientation reversing diffeomorphisms of class $C^{\infty}$, and that $\rho_1=\phi\circ\tilde{\rho}\circ\psi$ where $\phi$ and $\psi$ are orientation preserving diffeomorphisms of class $\mathcal{C}^l$, $l>s+2$. Define $\Phi=\rho_3\circ\rho_2\circ\phi$, then $\pi_Z=\Phi\circ\tilde{\rho}\circ\psi$, where $\Phi$ and $\psi$ are orientation preserving diffeomorphisms of class $\mathcal{C}^{l}$ with $l>s+2$. By definition, $\tilde{a}=\psi(a_Z)$. Let $I_0$ be a neighborhood of $a_Z$ where $\pi_Z=\Phi\circ\tilde{\rho}\circ\psi(x)$ is well defined for $x\in I_0$. Then the derivatives of order $1\leq i\leq s+2$ of $\Phi$ and $\psi$ are limited in $I_0$, and $\tilde{\rho}$ is differentiable in $(\tilde{a},\tilde{a}+\tilde{\delta})$ for $\tilde{\delta}>0$ sufficiently small. 
	
	Now $$\dfrac{d}{dx}\pi_Z(x)= \dfrac{d}{dx}\Phi(\tilde{\rho}(\psi(x)))\dfrac{d}{dx}\tilde{\rho}(\psi(x))\dfrac{d}{dx}\psi(x).$$
	Thus, for each $1< i\leq s+2$, the result follow by means of the chain and product rules for derivatives and from the fact that $\psi(x)\rightarrow\tilde{a}^+$ when $x\rightarrow a_Z^+$.
\end{proof}

\begin{prop}\label{limitcycle-prop}
	Consider $Z=(X,Y)\in\V_{Z_0}$ and suppose that $r\notin\Q$ is the hyperbolicity ratio of $S_X$. Then the following statements hold:
	\begin{itemize}
		\item[(i)] If $\pi_Z(a_Z)>a_Z$ and either $r>1$ or $r<1$ and $\beta(Z)\leq0$, then there exists an attractor fixed point of $\pi_Z$, $x_0\in(a_Z,\delta_Z)$, which corresponds to an attracting limit cycle of $Z$.
		\item[(ii)] If $\pi_Z(a_Z)=a_Z$ and either $r>1$ or $r<1$ and $\beta(Z)\leq0$ then, $a_Z$ is an attractor for $\pi_Z$. So, there exists an attracting degenerate cycle for $Z$ through a fold-regular point if $\be(Z)<0$, a saddle-regular point if $\be(Z)=0$, or a real saddle if $\be(Z)>0$.
		\item[(iii)] If $\pi_Z(a_Z)<a_Z$, $r<1$, and $\beta(Z)>0$, then there exists a repelling fixed point of $\pi_Z$, $x_0\in(a_Z,\delta_Z)$, which corresponds to a repelling limit cycle of $Z$.
	\end{itemize}
	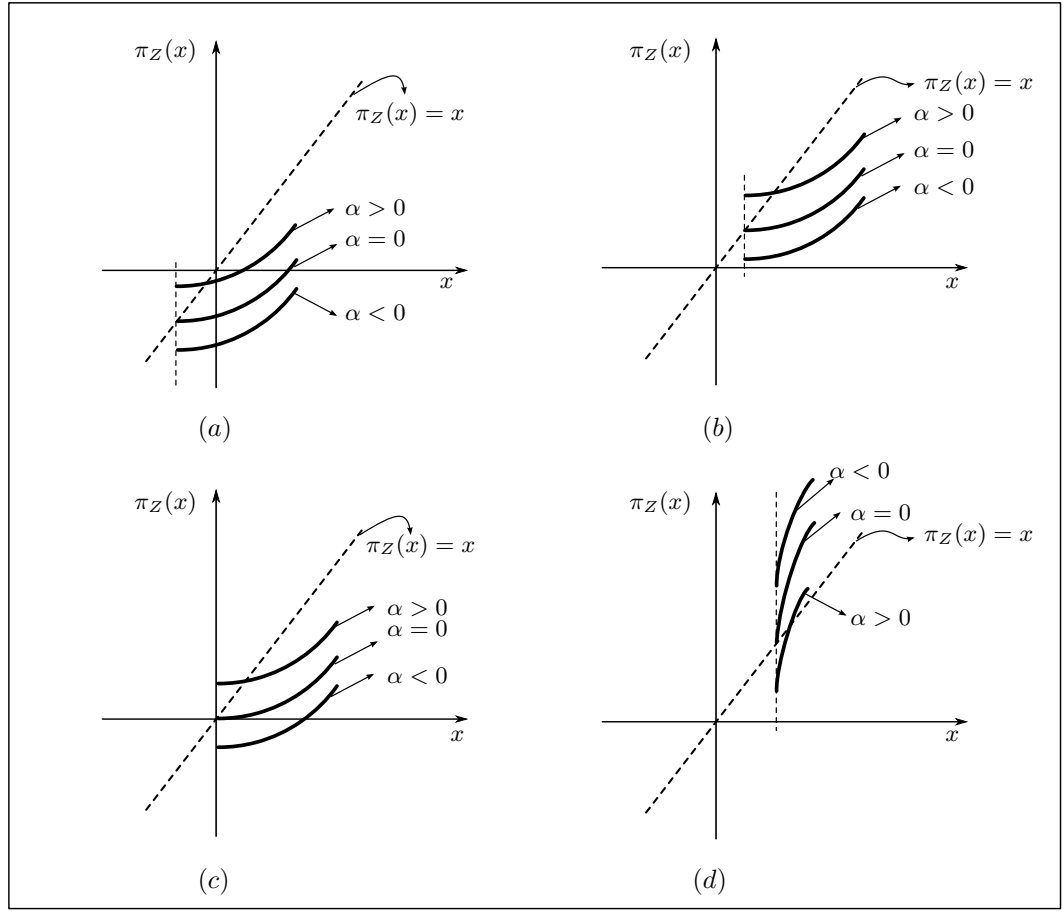
\begin{figure}[H]
		\centering
		\begin{overpic}[width=14cm]{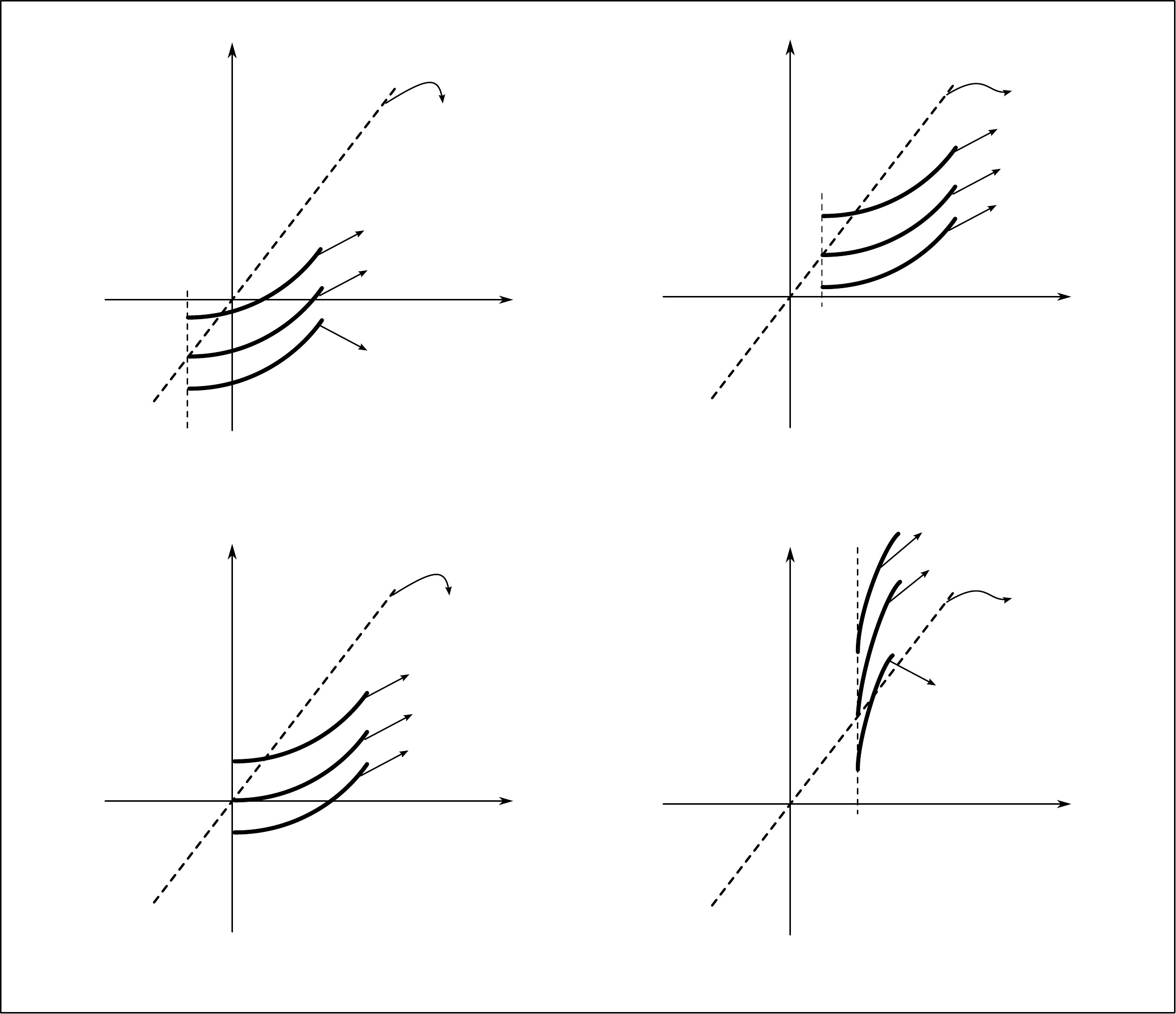}
			\put(32,66){\small $\al>0$}\put(32,63){\small ${\al}=0$}\put(32,56){\small ${\al}<0$}\put(41,59){\small$x$}\put(12,81){\small${\pi_Z}(x)$}\put(33,75){\small${\pi_Z}(x)=x$}	\put(86,75){\small ${\al}>0$}\put(86,71.5){\small ${\al}=0$}\put(86,68){\small ${\al}<0$}\put(89,59){\small$x$}\put(59,81){\small${\pi_Z}(x)$}\put(87,78){\small${\pi_Z}(x)=x$}	\put(36,28){\small ${\al}>0$}\put(36,26){\small ${\al}=0$}\put(36,21.5){\small ${\al}<0$}\put(42,16){\small$x$}\put(12,38){\small${\pi_Z}(x)$}\put(34,34){\small${\pi_Z}(x)=x$}	\put(80,27){\small $\al>0$}\put(80,37){\small ${\al}=0$}\put(78,41){\small ${\al}<0$}\put(89,16){\small$x$}\put(59,38){\small${\pi_Z}(x)$}\put(87,35){\small${\pi_Z}(x)=x$}\put(18,45){$(a)$}\put(65,45){ $(b)$}\put(18,2){$(c)$}\put(65,2){$(d)$}
		\end{overpic}
		\caption{Illustration of the graph of the first return map $\pi_Z(x)$ for $x\in[a_Z,a_Z+\delta_Z)$ where $\al=\pi_Z(a_Z)-a_Z$. $(a)$ $\be(Z)<0$, $(b)$ $\be(Z)>0$ and $r>1$, $(c)$ $\be(Z)=0$, and $(d)$ $\be(Z)>0$ and $r<1$. }\label{SaddleCycle-graph}
	\end{figure}
\end{prop}
\begin{proof}
	From the definition of $\tilde{\rho}(x)$ it is increasing in $[\tilde{a},\tilde{a}+\tilde{\delta})$. Also, from the proof of Proposition \ref{derivatives-of-transition2} we have $\pi_Z=\Phi\circ\tilde{\rho}\circ\psi$ where $\Phi$ and $\psi$ are orientation preserving diffeomorphisms, or, in other words, increasing maps since the first derivative of these maps are positive.
	
	Also from Proposition \ref{derivatives-of-transition2}, if $r>1$, the tangent vector of the graph of $\pi_Z(x)$ tends to be horizontal when $x\rightarrow a_Z^+$. Analogously, if $r<1$, the tangent vector of the graph of $\pi_Z(x)$ tends to be vertical when $x\rightarrow a_Z^+$. See Figure \ref{SaddleCycle-graph}.
	
	So if $\pi_Z(a_Z)=a_Z$ and $r>1$, taking the smallest possible $\delta_Z$ if necessary, $a_Z$ is the unique fixed point of $\pi_Z$ in $[a_Z,a_Z+\delta_Z)$ and $\pi_Z(x)<x$ for all $x\in(a_Z,a_Z+\delta_Z)$. Therefore $a_Z$ is an attracting fixed point of $\pi_Z$ that corresponds to a degenerate cycle $Z$. The definition of $a_Z$ gives the different types of cycles as listed in item $(ii)$. Analogously, if $\pi_Z(a_Z)=a_Z$, $r<1$ and $\be(Z)>0$, $a_Z$ is the unique fixed point of $p_Z$ in $[a_Z,a_Z+\delta_Z)$ and $\pi_Z(x)>x$ for all $x\in(a_Z,a_Z+\delta_Z)$. Then $a_Z$ is a repeller for $\pi_Z$ that corresponds to a repelling degenerate cycle of $Z$ through a real saddle point. This proves item $(ii)$.
	
	If $\pi_Z(a_Z)>a_Z$ and either $r>1$ or $r<1$ and $\be(Z)\leq0$, the analysis is similar to the case in item $(ii)$, the only difference is that the fixed point will change. Since in these cases the tangent vector of the graph of $\pi_Z$ tends to be horizontal at $a_Z$, for $Z$ sufficiently near $Z_0$ (taking the smallest possible $\V_{Z_0}$ if necessary), the graph of $\pi_Z$ intersects the graph of the identity map at a point $x_0\in(a_Z,a_Z+\delta_Z)$. Also, $\pi_Z(x)>x$ for all $x\in[a_Z,x_0)$ and $\pi_Z(x)<x$ for all $x\in(x_0,a_Z+\delta_Z)$. This proves item $(i)$.
	
	To prove item $(iii)$ it is enough to observe that, since the tangent vector of the graph of $\pi_Z$ tends to be vertical at $a_Z$, we obtain that the graph of $\pi_Z$ will cross the graph of the identity map at a point $x_0\in(a_Z,a_Z+\delta_Z)$. Also, $\pi_Z(x)<x$ for all $x\in[a_Z,x_0)$ and $\pi_Z(x)>x$ for all $x\in(x_0,a_Z+\delta_Z)$. This proves item $(iii)$. 
\end{proof}

%=================================================================================
%---------------------------------------------------------------------------------
\section{Main Results and Bifurcation Diagrams}\label{mainsec}
%---------------------------------------------------------------------------------
%=================================================================================
In this section, we focus on a discussion of all phenomena of codimension $1$ that appear in the characterization of the bifurcation diagram in question. There exist two independent ways of breaking the structure of the degenerate cycle of $Z_0$: to translate of the saddle or to destroy the homoclinic connection. More specifically, there exist two bifurcation parameters to consider, $\beta$ and $\alpha$, given as following:

\begin{itemize}
	\item[-] $\beta$ is a $C^{r}$-map that determines if the saddle of $X$ is real, on the boundary, or virtual. This map was determined in Lemma \ref{saddle-boundary-lemma}, and now it is naturally extended to $\V_{Z_0}$, 
	\begin{equation*}
	\begin{array}{cccc}
	\be:&\mathcal{V}(Z_0) & \rightarrow & \R\\
	& (X,Y) & \mapsto & \be(X)
	\end{array}.
	\end{equation*}
	\item[-] $\alpha$ is a $C^r$-map that determines whether or not the first return map has a fixed point, 
	\begin{equation*}
	\begin{array}{cccc}
	\alpha:&\mathcal{V}(Z_0) & \rightarrow & \R \\
	&Z & \mapsto & \pi_Z(a_Z)-a_Z
	\end{array}.
	\end{equation*}
	Despite the fact that $a_Z$ was defined in terms of $\mbox{Sgn}(\be)$, the quantity $\alpha(Z)$ is an intrinsic feature of $Z$. 
\end{itemize}

Define $\V_0=\{Z=(X,Y)\in\V_{Z_0}; S_X \mbox{ has irrational hyperbolicity ratio} \}$. From now on we only consider vector fields in $\V_0$ and describe the bifurcations of $Z_0$ in $\V_0$. To do so, consider a family of nonsmooth vector fields $Z_{\al,\be}$ of vector fields in $\V_0$, $(\al,\be)\in\mathcal{B}_0\subset\R^2$, such that $\mathcal{B}_0$ is an open neighborhood of $0\in\R^2$, $Z_{0,0}=Z_0$ and $\al$, $\be$ are the bifurcation parameters discussed above. 

From parameters $\al=\al(Z)$ and $\be=\be(Z)$ we can obtain all the bifurcations of the degenerate cycle $\Gamma_0$, of $Z_0$ in $\V_{0}$. Consider $Z_0\in \V_0$ satisfying $BS(1)$-$BS(3)$ and $BSC(1)$-$BSC(2)$. From the previous sections there are six cases to analyze:
\begin{itemize}
	\item[-] $DSC_{11}$: $Z_0$ has a saddle-regular point of type $BS_1$ and the hyperbolicity ratio of $X_0$ is greater than one;
	\item[-] $DSC_{12}$: $Z_0$ has a saddle-regular point of type $BS_1$ and the hyperbolicity ratio of $X_0$ is smaller than one;
	\item[-] $DSC_{21}$: $Z_0$ has a saddle-regular point of type $BS_2$ and the hyperbolicity ratio of $X_0$ is greater than one;
	\item[-] $DSC_{22}$: $Z_0$ has a saddle-regular point of type $BS_2$ and the hyperbolicity ratio of $X_0$ is smaller than one;
	\item[-] $DSC_{31}$: $Z_0$ has a saddle-regular point of type $BS_3$ and the hyperbolicity ratio of $X_0$ is greater than one;
	\item[-] $DSC_{32}$: $Z_0$ has a saddle-regular point of type $BS_3$ and the hyperbolicity ratio of $X_0$ is smaller than one.
\end{itemize}

All the cases have at least four bifurcation curves in the $(\al,\be)$-plane, given implicitly as functions of $\al$ and $\be$. To describe the codimension $1$ phenomena curves we identify $\s$ with $\R$.

\begin{itemize}
	\item[-] For $Z=(X,Y)\in\V_0$, let $P_E(Z)$ be the point $\s\cap PE_Z$ and $PE_Z$ is the curve where $X$ is parallel to $Y$. Let $\gamma_{P_E}$ be the curve implicitly defined by $\pi(a_{Z_{\al,\be}})-P_E(Z_{\al,\be})=0$, i.e., $\gamma_{P_E}=\{(\al,\be)\in\mathcal{B}_0;\pi(a_{Z_{\al,\be}})-P_E(Z_{\al,\be})=0 \}$. 
	Then $\gamma_{P_E}$ is the curve for which there exists a connection between the pseudo equilibrium and the point $(a_{Z_{\al,\be}},0)$. Thus this curve lies either on the half plane $\be>0$ or in the half plane $\be<0$.
	\item[-]  For $Z=(X,Y)\in\V_0$, let $F(Z)$ be the fold point of $X$ in $\s$ near $S_X$. Remember that $F_X$ is an invisible fold-regular point if $\beta(Z)>0$, a visible fold-regular point if $\be(Z)<0$, and a saddle-regular point if $\be(Z)=0$.		
	Define $\gamma_{F}=\{(\al,\be)\in\mathcal{B}_0;\pi(a_{Z_{\al,\be}})-F(Z_{\al,\be})=0 \}$.
	 Then $\gamma_F$ is the curve providing a connection between the fold point $F_{X_{\al,\be}}$ and the point 
	$(a_{Z_{\al,\be}},0)$. Since, for $\be\leq0$, $(a_{Z_{\al,\be}},0)$ corresponds to the fold point of $Z_{\alpha,\be}$, this curve coincides with the axis $\al$.
	\item[-] For $Z=(X,Y)\in\V_0$, let $P_1(Z)$ be the points in $\s$ where the invariant manifolds of $X$ at $S_X$ cross $\s$, as defined previously. Define the curve $\gamma_{P_1}=\{(\al,\be)\in\mathcal{B}_0;\be\geq0 \mbox{ and } \pi(a_{Z_{\al,\be}})-P_1(Z_{\al,\be})=0 \}$ that provides a pseudo-homoclinic connection between $P_1(Z_{\al,\be})$ and the saddle point.
\end{itemize}

These curves will be illustrated later in the bifurcation diagrams.
In some cases extra bifurcation curves will emerge. 
To obtain an order relation between the curves, we denote, with some abuse of terminology, $\gamma_j(\al,\be)=\{(\al,\be)\in\mathcal{B}_0;\pi(a_{Z_{\al,\be}})-j(Z_{\al,\be})=0\}$, for $j=P_E,\,F,P_1$.
Now we are ready to describe the bifurcation diagrams for the family $Z_{\al,\be}$.

%++++++++++++++++++++++++++++++++++++++++++++++++++++++
%%%%Cycles DSC_11 and DSC_12
%++++++++++++++++++++++++++++++++++++++++++++++++++++++++++++++++++++++++++++++++++++
The following three theorems concern cycles of types $DSC_{11}$ and $DSC_{12}$.
\begin{thmA}\label{DSC11_theo1}
	Suppose that $Z_0$ is of type $DSC_{11}$ and $\beta>0$. Then for a family $Z_{\alpha,\beta}=(X_{\alpha,\beta},Y_{\alpha,\beta})\in\V_0$, bifurcation curves, $\gamma_{P_E},\,\gamma_{P_1},$ and $\gamma_{P_F}$, emerge from the origin, there exists an attracting pseudo-node, and the following statements hold:
	\begin{itemize}
		\item[(a)] if $(\alpha,\beta)\in R^1_7$, where $R^1_7=\{(\al,\be);0<\be<\gamma_{P_E}(\al,\be) \}$, then there exists a sliding polycycle passing through $S_{X_{\alpha,\beta}}$ and $P_E({Z_{\alpha,\beta}})$, which contains two segments of sliding orbits;
		\item[(b)] if $(\alpha,\beta)\in\gamma_{P_E}$, then there exists a sliding polycycle passing through $S_{X_{\alpha,\beta}}$ and $P_E({Z_{\alpha,\beta}})$, which contains just one segment of sliding orbit;
		\item[(c)] if $(\alpha,\beta)\in R^1_6=\{(\al,\be);\gamma_{P_E}(\al,\be)<\be<\gamma_{P_1}(\al,\be)\}$, then there exists a sliding pseudo-cycle passing through $S_{X_{\alpha,\beta}}$;
		\item[(d)] if $(\alpha,\beta)\in\gamma_{P_1}$, then there exists a pseudo-cycle passing through $S_{X_{\alpha,\beta}}$;
		\item[(e)] if $(\alpha,\beta)\in R^1_5=\{(\al,\be);\gamma_{P_1}(\al,\be)<\be<\gamma_{F}(\al,\be)\}$, then there exists a sliding pseudo-cycle passing through $S_{X_{\alpha,\beta}}$;
		\item[(f)] if $(\alpha,\beta)\in\gamma_{F}$, then there exists a sliding pseudo-polycycle passing through $S_{X_{\alpha,\beta}}$ and $F(Z_{\alpha,\beta})$;
		\item[(g)] if $(\alpha,\beta)\in R^1_4=\{(\al,\be);\gamma_{F}(\al,\be)<\be\mbox{ and } \al<0\}$, then there exists a sliding pseudo-cycle passing through $S_{X_{\alpha,\beta}}$;
		\item[(h)] if $\al=0$ and $\be>0$, then there exists an attracting degenerate cycle passing through $S_{X_{\alpha,\beta}}$;
		\item[(i)] if $(\alpha,\beta)\in R^1_3=\{(\al,\be);\be>0 \mbox{ and } \al>0\}$, then there exists an attracting limit cycle through $\s^c$.
	\end{itemize}
\end{thmA}
	The bifurcation diagram is illustrated in Figure \ref{Table_SBif_diagram11}.
\begin{proof}
	For $\beta>0$ we have a real saddle $S_{X_{\al,\be}}$. Since $Z_{0}$ is in the case $DSC_{11}$, the saddle-regular point of $Z_0$ is in the case $BS_1$, so there exists an attracting pseudo-node, $P_E(Z_{\alpha,\beta})$, satisfying $P_E(Z_{\alpha,\beta})<P_1(Z_{\alpha,\beta})$ in $\s$. Also, $P_1(Z_{\alpha,\beta})<F(Z_{\alpha,\beta})<P_2(Z_{\alpha,\beta})$ and, for $i=1,2,E$, $\lim\limits_{(\al,\be)\rightarrow(0,0)}F(Z_{\alpha,\beta})=\lim\limits_{(\al,\be)\rightarrow(0,0)}P_i(Z_{\alpha,\beta})=0^-$. So, the curves $\gamma_{F_Z},\,\gamma_{P_1}$, and $\gamma_F$ emerge from the origin and lie on $\{(\al,\be);\al<0 \mbox{ and } \be>0\}$. The existence of the limit cycle follows from Proposition \ref{limitcycle-prop}. The result follows by analyzing the dynamics of the system for the possible values of $\pi_{Z_{\al,\be}}(a_{Z_{\al,\be}})$. 
\end{proof} 

\begin{thmB}\label{DSC12_theo1} 
	Suppose that $Z_0$ is of type $DSC_{12}$ and $\beta>0$. Then for a family $Z_{\alpha,\beta}=(X_{\alpha,\beta},Y_{\alpha,\beta})\in\V_0$, bifurcation curves, $\gamma_{P_E},\,\gamma_{P_1},$ and $\gamma_{P_F}$, emerge from the origin, there exists an attracting pseudo-node, and the following statements hold:
	\begin{itemize}
		\item[(a)] if $(\alpha,\beta)\in R^1_7$, where $R^1_7=\{(\al,\be);0<\be<\gamma_{P_E}(\al,\be) \}$, then a repelling limit cycle through $\s^c$ coexists with a sliding polycycle passing through $S_{X_{\alpha,\beta}}$ and $P_E({Z_{\alpha,\beta}})$, which contains two segments of sliding orbits;
		\item[(b)] if $(\alpha,\beta)\in\gamma_{P_E}$, then a repelling limit cycle through $\s^c$ coexists with a sliding polycycle passing through $S_{X_{\alpha,\beta}}$ and $P_E({Z_{\alpha,\beta}})$, which contains only one segment of sliding orbit;
		\item[(c)] if $(\alpha,\beta)\in R^1_6=\{(\al,\be);\gamma_{P_E}(\al,\be)<\be<\gamma_{P_1}(\al,\be)\}$, then a repelling limit cycle through $\s^c$ coexists with a sliding pseudo-cycle passing through $S_{X_{\alpha,\beta}}$;
		\item[(d)] if $(\alpha,\beta)\in\gamma_{P_1}$, then a repelling limit cycle through $\s^c$ coexists with a pseudo-cycle passing through $S_{X_{\alpha,\beta}}$ ;
		\item[(e)] if $(\alpha,\beta)\in R^1_5=\{(\al,\be);\gamma_{P_1}(\al,\be)<\be<\gamma_{F}(\al,\be)\}$, then a repelling limit cycle through $\s^c$ coexists with a sliding pseudo-cycle passing through $S_{X_{\alpha,\beta}}$;
		\item[(f)] if $(\alpha,\beta)\in\gamma_{F}$, then a repelling limit cycle through $\s^c$ coexists with a sliding pseudo-polycycle passing through $S_{X_{\alpha,\beta}}$ and $F(Z_{\alpha,\beta})$;
		\item[(g)] if $(\alpha,\beta)\in R^1_4=\{(\al,\be);\gamma_{F}(\al,\be)<\be\mbox{ and } \al<0\}$, then a repelling limit cycle through $\s^c$ coexists with a sliding pseudo-cycle passing through $S_{X_{\alpha,\beta}}$;
		\item[(h)] if $\al=0$ and $\be>0$, then there exists a repelling degenerate cycle passing through $S_{X_{\alpha,\beta}}$;
		\item[(i)] if $(\alpha,\beta)\in R^1_3=\{(\al,\be);\be>0 \mbox{ and } \al>0\}$, there exists no cycle.
	\end{itemize}
\end{thmB}
	The bifurcation diagram is illustrated in Figure \ref{Table_SBif_diagram12}.
\begin{proof}
	The proof is identical to the proof of Theorem A except that, as seen in Proposition \ref{limitcycle-prop}, limit cycles appear for $\al<0$ and the degenerate cycle for $\al=0$ is a repeller. 
\end{proof}

\begin{figure}[H]
	\centering
	\begin{overpic}[width=15cm]{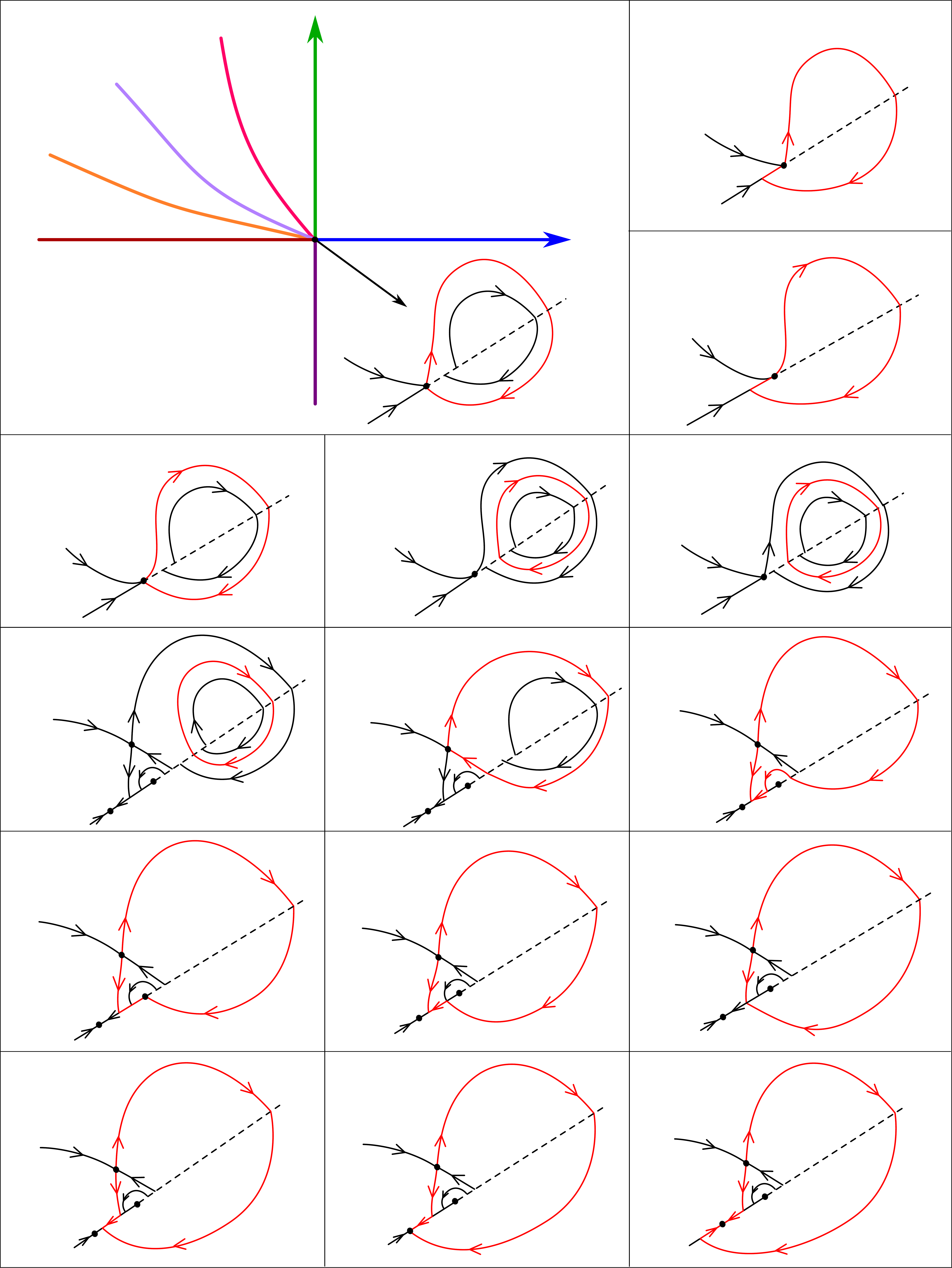}
		\put(23,97){\small $\beta$}\put(43,79){\small$\alpha$}\put(1,88){\small $\gamma_{P_E}$}\put(6.5,93.5){\small $\gamma_{P_1}$}\put(16,98){\small $\gamma_{F}$}
		\put(20,75){\circle{4}}\put(19.3,74.5){\small $R_1^1$}\put(28,75){\circle{4}}\put(27,74.5){\small $R_2^1$}\put(30,88){\circle{4}}\put(29,87.5){\small $R_3^1$}\put(22,91){\circle{4}}\put(21,90.5){\small $R_4^1$}\put(16,90){\circle{4}}\put(15,89.5){\small $R_5^1$}\put(8,89){\circle{4}}\put(7,88.5){\small $R_6^1$}\put(5,84){\circle{4}}\put(4,83.5){\small $R_7^1$}\put(51,98){\textcolor{vinho}{\small{$\beta=0$ and $\alpha<0$}}}\put(51,79){\small$R_1^1$}\put(1,64){\textcolor{roxo}{\small{$\beta<0$ and $\alpha=0$}}}\put(27,63.5){\small $R_2^1$}\put(51,64){\textcolor{azul}{\small{$\beta=0$ and $\alpha>0$}}}\put(1,48){\small $R_3^1$}\put(26,48.5){\textcolor{verde}{\small{$\beta>0$ and $\alpha=0$}}}\put(51,48){\small $R_4^1$}\put(1,32){\textcolor{rosa}{ Curve $\gamma_{F}$}}\put(27,32){\small $R_5^1$}\put(51,32){\textcolor{lilas1}{Curve $\gamma_{P_1}$}}\put(1,15){\small $R_6^1$}\put(26,15){\textcolor{mostarda}{ Curve $\gamma_{P_E}$}}\put(51,15){\small $R_7^1$}
	\end{overpic}
	\caption{Bifurcation diagram of $Z_{\al,\be}$: case $DSC_{11}$.}\label{Table_SBif_diagram11}
\end{figure}

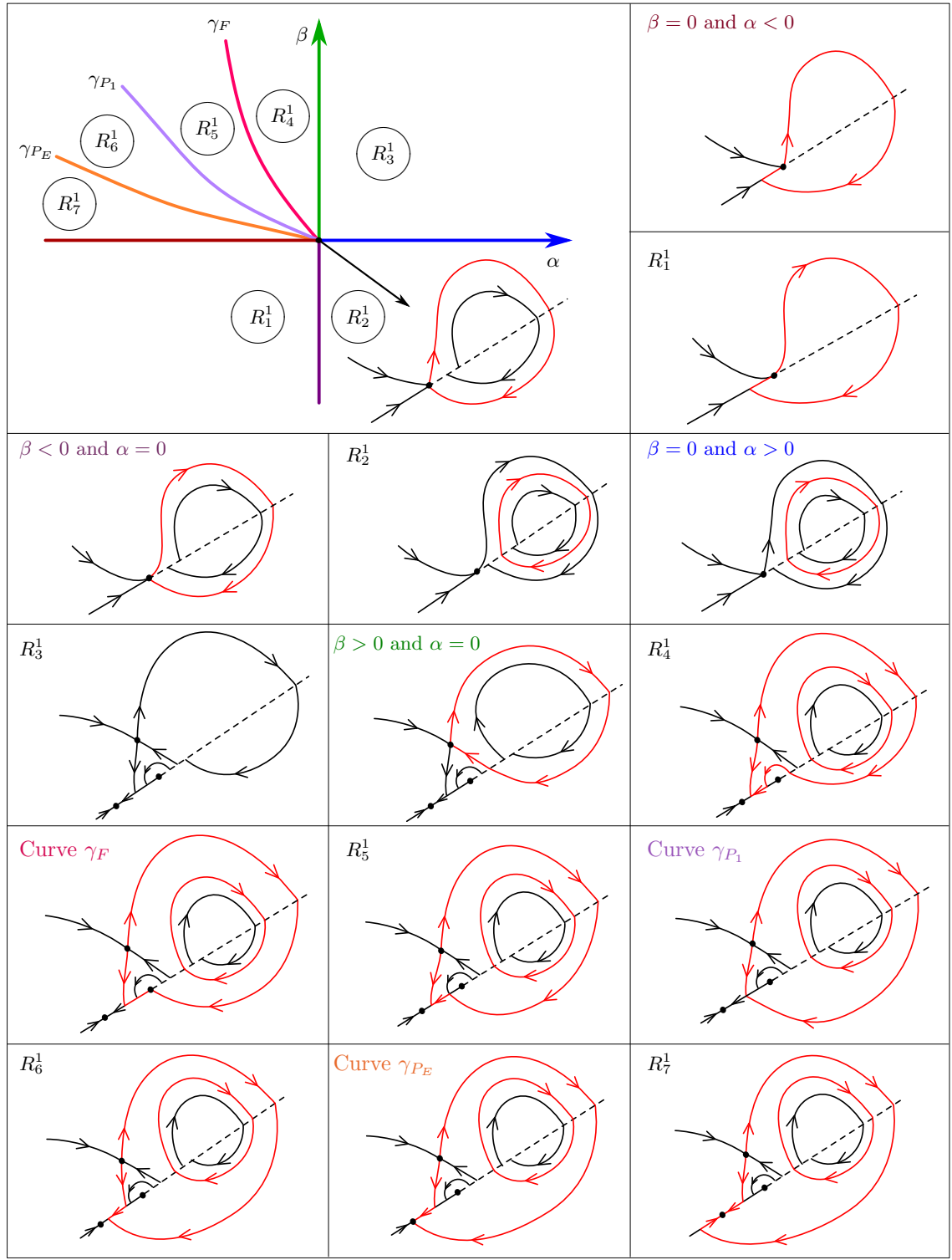
\begin{figure}[H]
	\centering
	\begin{overpic}[width=15cm]{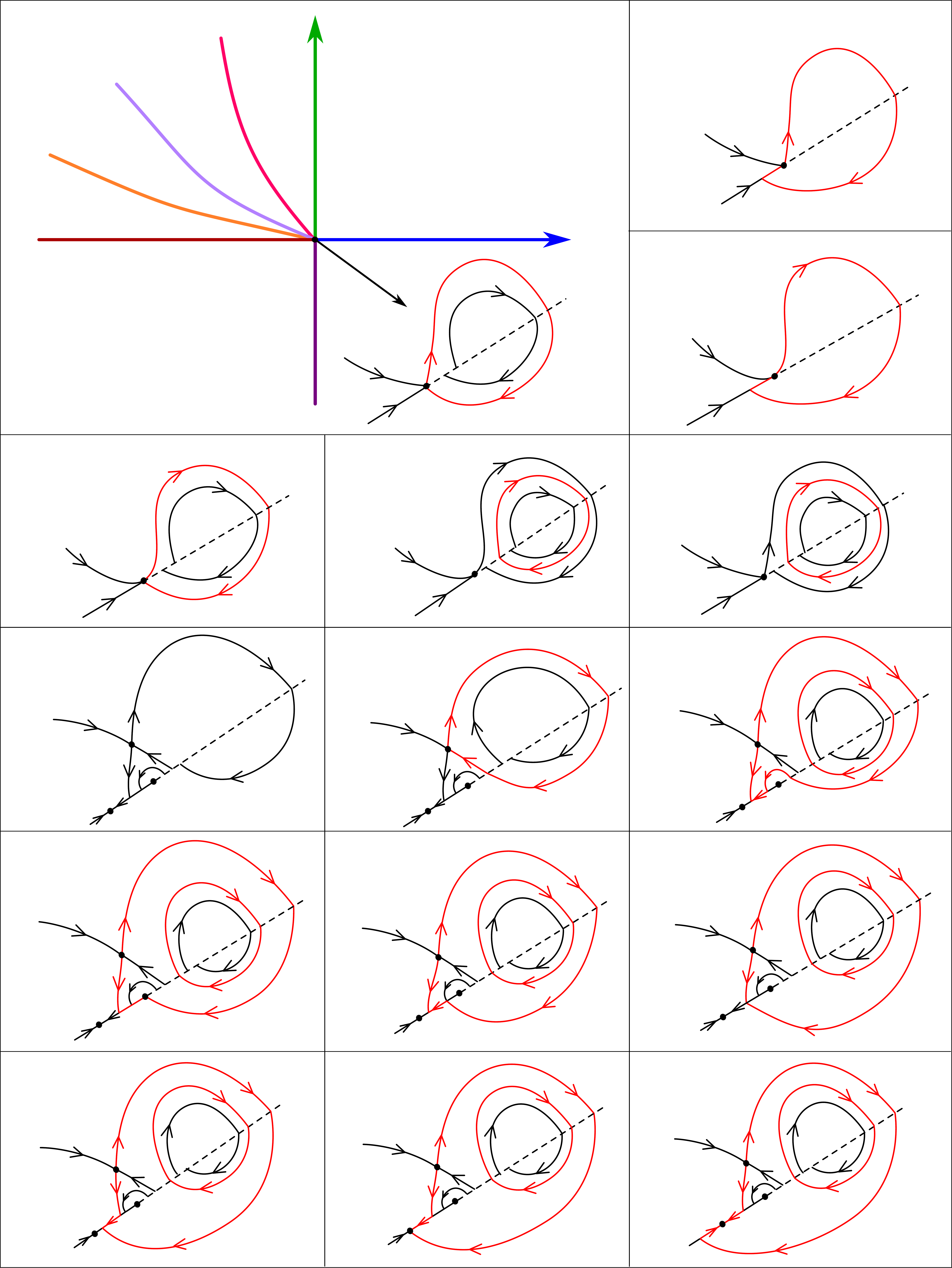}
		\put(23,97){\small $\beta$}\put(43,79){\small$\alpha$}\put(1,88){\small $\gamma_{P_E}$}\put(6.5,93.5){\small $\gamma_{P_1}$}\put(16,98){\small $\gamma_{F}$}
		\put(20,75){\circle{4}}\put(19.3,74.5){\small $R_1^1$}\put(28,75){\circle{4}}\put(27,74.5){\small $R_2^1$}\put(30,88){\circle{4}}\put(29,87.5){\small $R_3^1$}\put(22,91){\circle{4}}\put(21,90.5){\small $R_4^1$}\put(16,90){\circle{4}}\put(15,89.5){\small $R_5^1$}\put(8,89){\circle{4}}\put(7,88.5){\small $R_6^1$}\put(5,84){\circle{4}}\put(4,83.5){\small $R_7^1$}\put(51,98){\textcolor{vinho}{\small{$\beta=0$ and $\alpha<0$}}}\put(51,79){\small$R_1^1$}\put(1,64){\textcolor{roxo}{\small{$\beta<0$ and $\alpha=0$}}}\put(27,63.5){\small $R_2^1$}\put(51,64){\textcolor{azul}{\small{$\beta=0$ and $\alpha>0$}}}\put(1,48){\small $R_3^1$}\put(26,48.5){\textcolor{verde}{\small{$\beta>0$ and $\alpha=0$}}}\put(51,48){\small $R_4^1$}\put(1,32){\textcolor{rosa}{ Curve $\gamma_{F}$}}\put(27,32){\small $R_5^1$}\put(51,32){\textcolor{lilas1}{Curve $\gamma_{P_1}$}}\put(1,15){\small $R_6^1$}\put(26,15){\textcolor{mostarda}{ Curve $\gamma_{P_E}$}}\put(51,15){\small $R_7^1$}
	\end{overpic}
	\caption{Bifurcation diagram of $Z_{\al,\be}$: case $DSC_{12}$.}\label{Table_SBif_diagram12}
\end{figure}
\begin{thmC}\label{DSC11_theo2}
	Suppose that $Z_0$ is of type $DSC_{11}$ or $DSC_{12}$ and $\beta\leq0$. Then for a family $Z_{\alpha,\beta}=(X_{\alpha,\beta},Y_{\alpha,\beta})\in\V_0$ there exists no pseudo-equilibrium, $S_{X_{\alpha,\beta}}$ is as attractor for the sliding vector field, and:
	\begin{itemize}
		\item[(a)] if $\be=0$ and $\al<0$, then there exists a sliding cycle passing through the $S_{X_{\alpha,\beta}}$;
		\item[(b)] if $\alpha=0=\be$, then there exists an attractor cycle passing through $S_{X_{0,0}}$;
		\item[(c)] if $\be=0$ and $\al>0$, then there exists an attracting limit cycle through $\s^c$;
		\item[(d)] if $(\alpha,\beta)\in R^1_1=\{(\al,\be);\be<0 \mbox{ and } \al<0\}$, then there exists a sliding cycle passing through the fold-regular point $F(Z_{\alpha,\beta})$;
		\item[(e)] if $\al=0$ and $\be<0$, then there exists a degenerate cycle passing through the fold-regular point $F(Z_{\alpha,\beta})$;
		\item[(f)] if $(\alpha,\beta)\in R^1_2=\{(\al,\be);\be<0 \mbox{ and } \al>0\}$, then there exists an attracting limit cycle through $\s^c$.
	\end{itemize}
\end{thmC}
	The bifurcation diagrams for $DSC_{11}$ and $DSC_{12}$ are illustrated in Figures \ref{Table_SBif_diagram11} and \ref{Table_SBif_diagram12}, respectively.
\begin{proof}
	If $\beta=0$ then $F(Z_{\alpha,\beta})=S_{X_{\alpha,\beta}}\in\s$. Items $b$ and $c$ follow directly from Proposition \ref{limitcycle-prop}. If $\alpha<0$ then the unstable manifold of the saddle in $\s^+$ intersects $\s$, after that it follows the flow of $Y_{\alpha,0}$ and it intersects the sliding region. Since $F(Z_{\alpha,\beta})$ is an attractor for the sliding vector field (see the bifurcation of a saddle-regular point of type $BS_1$), then there exists a sliding cycle through the saddle-regular point.
	If $\be<0$, then the saddle is virtual and there is no pseudo-equilibrium, the only distinguished singularity is the fold-regular point. The result follows similarly to the proof of Theorem B.	 
\end{proof}

%++++++++++++++++++++++++++++++++++++++++++++++++++++++++++++++++++++++++++++++++++++
%%%%Cycles DSC_21 and DSC_22
%++++++++++++++++++++++++++++++++++++++++++++++++++++++++++++++++++++++++++++++++++++
The following theorems concern cycles of type $DSC_{21}$ and $DSC_{22}$.
\begin{thmD}\label{DSC21_theo1}
	Suppose that $Z_0$ is of type $DSC_{21}$. Then for a family $Z_{\alpha,\beta}=(X_{\alpha,\beta},Y_{\alpha,\beta})\in\V_0$, bifurcation curves, $\gamma_{P_E},\,\tilde{\gamma}_{P_E},\,\gamma_{P_1},$ and $\gamma_{P_F}$, emerge from the origin and the following statements hold:
	\begin{itemize}
		\item[1.] for $\beta\leq0$: identical to the cases given in Theorem Cl
		\item[2.] for $\beta>0$: there exists an attracting pseudo-node and
		\subitem(a) if $(\alpha,\beta)\in R^2_8=\{(\al,\be);0<\be<\gamma_{P_1}(\al,\be) \}$, then there exists a sliding pseudo-cycle passing through $S_{X_{\alpha,\beta}}$; 
		\subitem(b) if $(\alpha,\beta)\in\gamma_{P_1}$ then there exists a pseudo-cycle passing through $S_{X_{\alpha,\beta}}$;
		\subitem(c) if $(\alpha,\beta)\in R^2_7=\{(\al,\be);\gamma_{P_1}(\al,\be)<\be<\gamma_{P_E}(\al,\be)\}$, then there exists a sliding pseudo-cycle passing through $S_{X_{\alpha,\beta}}$;
		\subitem(d) if $(\alpha,\beta)\in\gamma_{P_E}$, then there exists a sliding polycycle passing through $S_{X_{\alpha,\beta}}$ and $P_E({Z_{\alpha,\beta}})$, which contains only one segment of sliding orbits;
		\subitem(e) if $(\alpha,\beta)\in R^2_6=\{(\al,\be);\gamma_{P_E}(\al,\be)<\be<\gamma_{F}(\al,\be)\}$, then there exists a sliding polycycle passing through $S_{X_{\alpha,\beta}}$ and $P_E({Z_{\alpha,\beta}})$, which contains two segments of sliding orbits;
		\subitem(f) if $(\alpha,\beta)\in\gamma_{F}$, then there exists a sliding polycycle passing through $S_{X_{\alpha,\beta}}$, $P_E({Z_{\alpha,\beta}})$ and $F(Z_{\alpha,\beta})$;
		\subitem(g) if $(\alpha,\beta)\in R^2_5=\{(\al,\be);\gamma_{F}(\al,\be)<\be<\tilde{\gamma}_{P_E}(\al,\be)\}$, then there exists a sliding pseudo-cycle passing through $S_{X_{\alpha,\beta}}$;
		\subitem(h) if $(\alpha,\beta)\in\tilde{\gamma}_{P_E}$ then there exists a sliding polycycle passing through $S_{X_{\alpha,\beta}}$ and $Q_{Z_{\alpha,\beta}}$, which contains only one sliding segment;
		\subitem(i) if $(\alpha,\beta)\in R^2_4=\{(\al,\be);\be>\tilde{\gamma}_{P_E}(\al,\be)\mbox{ and } \al<0\}$, then there exists a sliding polycycle passing through $S_{X_{\alpha,\beta}}$ and $P_E(Z_{\alpha,\beta})$, which contains two sliding segments;
		\subitem(j) if $\al=0$ and $\be>0$, then there exists an attracting degenerate cycle through $S_{X_{\alpha,\beta}}$;
		\subitem(k) if $(\alpha,\beta)\in R^2_3=\{(\al,\be);\be>0\mbox{ and } \al>0\}$, then there exists an attracting limit cycle passing through the crossing region near $P_{X_{\alpha,\beta}}$.	
	\end{itemize}
\end{thmD}
	The bifurcation diagram is illustrated in Figure \ref{Table_SBif_diagram21}.
\begin{proof}
	In the case $DSC_{21}$, the saddle-regular point satisfies case $BS_2$. The position of the curves $\gamma_{P_E}$ and $\gamma_{P_1}$ change between cases $DSC_{11}$ and $DSC_{12}$, so that a new bifurcation curve, $\tilde{\gamma}_{P_E}$, emerges from the origin in the half plane where $\be>0$. Since $\al<0$ and $\gamma_F<\be<0$, then $F({Z_{\al,\be}})<\pi_{Z_{\al,\be}}(a_{Z_{\al,\be}})<a_{Z_{\al,\be}}$, the trajectory that contains the unstable manifold of the saddle crosses the crossing region twice before reaching the sliding region from $\s^+$. Therefore, by continuity, there exist values of $\al$ and $\be$ so that this trajectory reaches $\s^s$ at the pseudo-equilibrium point. This new curve provides a connection between $S_{X_{\al,\be}}$ and $P_E(Z_{\al,\be})$.
	The rest of the proof is similar to the proof of Theorem A.
\end{proof}

\begin{thmE}\label{DSC22_theo1}
	Suppose that $Z_0$ is of type $DSC_{22}$. Then for a family $Z_{\alpha,\beta}=(X_{\alpha,\beta},Y_{\alpha,\beta})\in\V_0$, bifurcation curves, $\gamma_{P_E},\,\tilde{\gamma}_{P_E},\,\gamma_{P_1},$ and $\gamma_{P_F}$, emerge from the origin and the following statements hold:
	\begin{itemize}
		\item[1.] for $\beta\leq0$: identical to the cases given in Theorem C.
		\item[2.] for $\beta>0$: there exists a pseudo-node which is an attractor for the sliding vector field and:
		\subitem(a) if $(\alpha,\beta)\in R^2_8=\{(\al,\be);0<\be<\gamma_{P_1}(\al,\be) \}$, then a repelling limit cycle through $\s^c$ coexists with a sliding pseudo-cycle passing through $S_{X_{\alpha,\beta}}$; 
		\subitem(b) if $(\alpha,\beta)\in\gamma_{P_1}$ then a repelling limit cycle through $\s^c$ coexists with a pseudo-cycle passing through $S_{X_{\alpha,\beta}}$;
		\subitem(c) if $(\alpha,\beta)\in R^2_7=\{(\al,\be);\gamma_{P_1}(\al,\be)<\be<\gamma_{P_E}(\al,\be)\}$, then a repelling limit cycle through $\s^c$ coexists with a sliding pseudo-cycle passing through $S_{X_{\alpha,\beta}}$;
		\subitem(d) if $(\alpha,\beta)\in\gamma_{P_E}$, then a repelling limit cycle through $\s^c$ coexists with a sliding polycycle passing through $S_{X_{\alpha,\beta}}$ and $P_E({Z_{\alpha,\beta}})$, which contains only one segment of sliding orbits;
		\subitem(e) if $(\alpha,\beta)\in R^2_6=\{(\al,\be);\gamma_{P_E}(\al,\be)<\be<\gamma_{F}(\al,\be)\}$, then a repelling limit cycle through $\s^c$ coexists with a sliding polycycle passing through $S_{X_{\alpha,\beta}}$ and $P_E({Z_{\alpha,\beta}})$, which contains two segments of sliding orbits;
		\subitem(f) if $(\alpha,\beta)\in\gamma_{F}$, then a repelling limit cycle through $\s^c$ coexists with a sliding polycycle passing through $S_{X_{\alpha,\beta}}$, $P_E({Z_{\alpha,\beta}})$ and $F(Z_{\alpha,\beta})$;
		\subitem(g) if $(\alpha,\beta)\in R^2_5=\{(\al,\be);\gamma_{F}(\al,\be)<\be<\tilde{\gamma}_{P_E}(\al,\be)\}$, then a repelling limit cycle through $\s^c$ coexists with a sliding pseudo-cycle passing through $S_{X_{\alpha,\beta}}$;
		\subitem(h) if $(\alpha,\beta)\in\tilde{\gamma}_{P_E}$, a repelling limit cycle through $\s^c$ coexists with a sliding polycycle passing through $S_{X_{\alpha,\beta}}$ and $Q_{Z_{\alpha,\beta}}$, which contains only one sliding segment;
		\subitem(i) if $(\alpha,\beta)\in R^2_4=\{(\al,\be);\be>\tilde{\gamma}_{P_E}(\al,\be)\mbox{ and } \al<0\}$, then a repelling limit cycle through $\s^c$ coexists with a sliding polycycle passing through $S_{X_{\alpha,\beta}}$ and $P_E(Z_{\alpha,\beta})$, which contains two sliding segments;
		\subitem(j) if $\al=0$ and $\be>0$, then there exists a repelling degenerate cycle through $S_{X_{\alpha,\beta}}$;
		\subitem(k) if $(\alpha,\beta)\in R^2_3=\{(\al,\be);\be>0\mbox{ and } \al>0\}$, then there exist no cycles.	
	\end{itemize}
\end{thmE}
	The bifurcation diagram is illustrated in Figure \ref{Table_SBif_diagram22}.
	
\begin{proof}
	The only difference from the proof of Theorem A is the position of the limit cycles given in Proposition \ref{limitcycle-prop}.
\end{proof}

%++++++++++++++++++++++++++++++++++++++++++++++++++++++++++++++++++++++++++++++++++++
%%%%Cycles DSC_31 and DSC_32
%++++++++++++++++++++++++++++++++++++++++++++++++++++++++++++++++++++++++++++++++++++
The following theorems concern the last two cases, $DSC_{31}$ and $DSC_{32}$. 

\begin{thmF}\label{DSC31_theo1}
	Suppose that $Z_0$ is in the case $DSC_{31}$ and $\beta>0$. Then for a family $Z_{\alpha,\beta}=(X_{\alpha,\beta},Y_{\alpha,\beta})\in\V_0$, bifurcation curves, $\gamma_{P_1}$ and $\gamma_{P_F}$ emerge from the origin, there exists no pseudo-equilibrium, $S_{X_{\alpha,\beta}}$ is a repeller for the sliding vector field, and the following statements hold:
	\begin{itemize}
		\item[(a)] if $(\alpha,\beta)\in R^1_7$, where $R^3_7=\{(\al,\be);0<\be<\gamma_{P_1}(\al,\be) \}$, then there exists a sliding pseudo-cycle through $S_{X_{\alpha,\beta}}$;
		\item[(b)] if $(\alpha,\beta)\in\gamma_{P_1}$, then there exists a pseudo-cycle passing through $S_{X_{\alpha,\beta}}$;
		\item[(c)] if $(\alpha,\beta)\in R^3_6=\{(\al,\be);\gamma_{P_1}(\al,\be)<\be<\gamma_{F}(\al,\be)\}$, then there exists a sliding pseudo-cycle through $S_{X_{\alpha,\beta}}$;
		\item[(d)] if $(\alpha,\beta)\in\gamma_{F}$, then there exists a sliding pseudo-polycycle through $S_{X_{\alpha,\beta}}$ and $F(Z_{\alpha,\beta})$;
		\item[(e)] if $(\alpha,\beta)\in R^3_5=\{(\al,\be);\gamma_{F}(\al,\be)<\be\mbox{ and } \al<0\}$, then there exists a sliding pseudo-cycle through $S_{X_{\alpha,\beta}}$;
		\item[(f)] if $\al=0$ and $\be>0$, then there exists an attracting degenerate cycle through $S_{X_{\alpha,\beta}}$;
		\item[(g)] if $(\al,\be)\in R^3_4=\{(\al,\be);\be>0 \mbox{ and } \al>0\}$, then there exists a limit cycle passing through $\s^s$.
	\end{itemize}
\end{thmF}
	The bifurcation diagram is illustrated in Figure \ref{Table_SBif_diagram31}.
\begin{proof}
	Since $Z_0$ is in case $DSC_{31}$, the saddle-regular point is of type $BS_3$, i.e., the pseudo-equilibrium points appear when the saddle is virtual ($\be<0$). Thus, there is no pseudo-equilibrium for $\be>0$. The position of the limit cycles given in Proposition \ref{limitcycle-prop} implies that they happen for $\al>0$. The rest of the proof is similar to the proof of Theorem A.
\end{proof}

\begin{figure}[H]
	\centering
	\begin{overpic}[width=15cm]{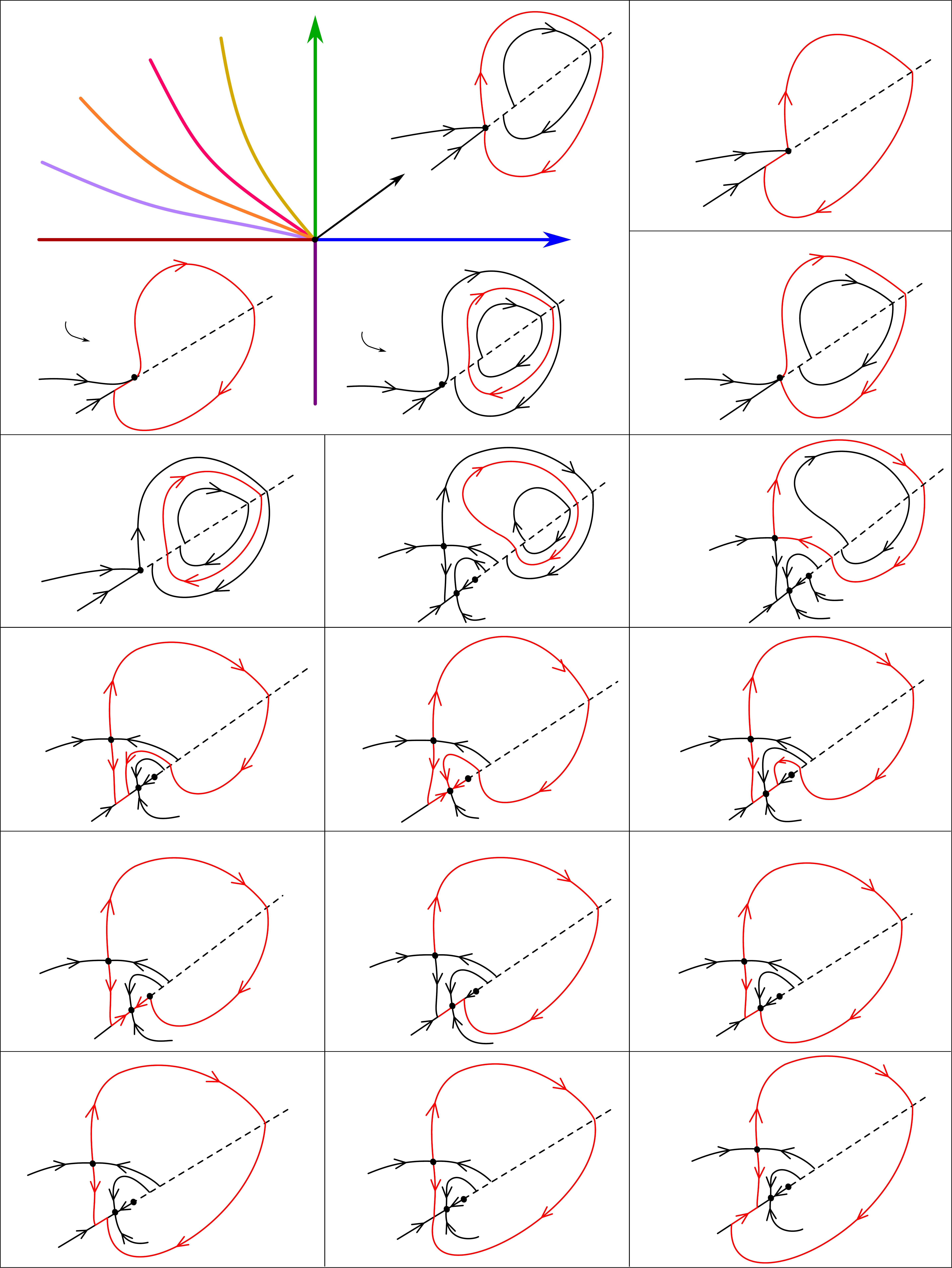}
		\put(23,97){\small $\beta$}\put(43,79){\small$\alpha$}\put(5,93){\small $\gamma_{P_E}$}\put(0.7,87.5){\small $\gamma_{P_1}$}\put(11,96){\small $\gamma_{F}$}\put(16,98){\small $\tilde{\gamma}_{P_E}$}
		\put(5,76.5){\circle{4}}\put(4,76){\small $R_1^2$}\put(28,75.5){\circle{4}}\put(27,75){\small $R_2^2$}\put(30,93){\circle{4}}\put(29,92.5){\small $R_3^2$}\put(22,91){\circle{4}}\put(21,90.5){\small $R_4^2$}\put(15.8,93){\circle{4}}\put(14.8,92.5){\small $R_5^2$}\put(10,92){\circle{4}}\put(9,91.5){\small $R_6^2$}\put(6,89){\circle{4}}\put(5,88.5){\small $R_7^2$}\put(5,84){\circle{4}}\put(4,83.5){\small $R_8^2$}\put(51,98){\textcolor{vinho}{\small{$\beta=0$ and $\alpha<0$}}}\put(51,80){\textcolor{roxo}{\small{$\beta<0$ and $\alpha=0$}}}\put(1,64){\textcolor{azul}{\small{$\beta=0$ and $\alpha>0$}}}\put(27,63.5){\small $R_3^2$}\put(50,64){\textcolor{verde}{\small{$\beta>0$ and $\alpha=0$}}}\put(1,48){\small $R_4^2$}\put(26,48.5){\textcolor{mostarda2}{ Curve $\tilde{\gamma}_{P_E}$}}\put(51,48){\small $R_5^2$}\put(1,32){\textcolor{rosa}{ Curve $\gamma_{F}$}}\put(27,32){\small $R_6^2$}\put(51,32){\textcolor{mostarda}{ Curve $\gamma_{P_E}$}}\put(1,15){\small $R_7^2$}\put(26,15){\textcolor{lilas1}{Curve $\gamma_{P_1}$}}\put(51,15){\small $R_8^2$}
	\end{overpic}
	\caption{Bifurcation diagram of $Z_{\al,\be}$: case $DSC_{21}$.}\label{Table_SBif_diagram21}
\end{figure}

\begin{figure}[H]
	\centering
	\begin{overpic}[width=15cm]{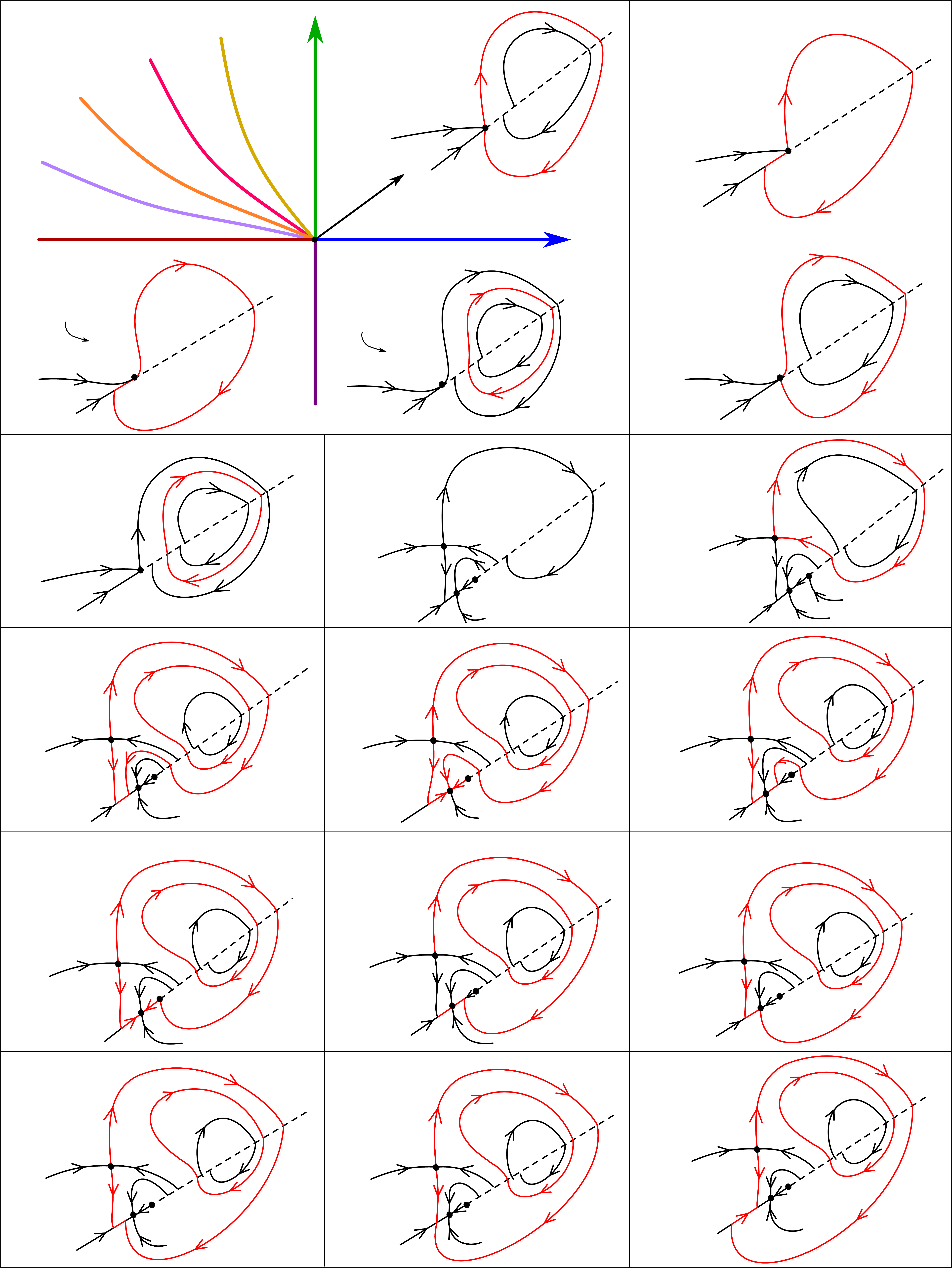}
		\put(23,97){\small $\beta$}\put(43,79){\small$\alpha$}\put(5,93){\small $\gamma_{P_E}$}\put(0.7,87.5){\small $\gamma_{P_1}$}\put(11,96){\small $\gamma_{F}$}\put(16,98){\small $\tilde{\gamma}_{P_E}$}
		\put(5,76.5){\circle{4}}\put(4,76){\small $R_1^2$}\put(28,75.5){\circle{4}}\put(27,75){\small $R_2^2$}\put(30,93){\circle{4}}\put(29,92.5){\small $R_3^2$}\put(22,91){\circle{4}}\put(21,90.5){\small $R_4^2$}\put(15.8,93){\circle{4}}\put(14.8,92.5){\small $R_5^2$}\put(10,92){\circle{4}}\put(9,91.5){\small $R_6^2$}\put(6,89){\circle{4}}\put(5,88.5){\small $R_7^2$}\put(5,84){\circle{4}}\put(4,83.5){\small $R_8^2$}\put(51,98){\textcolor{vinho}{\small{$\beta=0$ and $\alpha<0$}}}\put(51,80){\textcolor{roxo}{\small{$\beta<0$ and $\alpha=0$}}}\put(1,64){\textcolor{azul}{\small{$\beta=0$ and $\alpha>0$}}}\put(27,63.5){\small $R_3^2$}\put(50,64){\textcolor{verde}{\small{$\beta>0$ and $\alpha=0$}}}\put(1,48){\small $R_4^2$}\put(26,48.5){\textcolor{mostarda2}{ Curve $\tilde{\gamma}_{P_E}$}}\put(51,48){\small $R_5^2$}\put(1,32){\textcolor{rosa}{ Curve $\gamma_{F}$}}\put(27,32){\small $R_6^2$}\put(51,32){\textcolor{mostarda}{ Curve $\gamma_{P_E}$}}\put(1,15){\small $R_7^2$}\put(26,15){\textcolor{lilas1}{Curve $\gamma_{P_1}$}}\put(51,15){\small $R_8^2$}
	\end{overpic}
	\caption{Bifurcation diagram of $Z_{\al,\be}$: case $DSC_{22}$.}\label{Table_SBif_diagram22}
\end{figure}

\begin{thmG}\label{DC32_theo31}
	Suppose that $Z_0$ is in the case $DSC_{32}$ and $\beta>0$. Then for a family $Z_{\alpha,\beta}=(X_{\alpha,\beta},Y_{\alpha,\beta})\in\V_0$, bifurcation curves, $\gamma_{P_1}$ and $\gamma_{P_F}$ emerge from the origin, there exists no pseudo-equilibrium, $S_{X_{\alpha,\beta}}$ is a repeller for the sliding vector field, and the following statements hold:
	\begin{itemize}
		\item[(a)] if $(\alpha,\beta)\in R^1_7$, where $R^3_7=\{(\al,\be);0<\be<\gamma_{P_1}(\al,\be) \}$, then a repelling limit cycle through $\s^c$ coexists with a sliding pseudo-cycle through $S_{X_{\alpha,\beta}}$;
		\item[(b)] if $(\alpha,\beta)\in\gamma_{P_1}$, then a repelling limit cycle through $\s^c$ coexists with a pseudo-cycle passing through $S_{X_{\alpha,\beta}}$;
		\item[(c)] if $(\alpha,\beta)\in R^3_6=\{(\al,\be);\gamma_{P_1}(\al,\be)<\be<\gamma_{F}(\al,\be)\}$, then a repelling limit cycle through $\s^c$ coexists with a sliding pseudo-cycle through $S_{X_{\alpha,\beta}}$;
		\item[(d)] if $(\alpha,\beta)\in\gamma_{F}$, then a repelling limit cycle through $\s^c$ coexists with a sliding pseudo-polycycle through $S_{X_{\alpha,\beta}}$ and $F(Z_{\alpha,\beta})$;
		\item[(e)] if $(\alpha,\beta)\in R^3_5=\{(\al,\be);\gamma_{F}(\al,\be)<\be\mbox{ and } \al<0\}$, then a repelling limit cycle through $\s^c$ coexists with a sliding pseudo-cycle through $S_{X_{\alpha,\beta}}$;
		\item[(f)] if $\al=0$ and $\be>0$, then there exists a repelling degenerate cycle through $S_{X_{\alpha,\beta}}$;
		\item[(g)] if $(\alpha,\beta)\in R^3_4=\{(\al,\be);\be>0 \mbox{ and } \al>0\}$, then there exists no cycles passing through $\s^s$.
	\end{itemize}
\end{thmG}
	The bifurcation diagram is illustrated in Figure \ref{Table_SBif_diagram32}.
\begin{proof}
	The position of the limit cycles given in Proposition \ref{limitcycle-prop} implies that they happen for $\al<0$. The rest of the proof is similar to the proof of Theorem F.
\end{proof}

\begin{thmH}\label{DSC31_theo2}
	Suppose that $Z_0$ is in the case $DSC_{31}$ or $DSC_{32}$ and $\beta\leq0$. Then the a $Z_{\alpha,\beta}=(X_{\alpha,\beta},Y_{\alpha,\beta})\in\V_0$ satisfies: there exists no pseudo-equilibrium, $S_{X_{\alpha,\beta}}$ is an attractor for the sliding vector field, and:
	\begin{itemize}
		\item[(a)] if $\be=0$ and $\al<0$, then there exists a sliding cycle passing through the $S_{X_{\alpha,\beta}}$;
		\item[(b)] if $\alpha=0=\be$, then there exists an attracting degenerate cycle passing through $S_{X_{0,0}}$;
		\item[(c)] if $\be=0$ and $\al>0$, then there exists an attracting limit cycle through $\s^c$;
		\item[(d)] if $(\alpha,\beta)\in R^3_1=\{(\al,\be);\gamma_{P_E}(\al,\be)<\be<0\}$, then there exists a sliding cycle passing through the fold-regular point $F(Z_{\alpha,\beta})$;
		\item[(e)] if $(\alpha,\beta)\in\gamma_{P_E}$, there exists a polycycle passing through $F({Z_{\alpha,\beta}})$ and $P_E({Z_{\alpha,\beta}})$;
		\item[(f)] if $(\alpha,\beta)\in R^3_2=\{(\al,\be);\al<0 \mbox{ and } \be<\gamma_{P_E}(\al,\be)\}$, then there exists a sliding cycle passing through $P_{X_{\alpha,\beta}}$;
		\item[(g)] if $\al=0$ and $\be<0$, then an attracting degenerate cycle through $F(Z_{\al,\be})$;
		\item[(h)] if $(\al,\be)\in R^3_3=\{(\al,\be);\be<0 \mbox{ and } \al>0\}$, then there exists an attracting limit cycle through $\s^+$.
	\end{itemize}
\end{thmH}
	The bifurcation diagram for $DSC_{31}$ and $DSC_{32}$ are illustrated in Figures \ref{Table_SBif_diagram31} and \ref{Table_SBif_diagram32}, respectively.
\begin{proof}
	The position of the limit cycles given in Proposition \ref{limitcycle-prop} implies that they happen for $\al>0$. The rest of the proof is similar to the proof of Theorem A. 
\end{proof}
\begin{figure}[H]
	\centering
	\begin{overpic}[width=15cm]{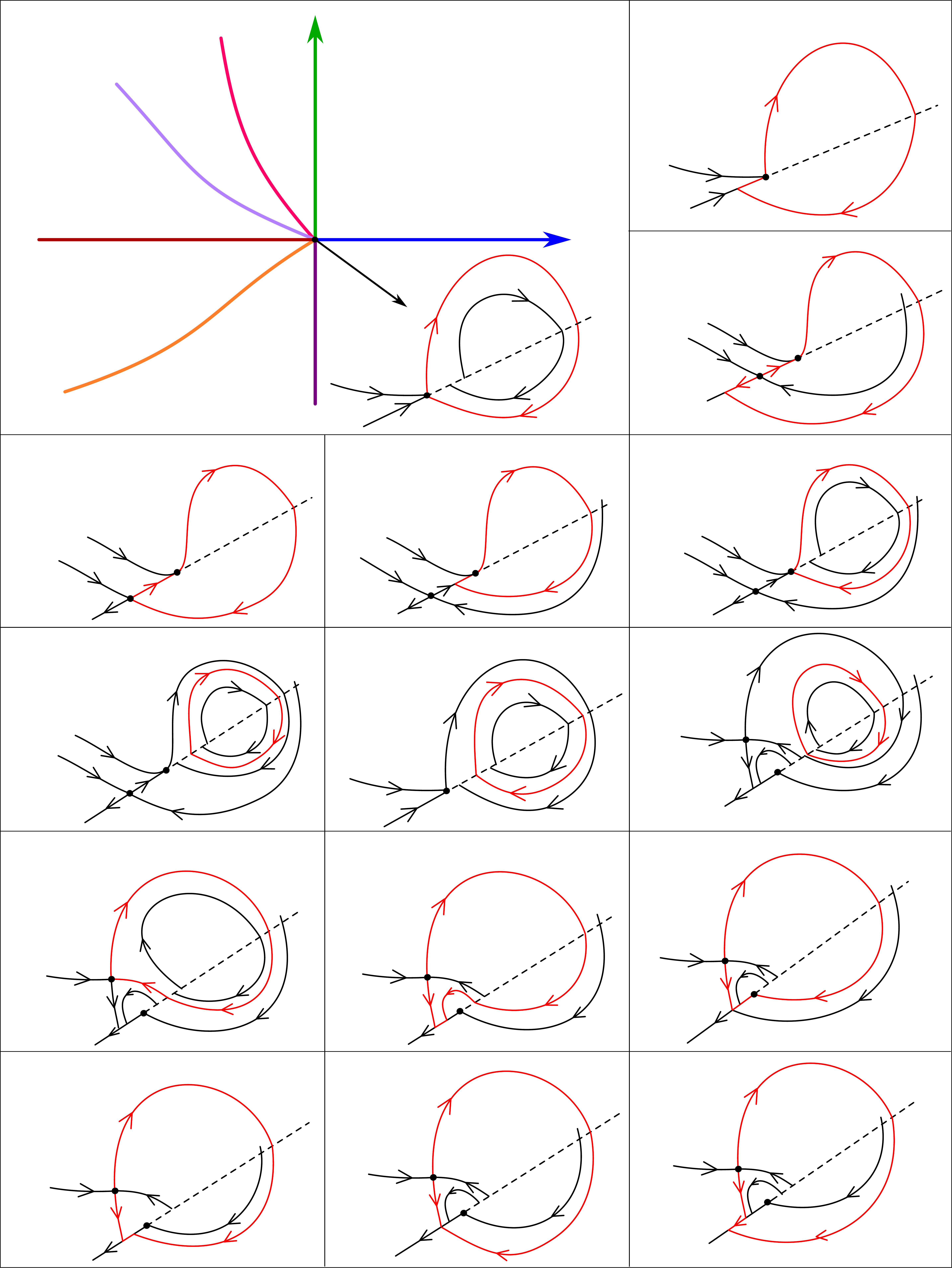}
		\put(23,97){\small $\beta$}\put(43,79){\small$\alpha$}\put(2,69){\small $\gamma_{P_E}$}\put(6.5,93.5){\small $\gamma_{P_1}$}\put(16,98){\small $\gamma_{F}$}
		\put(21,75){\circle{4}}\put(20.3,74.5){\small $R_2^3$}\put(28,75){\circle{4}}\put(27,74.5){\small $R_3^3$}\put(30,88){\circle{4}}\put(29,87.5){\small $R_4^3$}\put(22,91){\circle{4}}\put(21,90.5){\small $R_5^3$}\put(16,90){\circle{4}}\put(15,89.5){\small $R_6^3$}\put(10,86){\circle{4}}\put(9,85.5){\small $R_7^3$}\put(10,77){\circle{4}}\put(9,76.5){\small $R_1^3$}\put(51,98){\textcolor{vinho}{\small{$\beta=0$ and $\alpha<0$}}}\put(51,79){\small$R_1^3$}\put(1,64){\textcolor{mostarda}{ Curve $\gamma_{P_E}$}}\put(27,63.5){\small $R_2^3$}\put(51,64){\textcolor{roxo}{\small{$\beta<0$ and $\alpha=0$}}}\put(1,48){\small $R_3^3$}\put(26,48){\textcolor{azul}{\small{$\beta=0$ and $\alpha>0$}}}\put(51,48){\small $R_4^3$}\put(1,32){\textcolor{verde}{\small{$\beta>0$ and $\alpha=0$}}}\put(51,32){\textcolor{rosa}{ Curve $\gamma_{F}$}}\put(27,32){\small $R_5^3$}\put(26,15){\textcolor{lilas1}{Curve $\gamma_{P_1}$}}\put(1,15){\small $R_6^3$}\put(51,15){\small $R_7^3$}
	\end{overpic}
	\caption{Bifurcation diagram of $Z_{\al,\be}$: case $DSC_{31}$.}\label{Table_SBif_diagram31}
\end{figure}

\begin{figure}[H]
	\centering
	\begin{overpic}[width=15cm]{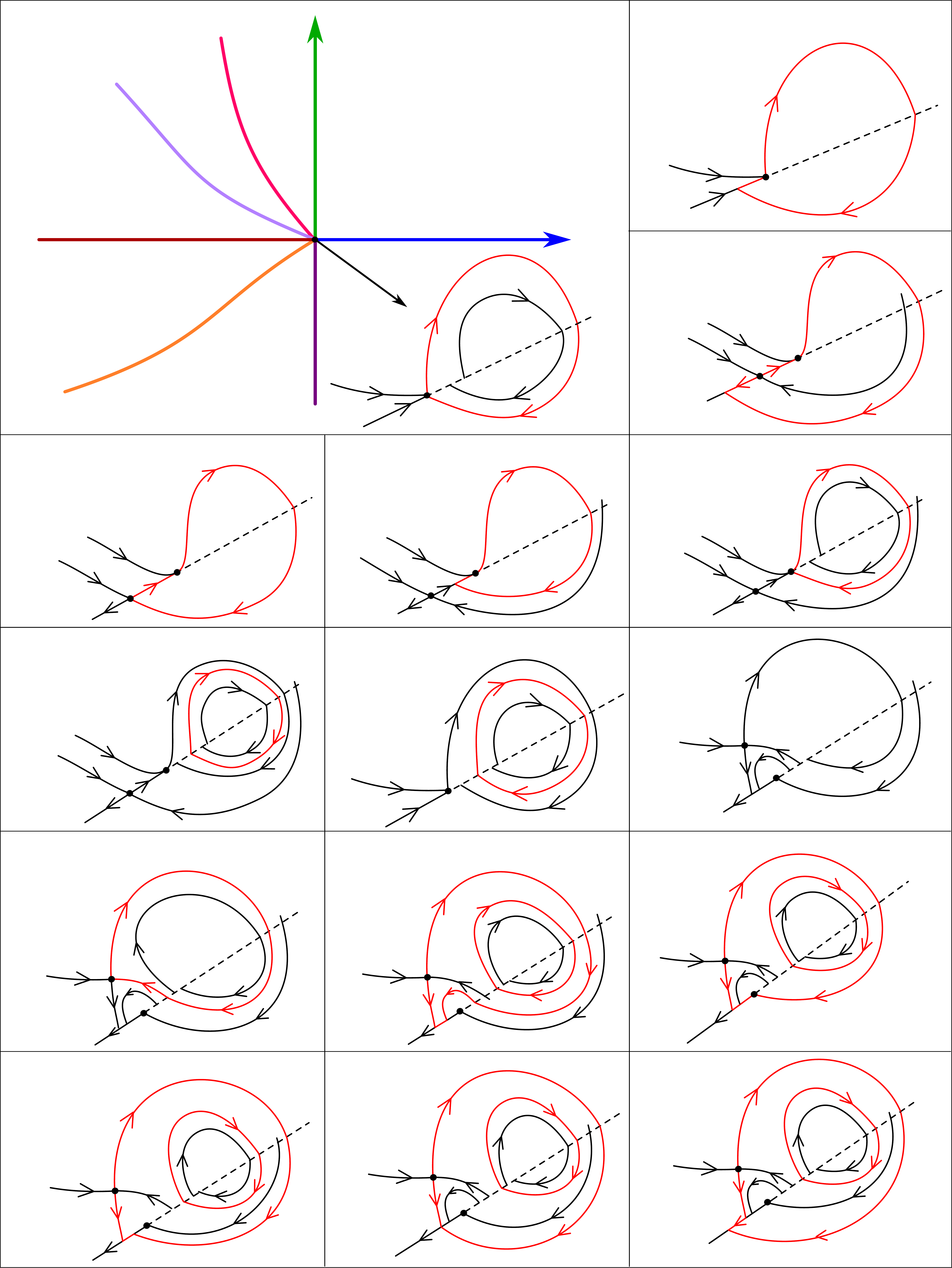}
		\put(23,97){\small $\beta$}\put(43,79){\small$\alpha$}\put(2,69){\small $\gamma_{P_E}$}\put(6.5,93.5){\small $\gamma_{P_1}$}\put(16,98){\small $\gamma_{F}$}
		\put(21,75){\circle{4}}\put(20.3,74.5){\small $R_2^3$}\put(28,75){\circle{4}}\put(27,74.5){\small $R_3^3$}\put(30,88){\circle{4}}\put(29,87.5){\small $R_4^3$}\put(22,91){\circle{4}}\put(21,90.5){\small $R_5^3$}\put(16,90){\circle{4}}\put(15,89.5){\small $R_6^3$}\put(10,86){\circle{4}}\put(9,85.5){\small $R_7^3$}\put(10,77){\circle{4}}\put(9,76.5){\small $R_1^3$}\put(51,98){\textcolor{vinho}{\small{$\beta=0$ and $\alpha<0$}}}\put(51,79){\small$R_1^3$}\put(1,64){\textcolor{mostarda}{ Curve $\gamma_{P_E}$}}\put(27,63.5){\small $R_2^3$}\put(51,64){\textcolor{roxo}{\small{$\beta<0$ and $\alpha=0$}}}\put(1,48){\small $R_3^3$}\put(26,48){\textcolor{azul}{\small{$\beta=0$ and $\alpha>0$}}}\put(51,48){\small $R_4^3$}\put(1,32){\textcolor{verde}{\small{$\beta>0$ and $\alpha=0$}}}\put(51,32){\textcolor{rosa}{ Curve $\gamma_{F}$}}\put(27,32){\small $R_5^3$}\put(26,15){\textcolor{lilas1}{Curve $\gamma_{P_1}$}}\put(1,15){\small $R_6^3$}\put(51,15){\small $R_7^3$}
	\end{overpic}
	\caption{Bifurcation diagram of $Z_{\al,\be}$: case $DSC_{32}$.}\label{Table_SBif_diagram32}
\end{figure}
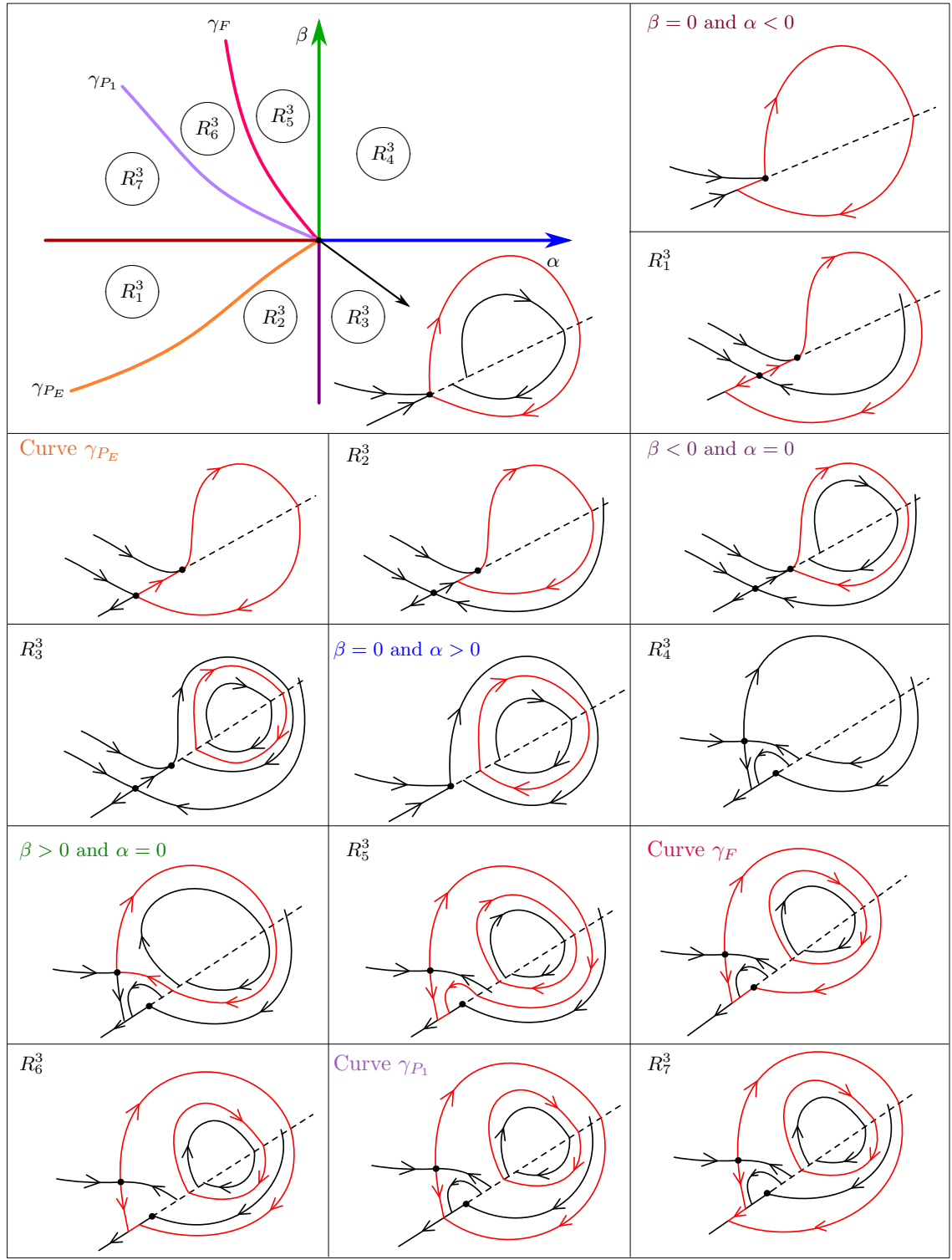

%
%
%
%
%
%Colocar???????????????????????????????????????????????????

\section{The Resonant Case}\label{sec-resonant}

%=================================================================================
%---------------------------------------------------------------------------------
\subsection{Class of Vector Fields with Hyperbolicity Ratio $=1$}
%---------------------------------------------------------------------------------
%================================================================================= 
Let $\mathcal{A}$ be the set of all nonsmooth vector fields $Z=(X,Y)\in\Omega^l$, $l>1$ big enough for our proposes,  where $X$ has a hyperbolic saddle $S_X$ and, in a neighborhood of $S_X$, $X$ is $\mathcal{C}^2$-conjugated to a vector field $W({\textbf x})=A{\textbf x} + {\textbf c}$, where 
\begin{equation}\label{matrixA}
\begin{array}{ccc}
A=\left(\begin{array}{cc}
0 & a \\ b & 0
\end{array}\right)
& \mbox{ and } & {\textbf c}=\left(\begin{array}{c}
-a\tilde{y} \\ -b\tilde{x}
\end{array}\right)
\end{array}
\end{equation}
with $a,b>0$, $S_X=(\tilde{x},\tilde{y})$, and $\s=h^{-1}(0)$ where $h(x,y)=y$. This conjugacy preserves the discontinuity set $\s$. The eigenvalues of $A$ in equation \ref{matrixA} are $\pm\sqrt{ab}$, therefore, for all $Z\in\mathcal{A}$, the hyperbolicity ratio of any $S_X$ is equal to $1$. 

By considering vector fields in $\V_{Z_0}\cap\mathcal{A}$ we now perform an analysis of the first return map. We proceed similarly to the analysis performed for the case $r\notin\Q$. The main difference is that we consider $\s$ fixed and the saddle point variable. Then for each $Z\in\V_{Z_0}\cap\mathcal{A}$, there is no loss of generality in assuming $a_Z=0$ (this is possible up to translation maps) and the image of $a_Z$ by conjugacy is also $0$ (this is possible because the conjugacy preserves $\s$).

Let $\tau\subset\s^+$ be a transversal section to the flow of $X$ such that, if $S_X\in\s^+$, then $\tau$ is above $S_X$. In this case, the transition map of $X$ near $S_X$ is $\rho_1:[0,\delta_Z)\rightarrow\sigma$ with $\rho_1(x)$ being the projection on the first coordinate of the point where the trajectory of $X$ passing through $(x,0)\in[0,\delta_Z)\times\{0\}$ meets $\tau$ at the first (positive) time. Let $\tau_W\subset\{(x,\varepsilon)\in\R^2\}$ be the transversal section to the flow of $W$ for some $\varepsilon>0$, such that this section is contained in the neighborhood where $W$ and $X$ are conjugated. Without loss of generality we assume that $\tau$ is the image by conjugacy of $\sigma_W$.

\begin{lemma}\label{lemma-rhoW}
	The transition map $\rho_W:[0,\delta)\rightarrow\sigma_W$ can be differentially extended to an open neighborhood of $0$.
\end{lemma}
\begin{proof}
	The trajectories of $W$ lie in the level curves of the function $G(x,y)=bx^2-ay^2+2c_1x-2c_2y$, where $c_1=-b\tilde{x}$ and $c_2=-a\tilde{y}$.
	Assume that $\rho_W$ is defined for $0\leq x<\delta$. For each $(x,0)\in\s$ with $0\leq x <\delta$ we obtain 
	\begin{equation}\label{map_rhoW}
	\rho_W(x)=-\dfrac{c_1}{b}+\sqrt{x^2+\dfrac{2c_1}{b}x+\dfrac{a}{b}\varepsilon^2+\dfrac{2c_2}{b}\varepsilon+\dfrac{c_1^2}{b^2}}.
	\end{equation}
	The $x$-independent part of the expression inside the square root, $Q(\varepsilon)=\dfrac{a}{b}\varepsilon^2+\dfrac{2c_2}{b}\varepsilon+\dfrac{c_1^2}{b^2}$, is a polynomial of degree $2$ in $\varepsilon$. We analyze this polynomial in three different cases:
	\begin{itemize}
		\item[(i)] If $\tilde{y}<0$ we have $c_2>0$. As $P_X=(0,0)$ we must have $c_1=0$ and since $a,b>0$ it follows that $Q(\varepsilon)>0$ for $\varepsilon>0$.
		\item[(ii)] If $\tilde{y}=0$ then $S_X=(0,0)$, consequently $c_1=c_2=0$ and $Q(\varepsilon)>0$ for $\varepsilon>0$.
		\item[(iii)] If $\tilde{y}>0$, then $c_2<0$. Imposing the condition that the stable manifold of $S_X$ intersects $\s$ in $(0,0)$ we obtain $c_1=-c_2\sqrt{b/a}>0$. It implies that $Q$ has only one root, which is $-c_2/a=\tilde{y}$ and $Q(\varepsilon)>0$ for all $\varepsilon>0$ and $\varepsilon\neq \tilde{y}$. However, for our proposes, $\varepsilon>\tilde{y}$. Thus, $Q(\varepsilon)>0$ for $\varepsilon>\tilde{y}$.
	\end{itemize}
	Therefore, choosing $\varepsilon>\max\{0,\tilde{y}\}$ we obtain $Q(\varepsilon)>0$, and we can find $\delta_W>0$ such that the expression inside the square root in \eqref{map_rhoW} is positive for $x\in(-\delta_W,\delta_W)$. It follows that $\rho_W$ can be differentially extended to an open neighborhood of $0$. 
\end{proof}
Since $W$ is $\mathcal{C}^2$-conjugated to $X$ in a neighborhood of $S_X$ there exists a diffeomorphism of class $\mathcal{C}^2$-, $\psi$, defined in a neighborhood of $S_X$, such that $\psi(S_W)=S_X$ and
\begin{equation}\label{map_rho1}
\rho_1(x)=\psi\circ\rho_W\circ\psi^{-1}(x).
\end{equation}
It follows from this expression that $\rho_1$ can be at least $\mathcal{C}^2$-extended to an open neighborhood of $0$, allowing us to calculate the Taylor series of $\pi_Z$ at $x=0$ to order $2$.

\begin{prop}\label{first_return_prop}
	Consider $Z=(X,Y)\in\mathcal{A}\cap\mathcal{V}_{Z_0}$ and let $S_X=(\tilde{x},\tilde{y})$ be the saddle point associated with $X$. Suppose that the first return map is defined in a half-open interval $[0,\delta_Z)$, for some $\delta_Z>0$. Then the first return map of $Z$ can be written as
	\begin{equation}\label{eq_first_return_1}
	\pi_Z(x)=\alpha+ k_1(\beta) x+ k_2(\beta) x^2+\mbox{ h.o.t.},
	\end{equation}
	where $0<|k_1(\beta)|<1$ if $\beta>0$ and $k_1(\beta)=0$, $k_2(\beta)>0$ if $\beta\leq0$.
\end{prop}

\begin{proof}
	The proof follows directly from Lemma \ref{lemma-rhoW} by calculating the derivatives of $\pi_{Z}(x)=\rho_3\circ\rho_2\circ\rho_1(x)$.
\end{proof}

It follows from this result that for $\bar{Z}$ we have $\pi_{\bar{Z}}(0)=0$, $\frac{d\pi_{\bar{Z}}}{dx}(0)=0$ and $\frac{d^2\pi_{\bar{Z}}}{dx^2}(0)>0$. Let $\Gamma$ be the degenerate cycle of $\bar{Z}$. Since the first return map is only defined in a half-open interval, there exists a neighborhood of $\Gamma$ such that the orbits of $\bar{Z}$ in this neighborhood, through points where the first return map is defined, having $\Gamma$ as $\omega$-limit set. 
Another consequence of this proposition, which has a similar proof as to that of Proposition \ref{limitcycle-prop}, is the following.
\begin{corollary}\label{corol_cycle}
	Under the hypotheses in the Proposition \ref{first_return_prop}, if $Z\in\mathcal{V}_{Z_0}\cap\mathcal{A}$ is such that $\alpha>0$ then, for some $x>0$ sufficiently small, $Z$ has a stable limit cycle passing through $(x,0)\in\s$. If $Z$ is such that $\alpha=0$, then there exists a degenerate cycle $\Gamma$ having a neighborhood for which the orbits of $Z$ in this neighborhood passing through points where the first return map is defined have $\Gamma$ as $\omega$-limit set.
\end{corollary}

Therefore, for $Z\in\mathcal{V}_{Z_0}\cap\mathcal{A}$ the first return $\pi_Z$ has the graph equal to the graph for nonresonant saddles with hyperbolicity ratio smaller than $1$. Therefore, depending on the structure of the local saddle-regular point, the bifurcation diagram in this case is given in Figure \ref{Table_SBif_diagram11}, or \ref{Table_SBif_diagram21}, or \ref{Table_SBif_diagram31}.

%%%=================================================================================
%%%---------------------------------------------------------------------------------
\subsection{Model with Hyperbolicity Ratio in $\Q$}
%%%---------------------------------------------------------------------------------
%%%=================================================================================

Now we consider a model with hyperbolicity ratio $r>0$, and study how the system evolves for some values of $r\in\Q$. Consider $Z_a=(X_a,Y_a)$ with $a=(r,k,d,m)$, $\s=h^{-1}(0)$, $h(x,y)=y+x/4-m$, and 
\begin{equation}\label{model_2-eq}
\begin{array}{ccc}
X_a=\left(\begin{array}{c}
x \\ ry-x^3-kx
\end{array}\right)& \mbox{ and }& Y_a=\left(\begin{array}{c}
-1 \\ -x+d
\end{array}\right).
\end{array}
\end{equation}

Observe that:
\begin{itemize}
	\item[-] For $r>0$, $S_{X_a}=S=(0,0)$ is a hyperbolic saddle of $X_a$ with hyperbolicity ratio $r$.
	\item[-] $W^s(S,X_a)=\{(x,y)\in\R;x=0 \}$ and\\$W^s(S,X)=\left\{(x,y)\in\R^2; y=-\frac{x^3}{r+3}-\frac{kx}{r+1}\right\}$.
	\item[-] $W^s(S,X_a)\cap\s=P_2$ and $W^u(S,X_a)\cap\s=\{P_4,P_1,P_3 \}$, where $P_j=P_j(r,k,m)=(x_j,-x_j/4+m)$, $j=1,3,4$, $P_2=(0,m)$ and $x_4<x_1<x_3$. If $m>0$ then $x_1>0$, $x_1<0$ for $m<0$, and $x_1=0$ if $m=0$. Also, for $m=0$, $x_4=-\sqrt{\frac{(r+1-4k)(r+3)}{4(r+1)}}$ and $x_3=\sqrt{\frac{(r+1-4k)(r+3)}{4(r+1)}}$. Then for $k<0$, $W^u(S,X_a)$ crosses $\s$ in three different points if $m\approx0$.
	\item[-] $F_{Z_a}=(x_F,-x_F/4+m)$ is the fold point of $X_a$ near $S$ for $m\neq0$. $F_{Z_a}$ is visible if $m>0$ and invisible if $m<0$.
	\item[-] $F_d=\left(d-\frac{1}{4},-\frac{d}{4}+\frac{1}{16}-m\right)$ is the unique fold point of $Y_a$, which is invisible.
	\item[-] Given $p_0=\left(x_0,-x_0/4+m\right)\in\s$, the trajectory of $Y_a$ through $p_0$ meets $\s$ again at the point $\left(2d-1/2-x_0, (x_0-2d)/4+1/8+m\right) $. In this case, consider $\rho_3(x_0)=2d-1/2-x_0$.
	\item[-] For $k<0$, $m=0$, and $d>0$ $S$ is a saddle-regular point of type $BS_3$.
	\item[-] By considering $\s$ parametrized by the first coordinate ($x\mapsto (x,-x/4+m)$ ), the sliding vector field is\\$Z^s(x)=-\dfrac{4x^3+4x^2+(4k-4d-r)+4mr}{4x^3-(5-4k+r)x+4mr+4d-1}.$
\end{itemize}
For any two points $p_1$ and $p_2$ in $\s$, we can choose the parameter $d$ so that they lie on the same trajectory of $Y_a$. Phase portraits of $X_a$ and $Y_a$ for $m=0$, $r>0$, and $k<0$ are given in Figure \ref{phaseportrait-HRmodel}. 
\begin{figure}[H]
	\centering
	\begin{overpic}[width=11cm]{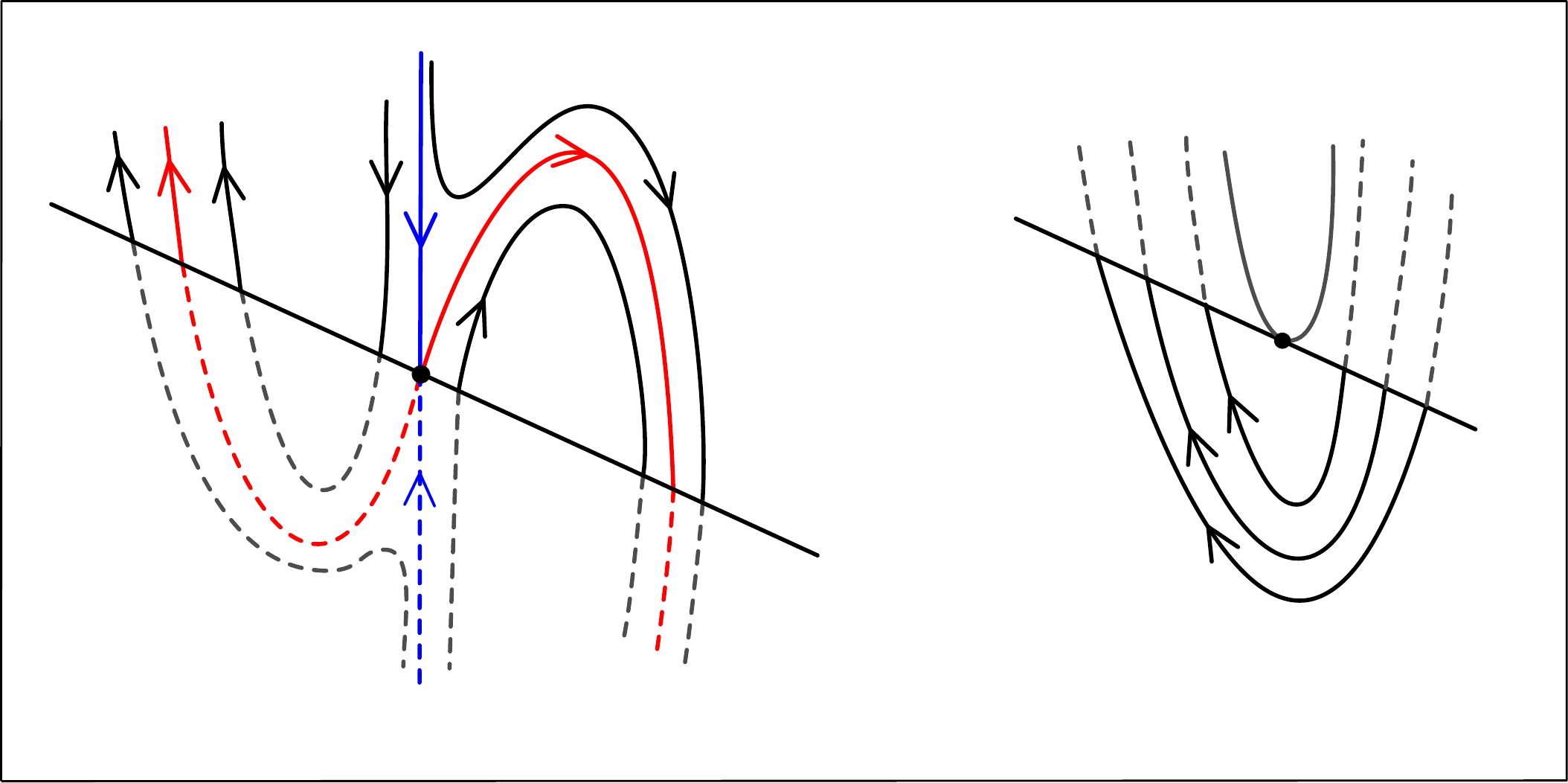}
		\put(27,2){$(a)$}\put(81,2){ $(b)$}\put(51,11.5){\small $\s$}\put(93,19){\small $\s$}
	\end{overpic}
	\caption{Phase portraits of $(a)$ $X_a$ and $(b)$ $Y_a$ with $m=0$, $r>0$, and $k<0$.}\label{phaseportrait-HRmodel}
\end{figure}
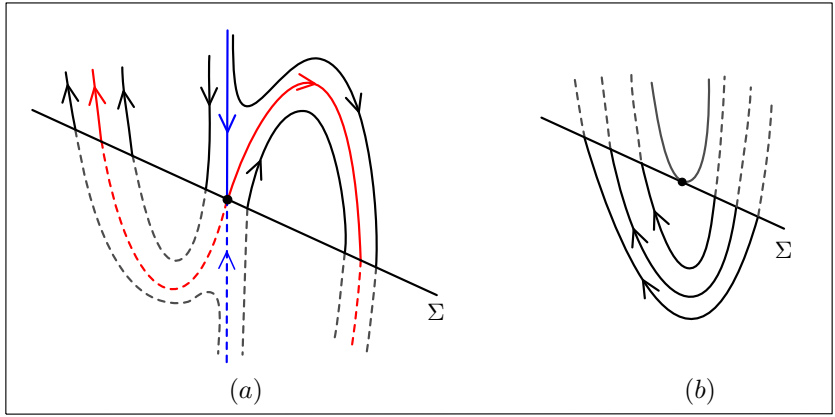

In what follows, assume $r>0$, $k<0$, and $m\approx0$ sufficiently small such that $ W^u(S,X_a)\cap\s=3$. Fix $r>0$ and $k,0$, for each $m$, by varying $d$, we obtain all configurations given by the curves $\gamma_F$, $\gamma_{P_1}$, $\gamma_{P_2}$, and $\gamma_{P_E}$, depending on the signal of $m$. By changing $m$ and $d$ we want to show that $Z_a$ realizes the bifurcation diagram $DSC_{31}$ if $r>1$ or $DSC_{32}$ if $r<1$. To do so it is enough to show the existence of the limit cycles given in that bifurcation diagrams, and that the pseudo-equilibria appear when $m>0$, i.e., when the saddle is virtual. Since a full account of the algebraic cases is not feasible, we illustrate them numerically.

Graphs of the first return map for some values of the parameters were obtained numerically. All graphs are given with the variable $x$ having an initial point at the corresponding $a_Z$ for which the first return is defined. See Figures \ref{firstreturn-ModelRH-1_2---30}, \ref{firstreturn-ModelRH-1_2---31}, and \ref{firstreturn-ModelRH-1_2---32} for some values of $r\in\Q$ satisfying $0<r<1$ and see Figures \ref{firstreturn-ModelRH-3_2---30}, \ref{firstreturn-ModelRH-3_2---31}, and \ref{firstreturn-ModelRH-3_2---32} for some values of $r\in\Q$ with $r>1$. For each $m$ fixed, the first return map varies continuously as the other parameters vary, implying that the results for nonresonant saddles remains true for this model.
By performing a numerical analysis of the pseudo-equilibrium, we have that the bifurcations of the degenerate cycles obey bifurcation diagrams given in Figure \ref{Table_SBif_diagram31} if $r>1$ and Figure \ref{Table_SBif_diagram32} if $r<1$.

\begin{figure}[H]
		\centering \begin{overpic}[width=10cm]{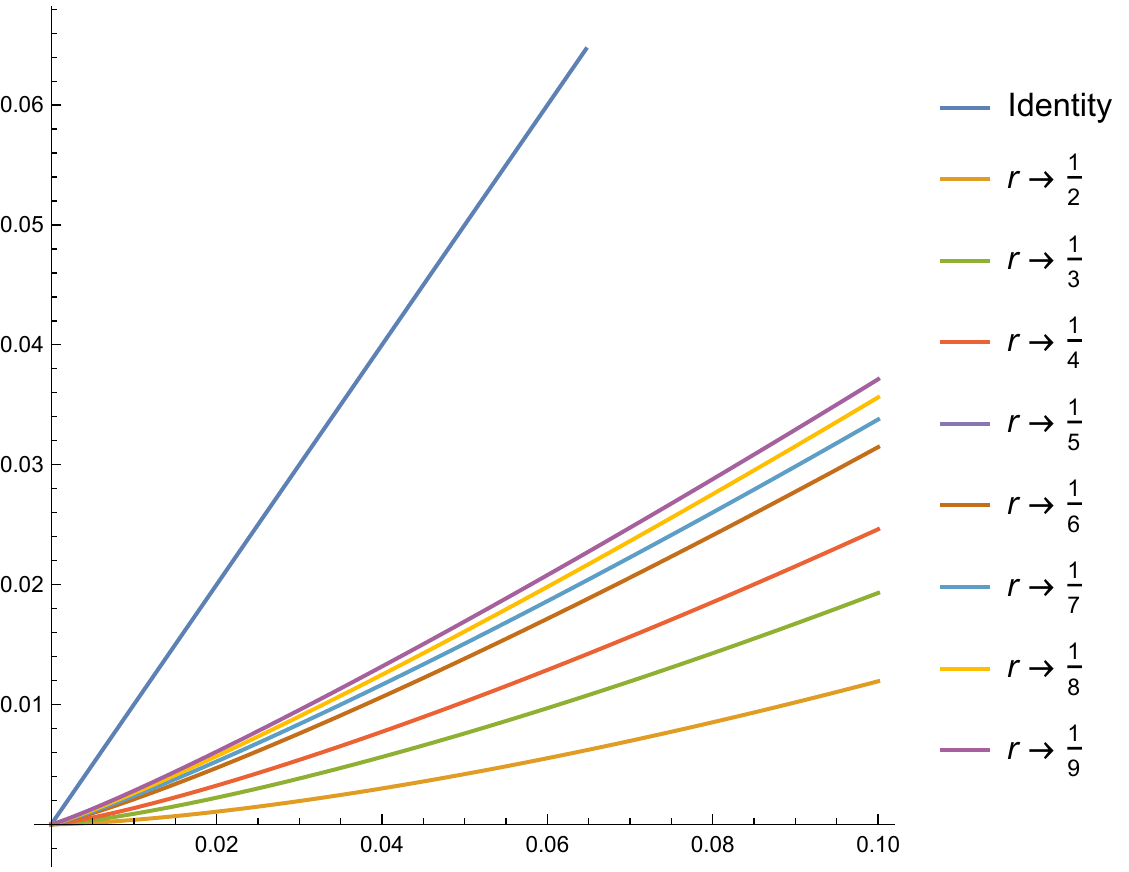}
		\end{overpic}
		\caption{First return map: $m=0$, $k=-1$, and $r<1$.}\label{firstreturn-ModelRH-1_2---30}  
\end{figure}

\begin{figure}[H]
		\centering	\begin{overpic}[width=10cm]{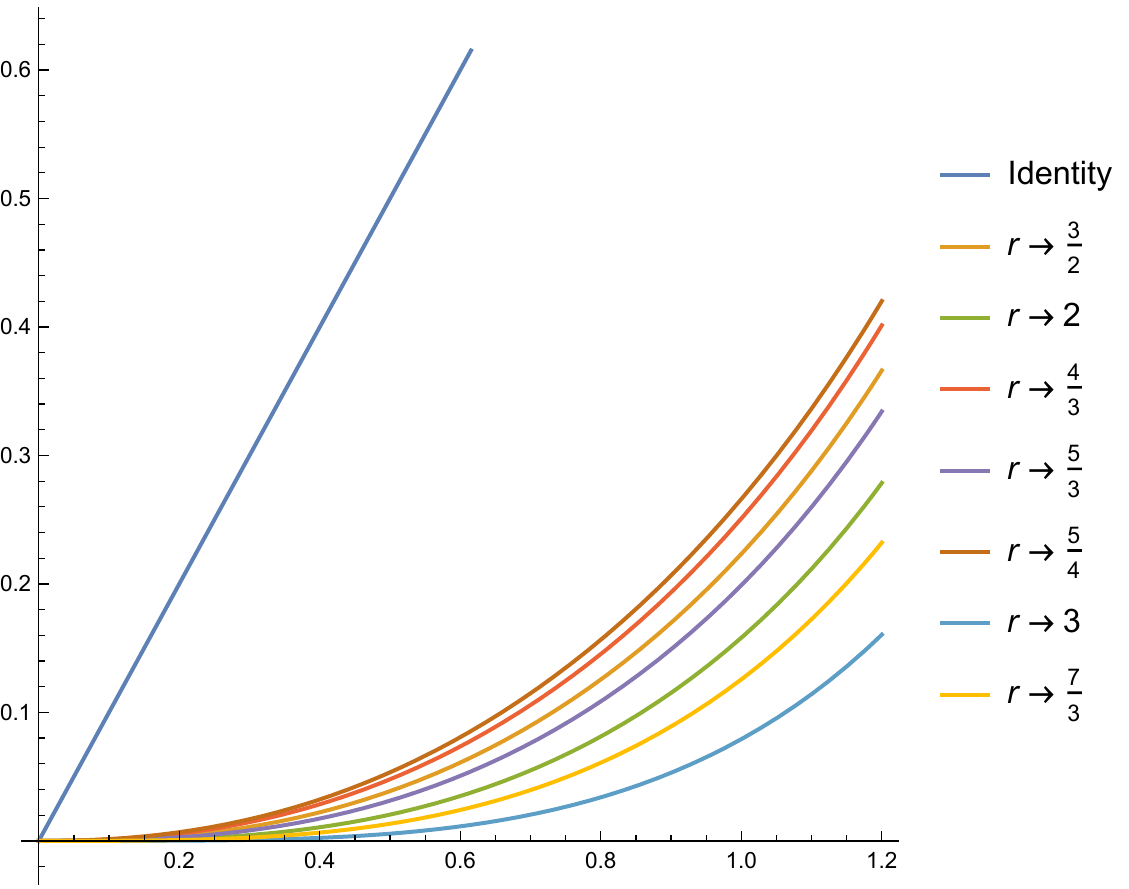}
		\end{overpic}
		\caption{First return map: $m=0$, $k=-1$, and $r>1$.}\label{firstreturn-ModelRH-3_2---30}
\end{figure}\vspace{-2cm}

\begin{figure}[H]
		\centering\begin{overpic}[width=10cm]{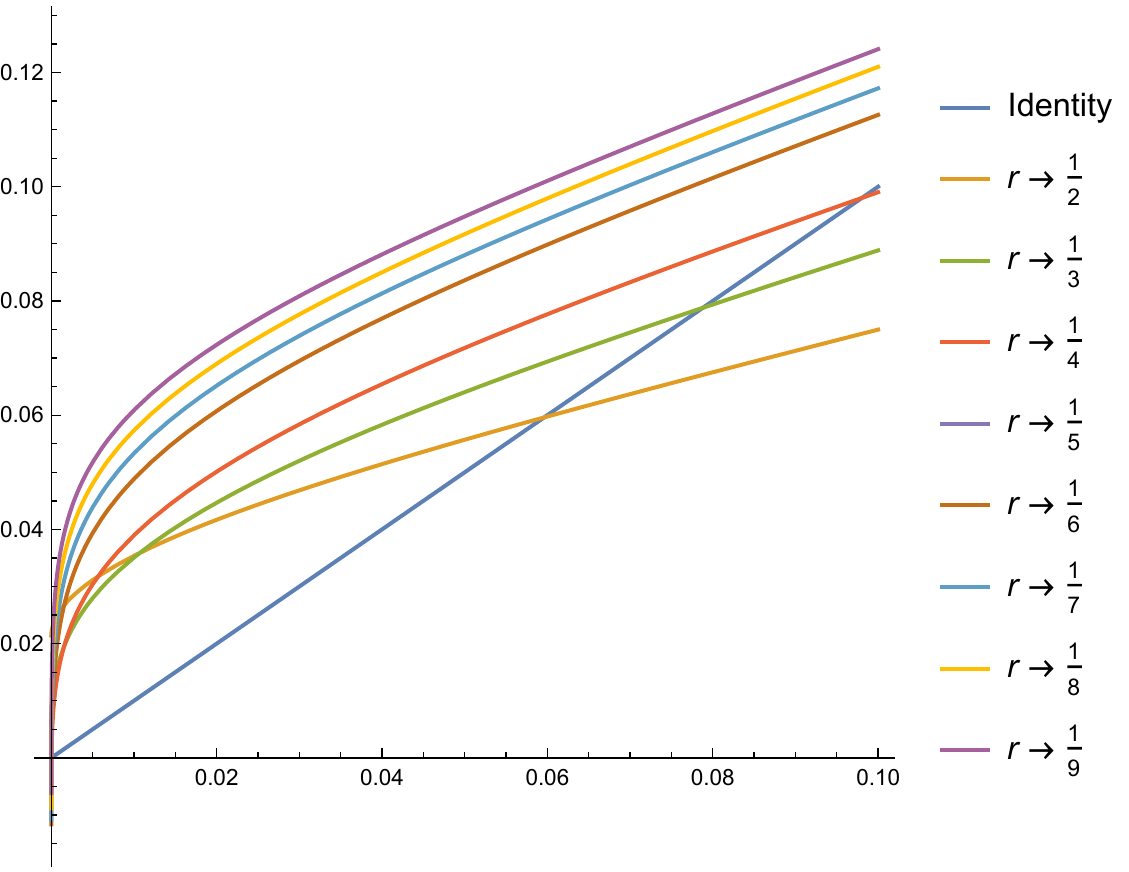}
		\end{overpic}
		\caption{First return map: $m=-0.5$, $k=-1$, $d=1.27$, and $r<1$.}\label{firstreturn-ModelRH-1_2---31}
\end{figure}

\begin{figure}[H]
		\centering	\begin{overpic}[width=10cm]{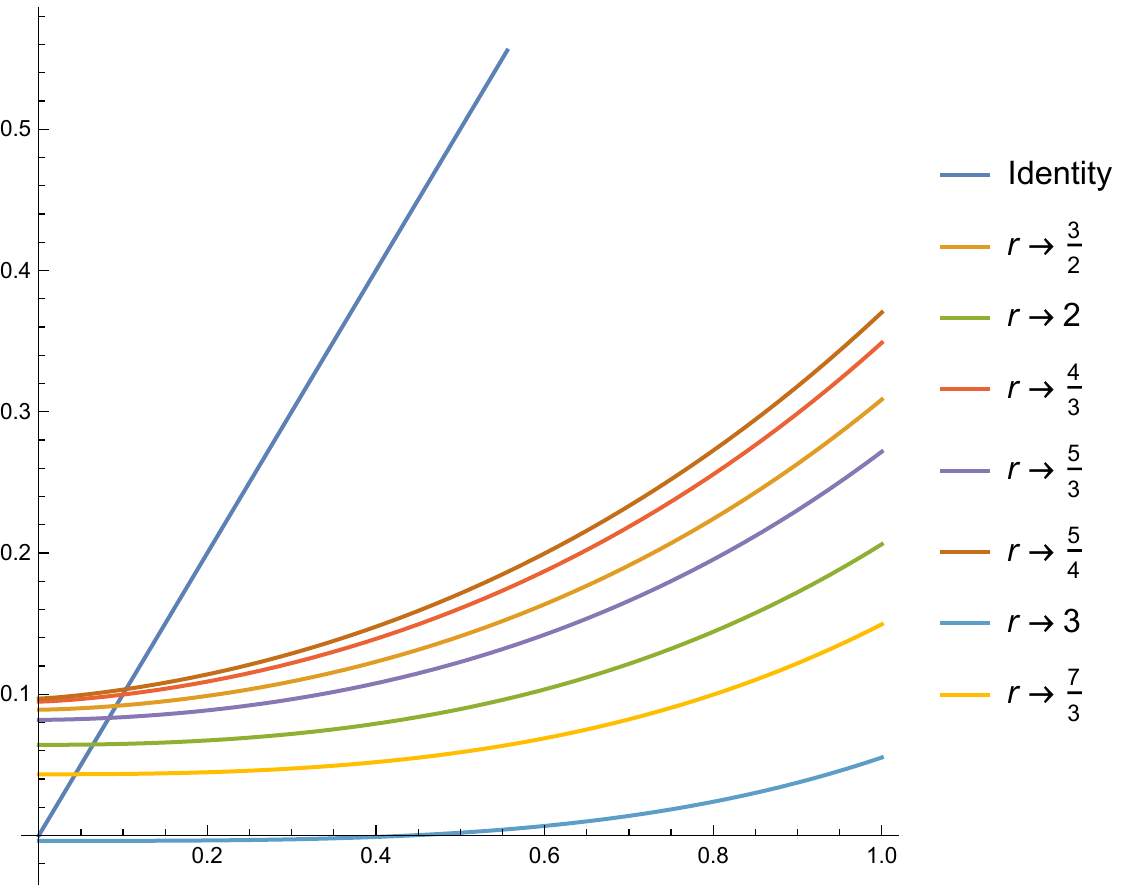}
		\end{overpic}
		\caption{First return map: $m=-0.5$, $k=-1$, $d=1.3$, and $r>1$.}\label{firstreturn-ModelRH-3_2---31}
\end{figure}

\begin{figure}[H]
		\centering	\begin{overpic}[width=10cm]{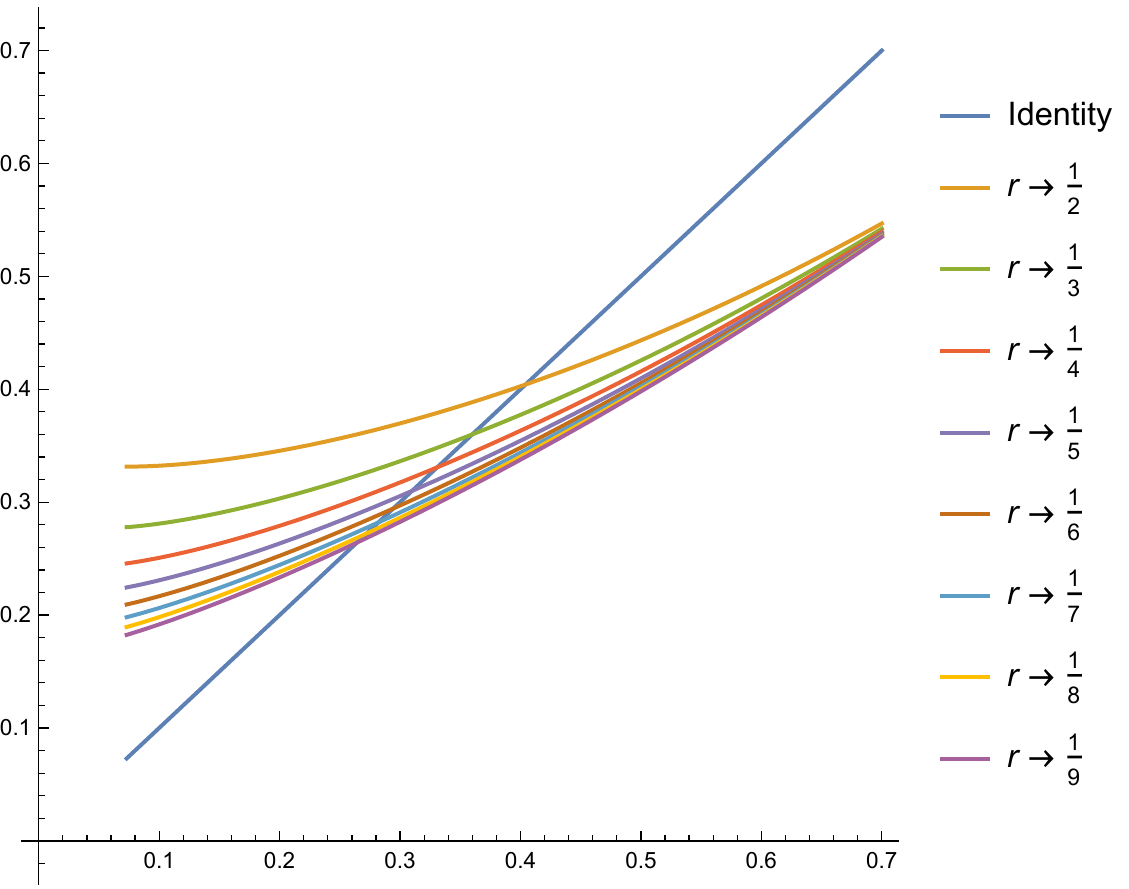}
		\end{overpic}
		\caption{First return map: $m=0.2$, $k=-1$, $d=1.26$, and of $r<1$.}\label{firstreturn-ModelRH-1_2---32}
\end{figure}

\begin{figure}[H]
		\centering	\begin{overpic}[width=10cm]{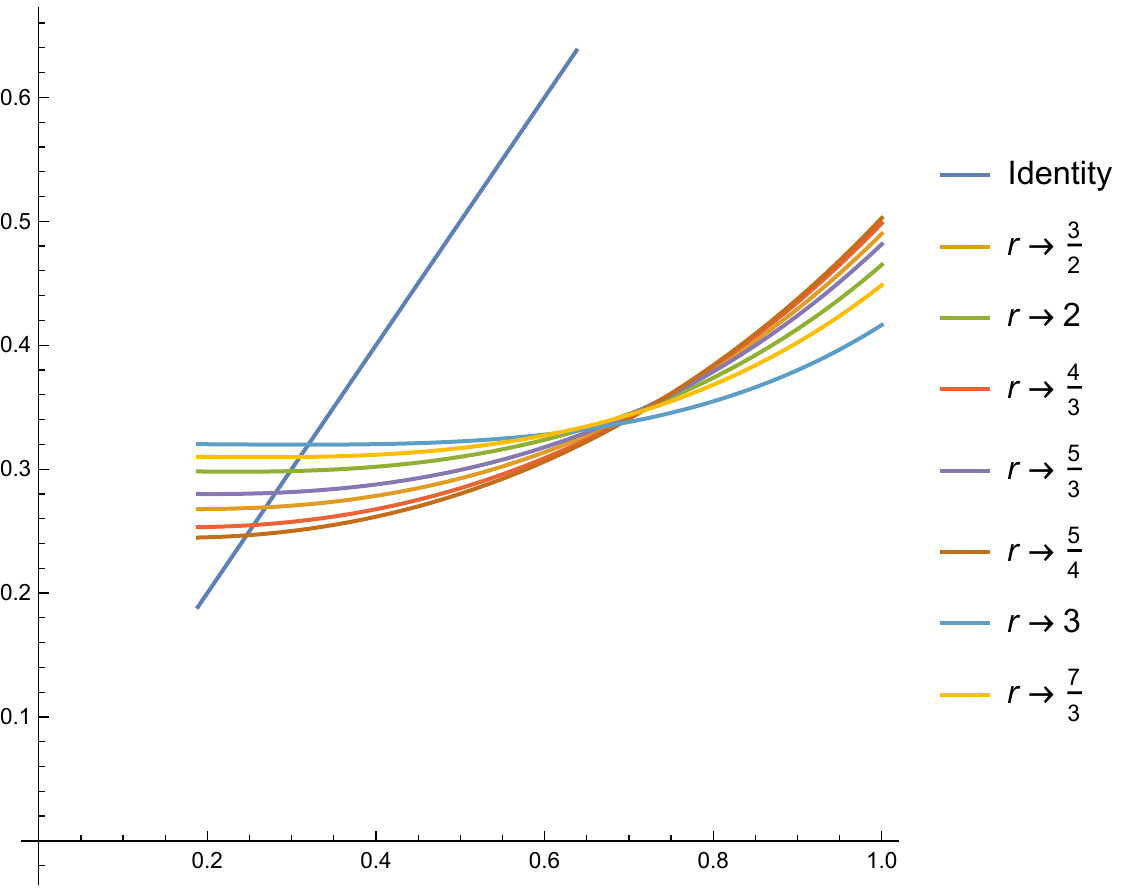}
		\end{overpic}
		\caption{First return map: $m=0.2$, $k=-1$, $d=1.15$, and $r>1$.}\label{firstreturn-ModelRH-3_2---32}
\end{figure}

\newpage

\section{Application: a pendulum with on/off control}\label{sectionapplication}

%Pendulum is one of the simplest types of oscillators, many different models and more details can be found in \cite{AVK} and \cite{minorsky2009nonlinear}. 

Consider the model for a simple pendulum with damping given by 
\begin{equation}\label{pend_damp}
\left(\begin{array}{c}
\dot{x} \\ \dot{y}
\end{array}
\right)=\left(\begin{array}{c}
y \\ a_1y-\sin(x)
\end{array}\right)=X_{a_1}(x,y),
\end{equation}
where $x$ is the angle with the vertical axis, $y=\dot{x}$ is the angular speed and $a_1$ is a negative constant.

Let us then apply a control to the pendulum, in the form of an extra driving force, $a_2\left(x+\frac{\pi}{2}\right)\dy$, added to \eqref{pend_damp} if $\dot{x}<a_3-a_4(x+\pi)$. The result is a nonsmooth system with nonsmooth vector field $Z_{a}=(X_{a_1},Y_{a_1,a_2})$, with $a=(a_1,a_2,a_3,a_4)$,
$\Sigma=\{(x,y)\in\R^2;y+a_4(x+\pi)=a_3\}$, and 
\begin{equation*}
Y_{a_1,a_2}(x,y)=\left(\begin{array}{c}
y \\ a_1y-\sin(x)+a_2\left(x+\dfrac{\pi}{2}\right)
\end{array}\right).
\end{equation*}

$S_X=(-\pi,0)$ is a saddle point of $X_{a_1}$, with hyperbolicity ratio $r(a_1)\hspace{-0.1cm}	=\hspace{-0.1cm}-\frac{a_1-\sqrt{a_1^2+4}}{a_1+\sqrt{a_1^2+4}}$. It is a real saddle when $a_3<0$, a boundary saddle when $a_3=0$, and a virtual saddle when $a_3>0$. For $a_3\approx0$ there exists a tangency point, in $\s$, near $S_X$ which we label $P_{Z_a}$ and the first coordinate of this point will be denoted by $p_a$. The point $P_{Z_a}$ coincides with $S_X$ when $a_3=0$ and it is a fold point if $a_3\neq0$. A direct calculation gives that, for $x\approx p_a$, there exists a crossing region when $x>p_a$ and there is a sliding region when $x<p_a$.

This model realizes a degenerate cycle as studied in the previous sections. %Let $x_0$ be a point in $\s^+\cup\s^c$, sufficiently near the saddle point, and consider $\pi_a(x_0)$ as the first coordinate of the point where the trajectory passing through $x_0$ meets $\s$ for the second time. By using numeric calculations we obtain some trajectories, listed in Table \ref{Tab1}.
%\begin{table}[h]
%	\centering
%	\begin{tabular}{c|c|c}
%		$a=(a_1,a_2,a_3,a_4)$ & $x_0$ & $\pi_a(x_0)$\\
%		\hline
%		$(-0.1,-0.77,0,0.1)$ & $(-2.8,-0.1(\pi-2.8))$ & $-4.33775\dots$\\
%		\hline
%		$(-0.1,-0.77,0,0.1)$ & $(-\pi,0.5)$ & $-4.54177\dots$\\ 
%		\hline
%		$(-0.2,-0.77,0,0.1)$ & $(-2.8,-0.1(\pi-2.8))$ & $-2.93979\dots$\\ 
%		\hline
%		$(-0.2,-0.77,0,0.1)$ & $(-\pi,0.5)$ & $-3.02473\dots$
%	\end{tabular}
%	\caption{Data for some values of the parameters of $Z_{a}$.}\label{Tab1}
%\end{table}
Some examples of trajectories %with values given in Table \ref{Tab1} 
are illustrated in Figure \ref{FigPend1}. These show that the unstable manifold of the saddle in $\s^+$ meets $\s$ in the sliding region, after it passes through $\s^c$, and the unstable manifold of $S_X$ in $\s^+$ meets $\s$ in the crossing region twice. It follows from continuity that $Z_a$ presents a degenerate cycle through a saddle-regular point for $a_2=-0.77$, $a_3=0$, $a_4=0.1$ and some $a_1\in(-0.2,-0.1)$. 
\begin{figure}[H]
	\centering
	\begin{overpic}[width=13cm]{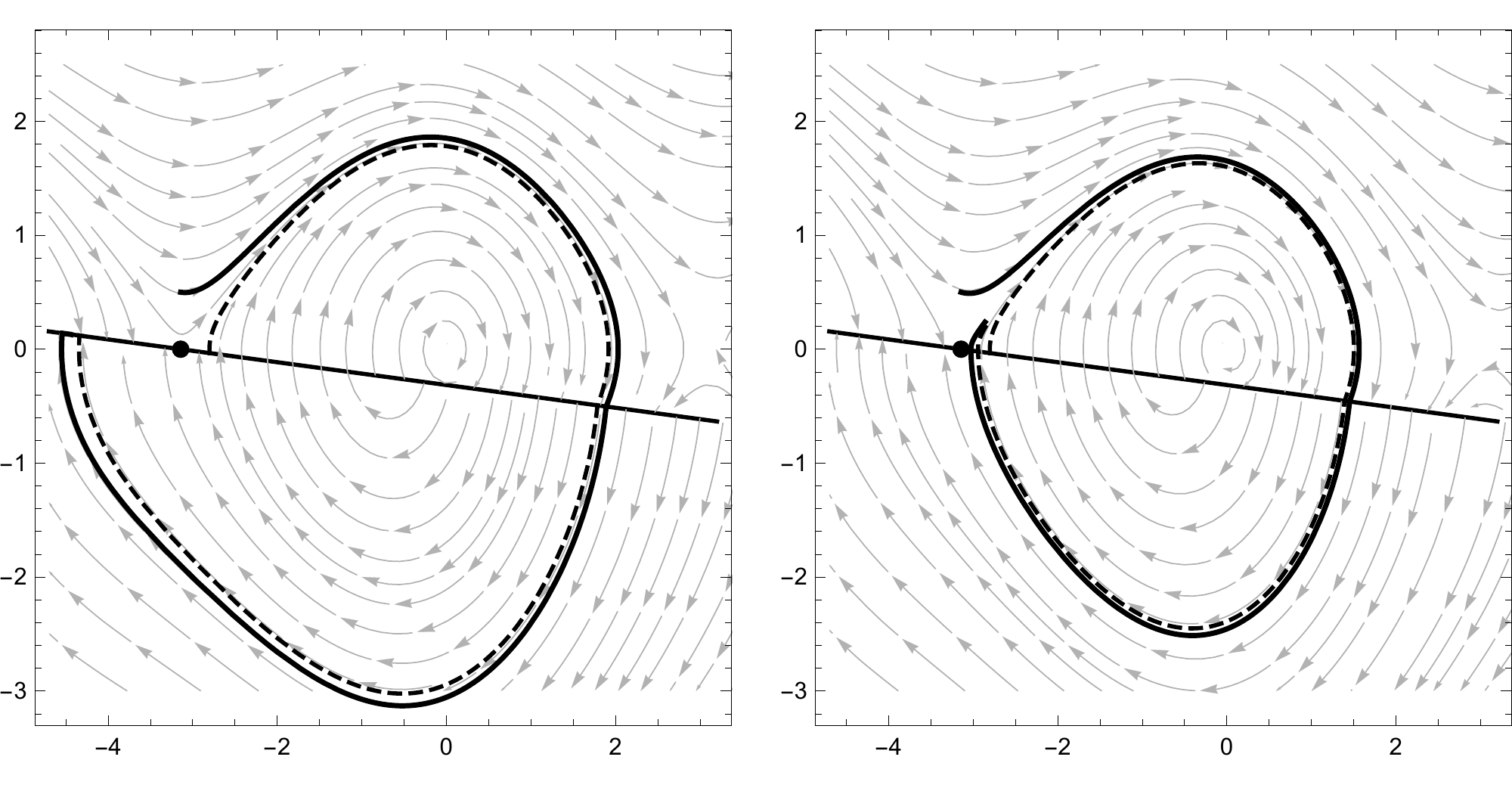}
	\end{overpic}
	\caption{Trajectories of $Z_{a}$. % with data in Table \ref{Tab1}. 
	Left hand side corresponds to $a=(-0.1,-0.77,0,0.1)$ and right hand side corresponds to $a=(-0.2,-0.77,0,0.1)$.}\label{FigPend1} 
\end{figure}

This system represents the case $DSC_{11}$, and realizes each one of the regions in the bifurcation diagram in Figure \ref{Table_SBif_diagram11}. %In Table \ref{Tab1} we have values of the parameter for which the vector field lies on the negative and positive parts of axes $\alpha$ of that bifurcation diagram. Just f
For simplicity, let $q_a$ be the first coordinate of the point $Q_{Z_a}$ (near $P_{Z_a}$) which vanishes the sliding vector field, i.e., $Q_{Z_a}$ is the pseudo-equilibrium point when it is in the sliding region. 
\vspace{1cm}
\begin{itemize}
	\item \textbf{Region $R_1^1$}
\end{itemize}	
	Consider the following values of parameters and initial conditions:  
	$a=(-0.1,-0.77,\newline0.1,0.1)$, $x_{01}=(-\pi,05)\in\s$, and $x_{02}=(-2.8,0.1-0.1(\pi-2.8))\in\s^+$.
	For these values $p_a=-3.14159\dots$, $q_a=-2.14159\dots$, $\pi_a(x_{01})=-4.51446\dots,$ and $\pi_a(x_{02})=-4.37873\dots$. Since $q_a>p_a$, there is no pseudo-equilibrium point. 
	The correspondent trajectories are shown in Figure \ref{R1}.
	By continuity, the trajectory passing through $S_X$ intersects $\s$ in the sliding region. 
	\begin{figure}[H]
		\centering
		\begin{overpic}[width=6.5cm]{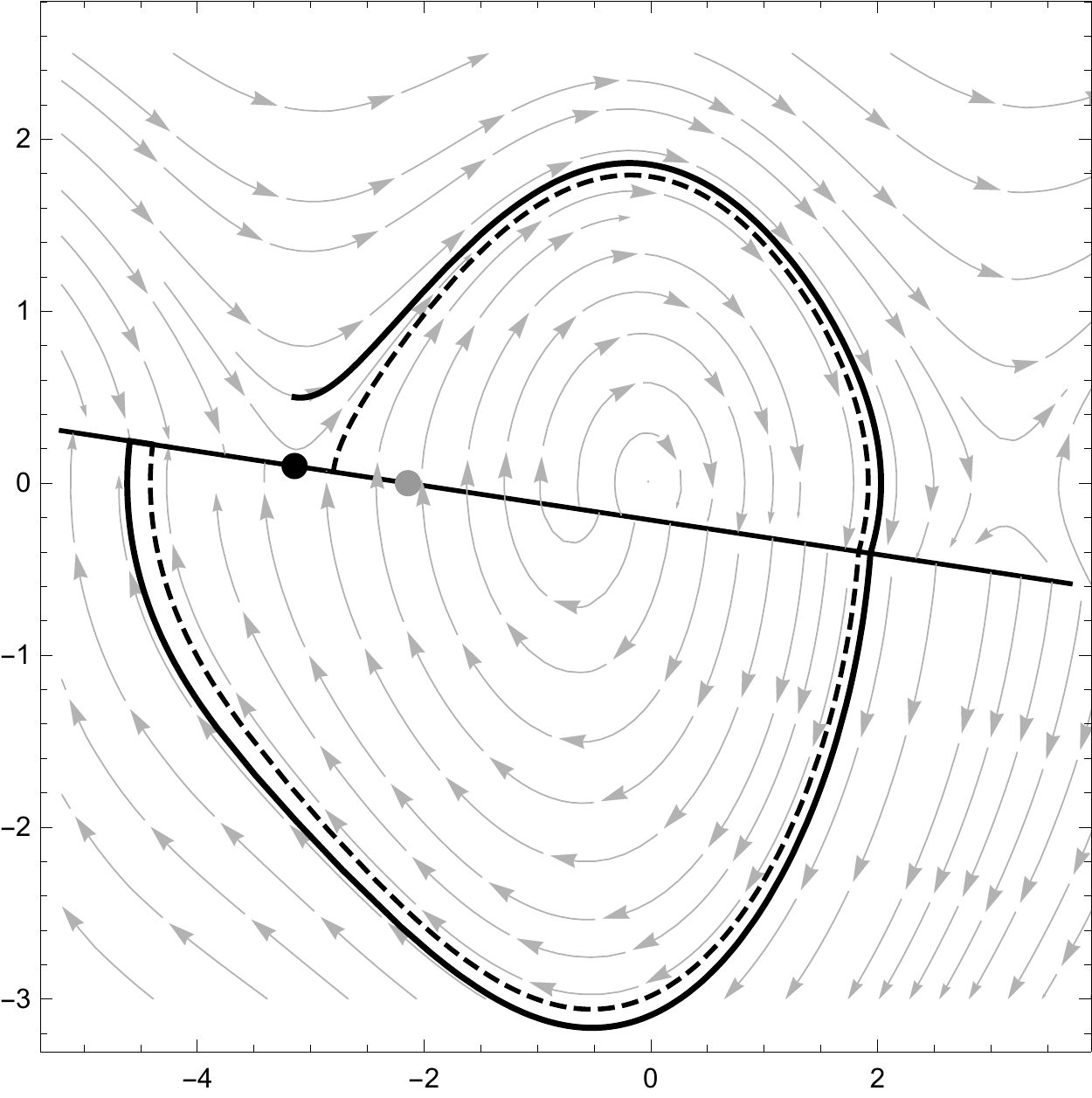}
		\end{overpic}
		\caption{Trajectories of $Z_{a}$ in $R_1^1$: $a=(-0.1,-0.77,0.1,0.1)$. The solid trajectory corresponds to to $x_{01}$ and the dashed trajectory corresponds to to $x_{02}$. $P_{Z_a}$ is the black point and $Q_{Z_a}$ is the gray point.}\label{R1} 
	\end{figure}
	
	\begin{itemize}
	\item \textbf{Region $R_2^1$}
	\end{itemize}
	
	Considering the following values of parameters and initial conditions: 
	$ a=(-0.2,\newline-0.77,0.1,0.1)$, $x_{01}=(-\pi,0.5))\in\s^+$ and  $x_{02}=(-2.5,0.1-0.1(\pi-2.5))\in\s$, we obtain
	$p_a=-3.13169\dots$, $q_a=-2.14159\dots$, $\pi_a(x_{01})= -3.06627\dots$ and $\pi_a(x_{02})=-2.90533\dots$. Since $q_a>p_a$ there is no pseudo-equilibrium point. These trajectories are shown in Figure \ref{R2}. 
	By continuity, the trajectory trough $P_{Z_a}$ (which is a fold point) intersects $\s$ twice in the crossing region twice. These trajectories do not cross the sliding region near $P_{Z_a}$. We have $\pi_a((-3.1,0.1-0.1(\pi-3.1)))=-3.00766\dots>-3.1$ and $\pi_a((-2.9,0.1-0.1(\pi-2.9)))=-2.9955\dots<-2.9$, since the first return map is continuous, must exists an attracting limit cycle through a point $(x_c,0.1-0.1(\pi+x_c))$ for some $x_c\in(-3.1,-2.9)$. See the graph of the first return map in Figure \ref{FirstReturnMapR2}.
	\begin{figure}[H]
		\begin{minipage}[t]{0.450\linewidth}
			\begin{overpic}[width=6.3cm]{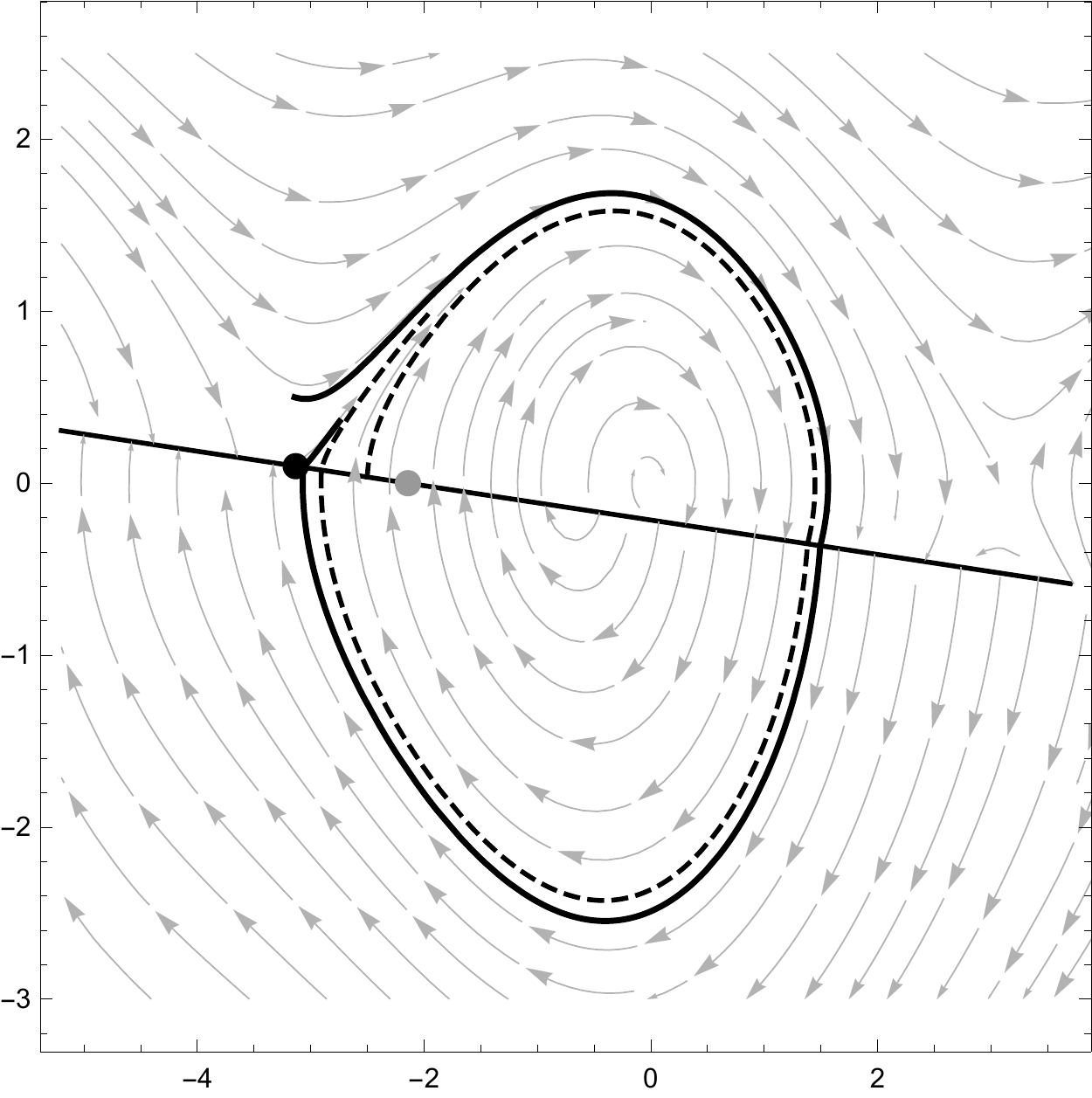}
			\end{overpic} 
			\caption{Trajectories of $Z_{a}$ in $R_2^1$: $ a=(-0.2,-0.77,0.1,0.1)$. The solid trajectory corresponds to to $x_{01}$ and the dashed trajectory corresponds to to $x_{02}$. $P_{Z_a}$ is the black point and $Q_{Z_a}$ is the gray point.}\label{R2}  
		\end{minipage} \hspace{1cm}
		\begin{minipage}[t]{0.450\linewidth}
			\begin{overpic}[width=6.3cm]{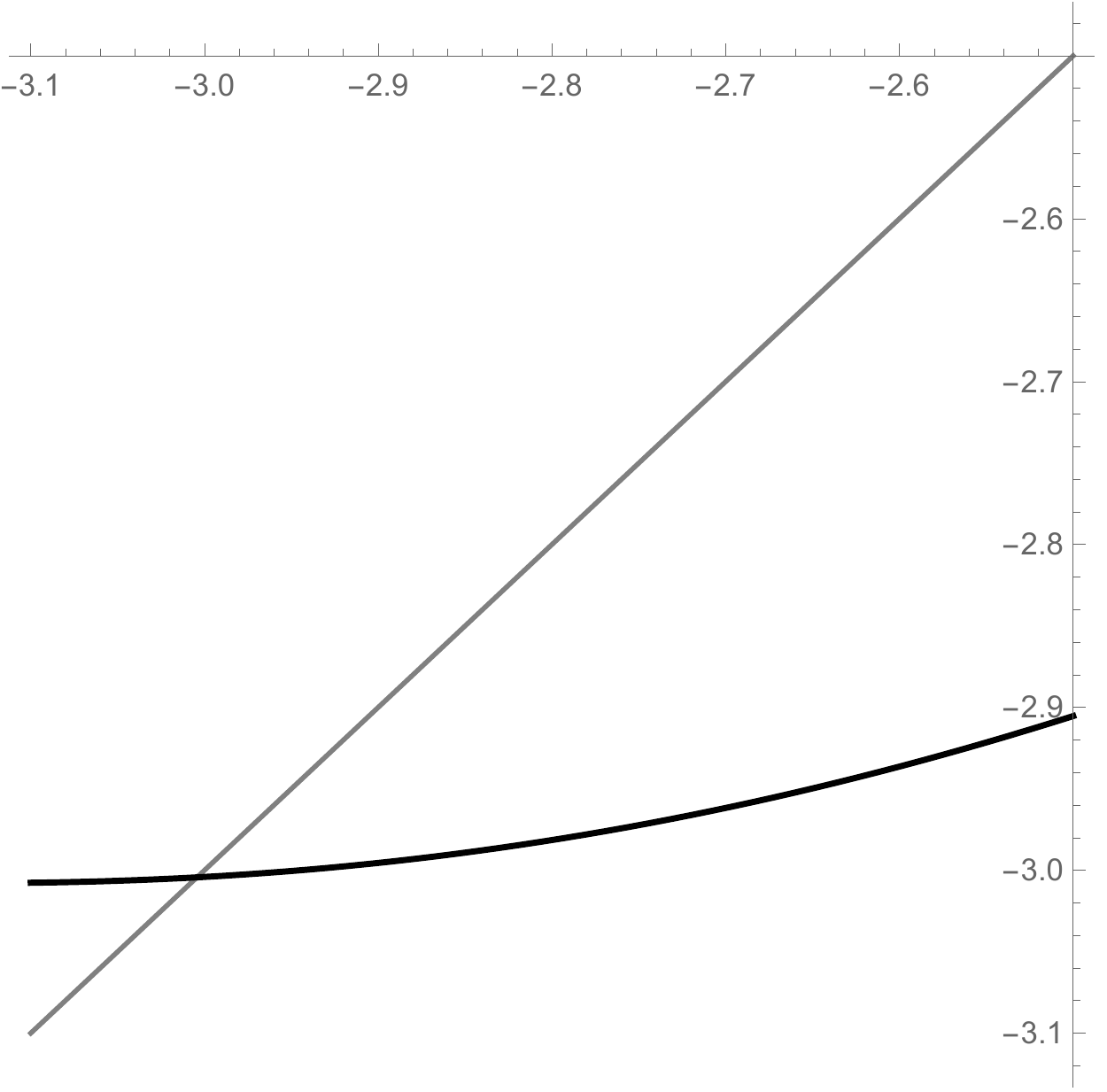}
			\end{overpic}
			\caption{First return map in $R_2$: $ a=(-0.2,-0.77,0.1,0.1)$. The origin is located in $(-2.5,-2.5)$.}\label{FirstReturnMapR2}	
		\end{minipage}
	\end{figure}
	
	\begin{itemize}
		\item \textbf{Curve $\alpha^+=\{(\alpha,0);\alpha>0\}$}
	\end{itemize}

	When $a_3=0$ the saddle point is on the boundary. For the values of parameters and initial conditions $a=(-0.2,-0.77,0,0.1)$, $x_{01}=(-\pi,0.5)$, and $x_{02}=(-2.8,-0.1(\pi-2.8))$ we obtain $p_a=q_a=-3.14159\dots$, $\pi_a(x_{01})= -3.02473\dots$, and $\pi_a(x_{02})=-2.93979 \dots$. These trajectories are illustrated in Figure \ref{PosAxes}.
	The unstable manifold of $S_X$ in $\s^+$ intersects $\s^c$, at the second time, in a neighborhood of $S_X$. We have $\pi_a((-3.1,-0.1(\pi-3.1)))=-2.96489\dots>-3.1$ and $\pi_a((-2.9,-0.1(\pi-2.9)))=-2.95331\dots<-2.9$. Since the first return map is  continuous, it implies in the existence of an attracting limit cycle through $(x_c,-0.1(\pi+x_c))$ for some $x_c\in(-3.1,-2.9)$. See the graph of the first return map in figure \ref{FirstReturnMapPositiveAxes}.

	\begin{itemize}
		\item \textbf{Region $R_3^1$}
	\end{itemize}

	Consider the values of parameters and initial conditions:\\$a\hspace{-0.1cm}=\hspace{-0.1cm}(-0.2,-0.77,-0.1,0.1)$, 
	$x_{01}=(-\pi,0.6)\in\s^+$, and $x_{02}=(-2.9,-0.1-0.1(\pi-2.9))\in\s^+$. We obtain
	$p_a=-3.15149\dots$, $q_a=-4.14159\dots$, $\pi_a(x_{01})= -2.99339\dots$, and $\pi_a(x_{02})= -2.89616\dots<-2.9$. These trajectories are shown in Figure \ref{R3}. 
	Hence, the unstable manifold in $\s^+$ that intersects $\s$ at the crossing region must intersect $\s^c$ again near $P_{Z_a}$. We have $\pi_a((-3.1,-0.1(\pi-3.1)))=-3.31943\dots>-3.1$ then there exists an attracting limit cycle through $(x_c,-0.1-0.1(\pi+x_c))$ for some $x_c\in(-3.1,-2.9)$. See the graph of the first return map in Figure \ref{FirstReturnMapR3}.

		\begin{figure}[H]
			\begin{minipage}[t]{0.450\linewidth}
				\begin{overpic}[width=6.3cm]{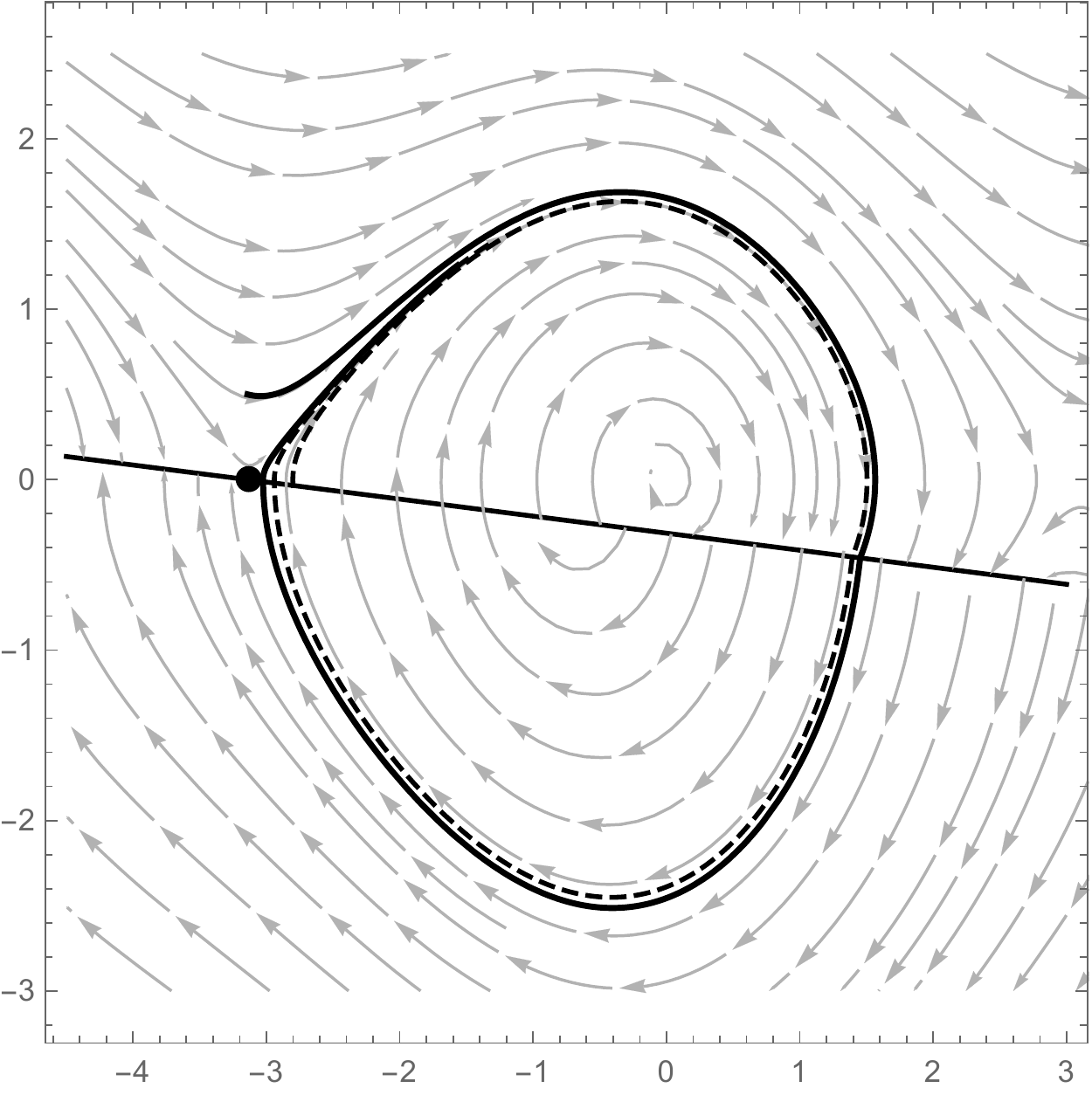}
				\end{overpic}
				\caption{Trajectories of $Z_{a}$ in $\alpha^+$: $a=(-0.2,-0.77,0,0.1)$. The solid trajectory corresponds to to $x_{01}$ and the dashed trajectory corresponds to to $x_{02}$. $P_{Z_a}$ is the black point and $Q_{Z_a}$ is the gray point.}\label{PosAxes}  
			\end{minipage} \hspace{1cm}
			\begin{minipage}[t]{0.450\linewidth}
				\begin{overpic}[width=6.3cm]{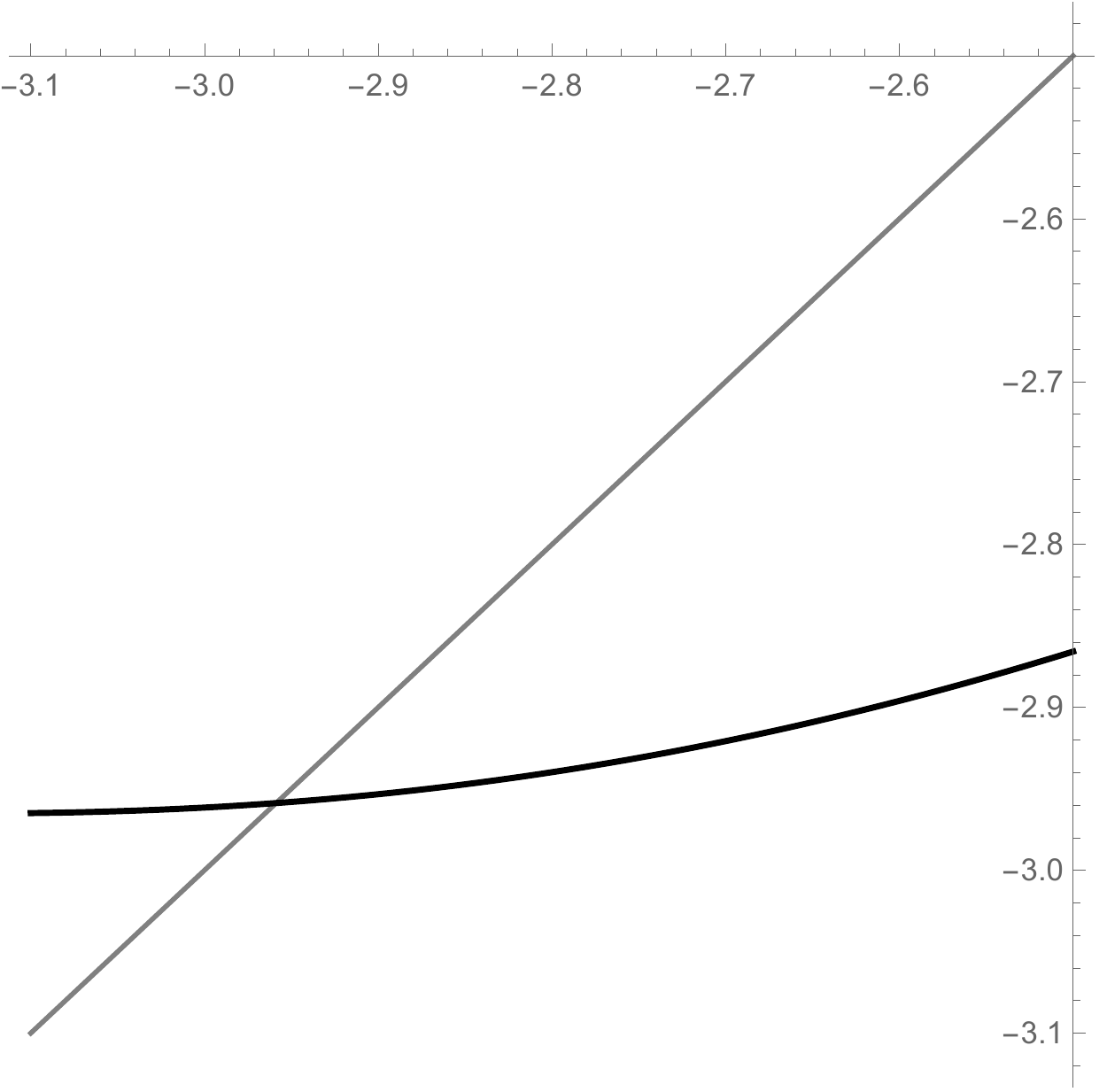}
				\end{overpic}
				\caption{First return map in $\alpha^+$: $a=(-0.2,-0.77,0,0.1)$. The origin of this axes is located at $(-2.5,-2.5)$.}\label{FirstReturnMapPositiveAxes} 
			\end{minipage}
		\end{figure}

	\begin{figure}[H]
		\begin{minipage}[t]{0.450\linewidth}
			\begin{overpic}[width=6.3cm]{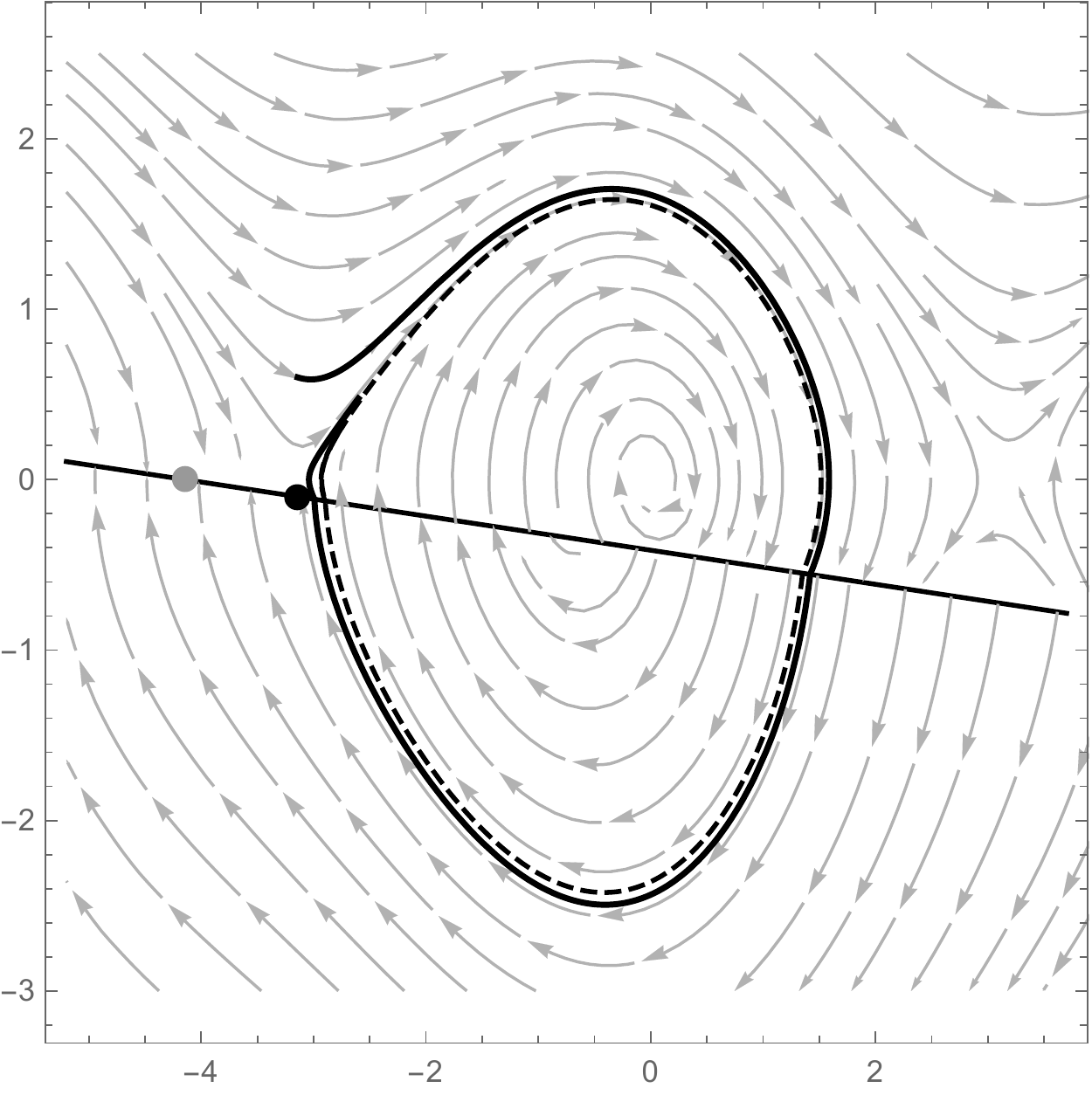}
			\end{overpic}
			\caption{Trajectories of $Z_{a}$ in $R_3^1$: $a=(-0.2,-0.77,-0.1,0.1)$. The solid trajectory corresponds to to $x_{01}$ and the dashed trajectory corresponds to to $x_{02}$. $P_{Z_a}$ is the black point and $Q_{Z_a}$ is the gray point.}\label{R3} 
		\end{minipage} \hspace{1cm}
		\begin{minipage}[t]{0.45\linewidth}
			\begin{overpic}[width=6.3cm]{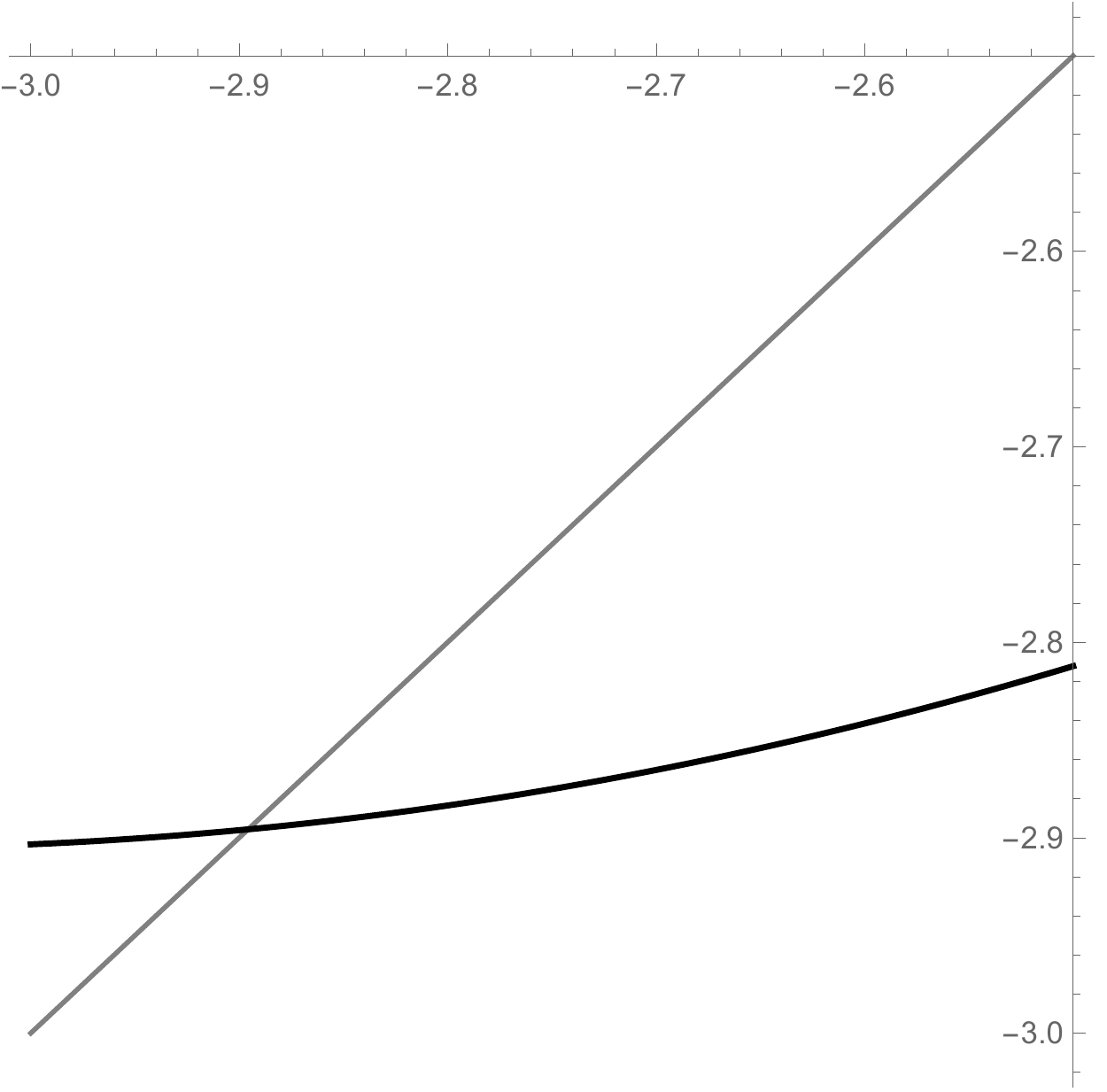}
			\end{overpic}
			\caption{First return map in $R_3^1$: $a=(-0.2,-0.77,-0.1,0.1)$. The origin of this axes is located in $(-2.5,-2.5)$.}\label{FirstReturnMapR3} 
		\end{minipage}
	\end{figure}
	
\newpage	\begin{itemize}
		\item \textbf{Region $R_4^1$} 
	\end{itemize}

	Consider the values of parameters and initial conditions
	$a=(-0.185,-0.77,-0.2,0.1)$,
	$x_{01}=(-\pi,0.5)\in\s^+$, and $x_{02}=(-2.8,-0.2-0.1(\pi-2.8))\in\s$
	we obtain
	$p_a=-3.15845\dots$, $q_a=-5.14159\dots$, $\pi_a(x_{01})= -3.33481\dots$, and $\pi_a(x_{02})=-2.9545\dots$. These trajectories are shown in Figure \ref{R4}.
	The unstable manifold in $\s^+$, which intersects $\s^c$, reach the sliding region near $P_{Z_a}$ after crossing through $\s$ at twice. 
	
	\begin{itemize}
		\item \textbf{Regions $R_5^1\cup\gamma_{x_1}\cup R_6^1$}
	\end{itemize}

	Consider the values of parameters and initial conditions
	$a=(-0.15,-0.77,-0.1,0.1)$, $x_{01}=(-\pi,0.5)\in\s^+$, and $x_{02}=(-2.7,-0.1-0.1(\pi-2.7))\in\s$
	we obtain the values
	$p_a=-3.14657\dots$, $q_a=-4.14159\dots$, $\pi_a(x_{01})= -3.57493\dots$ and $\pi_a(x_{02})=-3.41217\dots$. These trajectories are shown in figure \ref{R5}.
	The unstable manifold in $\s^+$, which crosses $\s^c$, intersects $\s^s$ at a point between the pseudo-equilibrium and the fold point. 
	\begin{figure}[H]
		\begin{minipage}[t]{0.450\linewidth}
			\begin{overpic}[width=6cm]{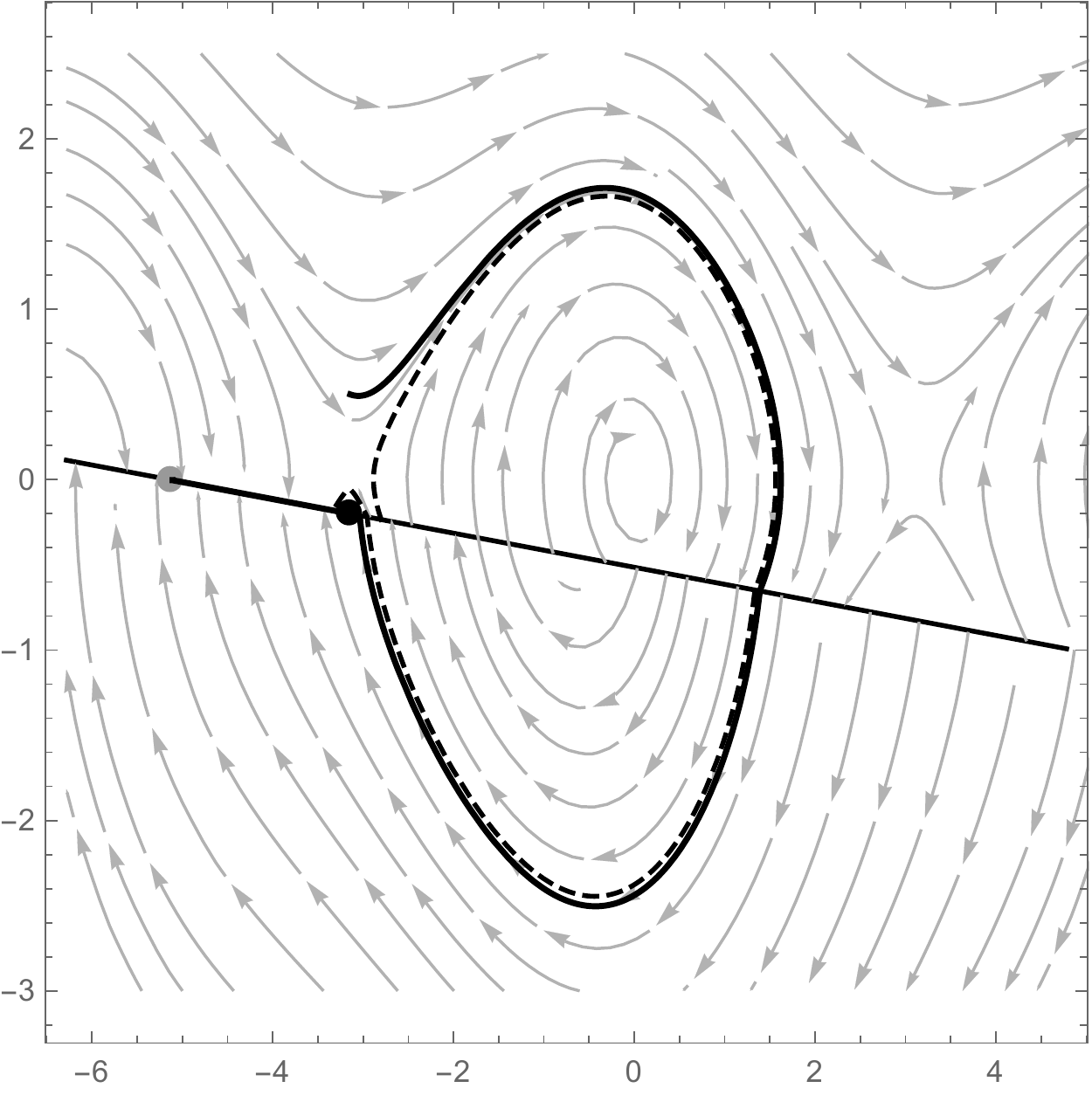}
			\end{overpic}
			\caption{Trajectories of $Z_{a}$ in $R_4^1$: $a=(-0.185,-0.77,-0.2,0.1)$. The solid trajectory corresponds to to $x_{01}$ and the dashed trajectory corresponds to to $x_{02}$. $P_{Z_a}$ is the black point and $Q_{Z_a}$ is the gray point.}\label{R4}  
		\end{minipage} \hspace{1cm}
		\begin{minipage}[t]{0.450\linewidth}
			\begin{overpic}[width=6cm]{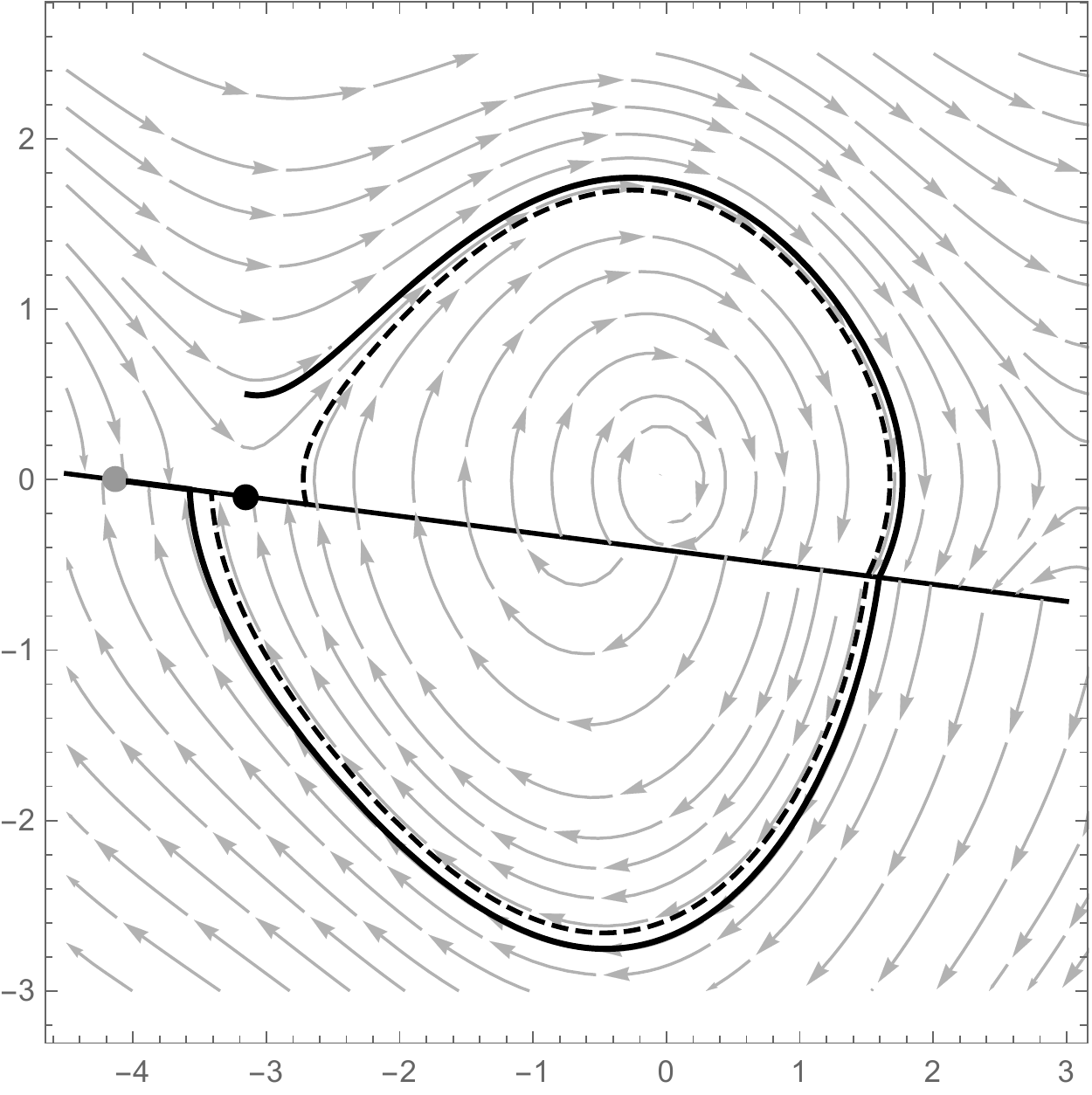}
			\end{overpic}
			\caption{Trajectories of $Z_{a}$ in $R_5^1\cup\gamma_{x_1}\cup R_6^1$: $a=(-0.15,-0.77,-0.1,0.1)$. The solid trajectory corresponds to to $x_{01}$ and the dashed trajectory corresponds to to $x_{02}$. $P_{Z_a}$ is the black point and $Q_{Z_a}$ is the gray point.}\label{R5}  
		\end{minipage}
	\end{figure}
	
	\begin{itemize}
		\item \textbf{Region $R_7^1$}
	\end{itemize}

	Consider the values of parameters and initial conditions
	$a=(-0.1,-0.77,-0.1,0.1)$, $x_{01}=(-2.9,-0.1-0.1(\pi-2.9))\in\s^+$, and $x_{02}=(-2.9,-0.1-0.1(\pi-2.9))\in\s^+$. So, 
	$p_a=-3.14159\dots$, $q_a=-4.14159\dots$, $\pi_a(x_{01})= -4.46432\dots$, and $\pi_a(x_{02})= -4.30114\dots$. These trajectories are shown in figure \ref{R7}. The branch of the unstable manifold in $\s^+$, which crosses $\s$ transversely in $\s^c$, intersects $\s^s$ at a point $P_7$ such that the pseudo-equilibrium is located between $P_7$ and $P_{Z_a}$.
	
\newpage	\begin{itemize}
		\item \textbf{ Curve $\alpha^-=\{(\al,0);\al<0\}$}
	\end{itemize}

	Considering the values of parameters and initial conditions
	$a=(-0.1,-0.77,0,0.1)$, 
	$x_{01}=(-\pi,0.5)\in\s^+$, and $x_{02}=(-2.8,-0.1(\pi-2.8))\in\s$
	we obtain 
	$p_a=q_a=-3.14159\dots$, $\pi_a(x_{01})= -4.54177\dots$ and $\pi_a(x_{02})= -4.33775\dots$. These trajectories are shown in Figure \ref{NegAxes}.
	So, the unstable manifold in $\s^+$ crosses $\s^c$ once before it reaches $\s^s$.

\begin{figure}[H]
	\begin{minipage}[t]{0.45\linewidth}
		\begin{overpic}[width=6cm]{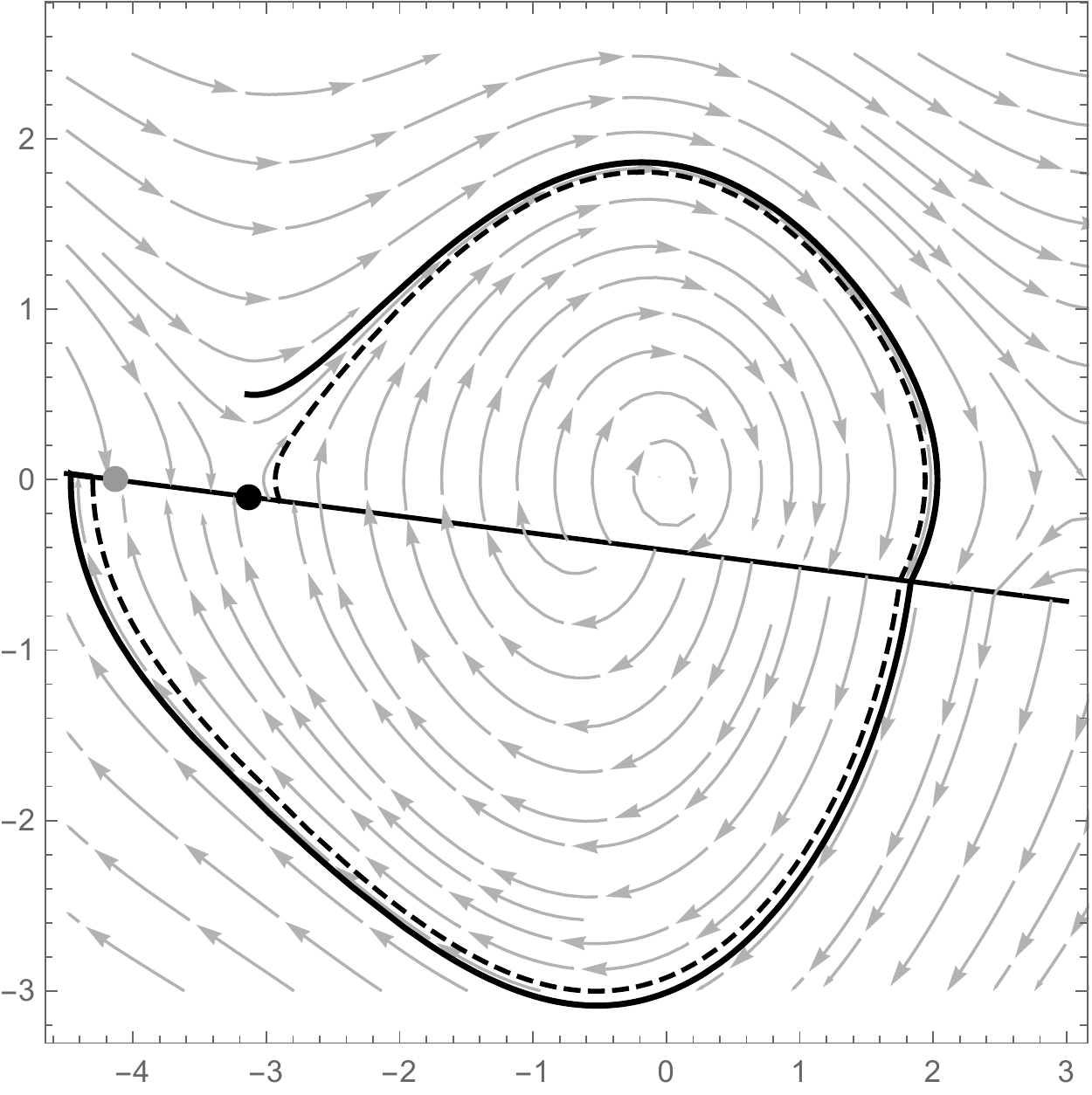}
		\end{overpic}
		\caption{Trajectories of $Z_{a}$ in $R_7^1$: $a=(-0.1,-0.77,-0.1,0.1)$. The solid trajectory corresponds to to $x_{01}$ and the dashed trajectory corresponds to to $x_{02}$. $P_{Z_a}$ is the black point and $Q_{Z_a}$ is the gray point.}\label{R7} 
	\end{minipage} \hspace{0.5cm}
	\begin{minipage}[t]{0.45\linewidth}
		\begin{overpic}[width=6cm]{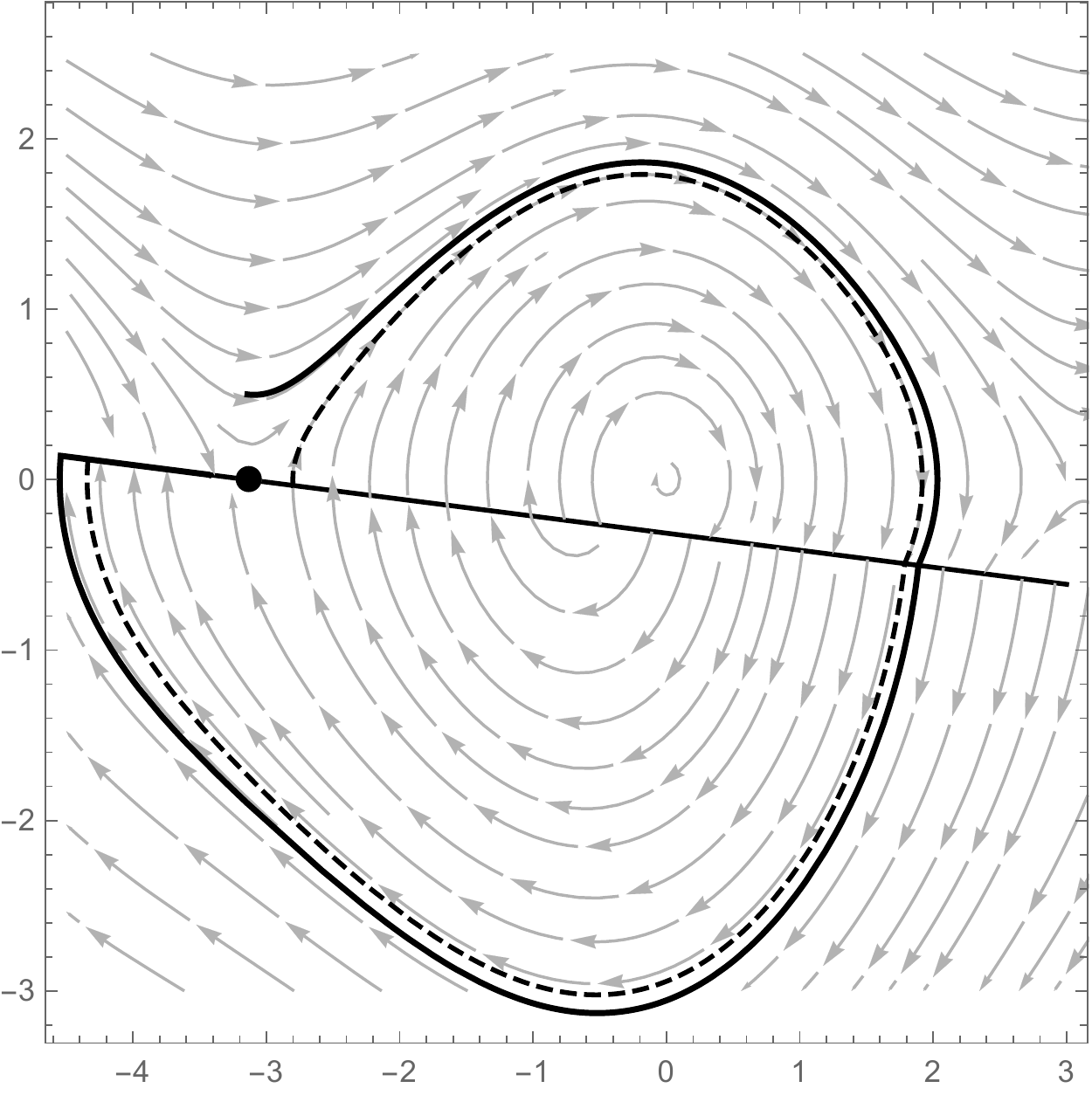}
		\end{overpic}
		\caption{Trajectories of $Z_{a}$ in $\alpha^-$: $a=(-0.1,-0.77,0,0.1)$.The solid trajectory corresponds to to $x_{01}$ and the dashed trajectory corresponds to to $x_{02}$. $P_{Z_a}$ is the black point and $Q_{Z_a}$ is the gray point.}\label{NegAxes} 
	\end{minipage}
\end{figure}

\section{Closing Remarks}

We have unfolded the homoclinic connection to a boundary saddle in nonsmooth dynamical system. The local transition between regular saddle off and pseudo-saddle on the switching surface, and the global bifurcation of periodic orbits from the homoclinic orbit, create a complicated bifurcation structure. We have presented the bifurcation diagrams for non-resonant saddles only. The application to a forced pendulum demonstrates how readily these bifurcations appear in practical control scenarios. 

\section{Acknowledgments}

This research has been partially supported by 
EU Marie-Curie IRSES "Brazilian-European partnership in Dynamical 
Systems" (FP7-PEOPLE-2012-IRSES 318999 BREUDS), FAPESP Thematic Project (2012/18780-0), FAPESP Regular Project 2015/06903-8 and FAPESP PhD Scholarship (Regular: 2013/07523-9 and BEPE: 2014/21259-5).

%\appendix
%
%\section{Section in Appendix}
%\label{appendix-sec1}
%
%Sample text. Sample text. Sample text. Sample text. Sample text. Sample text. 
%Sample text. Sample text. Sample text. Sample text. Sample text. Sample text. 
%Sample text. 

%% References
%%  \cite{key}     ==>> [#]
%%  \cite[chap. 2]{key} ==>> [#, chap. 2]
%%


\bibliographystyle{elsarticle-num}

\nocite{*}

\end{document}